\documentclass[11pt]{article}
\usepackage{booktabs}
\usepackage{longtable}
\usepackage{array}
\usepackage{graphicx, color}
\usepackage[rgb]{xcolor} 
\usepackage[T1]{fontenc}
\usepackage[utf8]{inputenc}
\usepackage{lmodern}
\usepackage{amsmath,amsfonts,amssymb}
\usepackage[unicode=true]{hyperref}
\def\forall{\hbox{for all}~}
\def\L{\mathbf{L}}

\def\ve{\varepsilon}

\def\D{{\cal D}}

\def\T{{\cal T}}

\def\R{{\mathbb R}}

\def\implies{\Longrightarrow}

\def\vs{\vskip 2em}

\def\v{\vskip 1em}
\def\G{{\cal G}}

\def\O{{\cal O}}
\def\C{{\cal C}}
\def\U{{\cal U}}

\def\ov{\overline}
\def\Tilde{\widetilde}
\def\Hat{\widehat}

\def\bega{\begin{array}}
\def\enda{\end{array}}
\def\begi{\begin{itemize}}
\def\endi{\end{itemize}}

\def\ds{\displaystyle}
\def\bel{\begin{equation}\label}
\def\eeq{\end{equation}}
\def\sqr#1#2{\vbox{\hrule height .#2pt
\hbox{\vrule width .#2pt height #1pt \kern #1pt
\vrule width .#2pt}\hrule height .#2pt }}
\def\square{\sqr74}
\def\endproof{\hphantom{MM}\hfill\llap{$\square$}\goodbreak}

    \definecolor{darkred}{RGB}{180,0,0}
    \definecolor{darkblue}{RGB}{0,0,180}
    \definecolor{darkgreen}{RGB}{0,100,0}
    \definecolor{darkorange}{RGB}{200,100,0}
    \definecolor{darkpurple}{RGB}{128,0,128}

\newtheorem{theorem}{Theorem}[section]

\newtheorem{lemma}{Lemma}[section]

\newtheorem{remark}{Remark}[section]
\newtheorem{definition}{Definition}[section]

\newcommand{\abs}[1]{\left| #1 \right|}

\begin{document}

\title{\bf  Local Asymptotic Patterns for  Viscous Approximations of Conservation Laws}
\vs

\author{Alberto Bressan$^{(*)}$, Laura Caravenna$^{(**)}$  and
Wen Shen$^{(*)}$\\~~\\
 {\small $^{(*)}$~Department of Mathematics, Penn State University,}
 \\
 {\small $^{(**)}$~Department of Mathematics 
 ``Tullio Levi-Civita", University of Padova,}\\~~\\
{\small E-mails: axb62@psu.edu,~~laura.caravenna@unipd.it, ~wxs27@psu.edu}
}
\maketitle

\begin{abstract} Solutions to hyperbolic conservation laws can be approximated in many different ways:
by vanishing viscosity, relaxations, discrete or semi-discrete numerical schemes, 
approximation with a nonlocal flux, etc$\ldots$ 
For some of these methods, general ${\bf L}^1$ convergence results are available. 
Aim of this paper is to understand the local behavior of these approximations,
in a neighborhood of point where the hyperbolic solution has a  singularity.   Specifically: a point along a shock, or where two shocks interact, or where a new shock is formed.  

Given a sequence of $\epsilon$-approximate solutions, 
a general expectation is that, by a suitable local rescaling of coordinates,
as $\epsilon\to 0$ a well defined limit is obtained. This corresponds to a specific ``eternal solution" (globally defined both in space and in time)
to the approximating equation.
Precise results this direction are here given, in the case of vanishing viscosity.
\end{abstract}

\tableofcontents

\section{Introduction}
\label{sec:1}
\setcounter{equation}{0}
The emergence of singularities in solutions to PDEs has been a topic of extensive study.
For various classes of nonlinear hyperbolic equations, solutions can be uniquely extended in time
beyond singularity formation. In such cases, it is of interest to identify ``structurally stable singularities",
which arise from an open set of initial data (in a suitable topology), and are thus more likely to be
observed in a physical context.  For some particular classes of PDEs, such as scalar conservation laws,
variational wave equations, or Camassa-Holm equations, a classification of generic
singularities can be found in \cite{Gk}, \cite{BC17, BHY} and \cite{LZ}, respectively.
A similar approach is of interest also in connection with multidimensional Hamilton-Jacobi equations.

Typically, this classification corresponds to a limit of local rescalings.  
Namely, if a solution $u(t,x)$ has a singularity at a point $(\tau, \xi)$, 
 one can consider
the rescaled functions
\bel{resc1}u^\delta(t,x)~=~u\bigl(\tau+ \delta t, \xi+\delta x\bigr),\eeq
or, depending on the type of singularity,
\bel{resc2}
u^\delta(t,x)~=~\delta^{-\gamma}  \Big[  u\bigl(\tau + \delta^\alpha t , \,\xi + \delta^\beta x\bigr)- u(\tau, \xi)\Big],\eeq
for suitable exponents $\alpha,\beta,\gamma$.   
In a favorable case, taking the limit as $\delta\to 0$,  one obtains a self-similar function, depending on a finite number of parameters, which describes the asymptotic behavior of the solution in a neighborhood of the singular point $(\tau, \xi)$.

Purpose of the present paper is to point out that, in many cases, this classification of generic singularities
can be carried out on a further level, involving various classes of approximate solutions.

As an example,
consider a  scalar conservation law with smooth, convex flux
\bel{1}
u_t+ f(u)_x~=~0,\qquad\qquad t\in [0,T].\eeq
For a generic  initial data
\bel{2} u(0,x)~=~\bar u(x),\eeq
with $\bar u$ in the intersection of countably many open dense subsets of $\C^3(\R)$, 
it is well known \cite{Gk, S} that the unique entropy weak solution is piecewise smooth, with finitely many shock curves.
Structurally stable singular points can be of three types (see Fig.~\ref{f:z200}).
\begi
\item[(I)] points along a shock curve,
\item[(II)]  points where two shock curves merge,
\item[(III)]  points where a new shock is formed.
\endi

Various approximations of (\ref{1})  have been extensively studied in the literature.
In particular:
\begi
\item[(i)] Vanishing viscosity approximations \cite{Dafermos, HR, Kr}, such as
\bel{vclaw}
u_t + f(u)_x~=~\ve\bigl(b(u) u_x\bigr)_x\,.\eeq
\item[(ii)] Relaxation approximations \cite{KT, Liu, VH, Y}. For example
\bel{R1}\left\{\bega{cl} 
v_t + v_x &=\, \ve^{-1} \bigl(w-g(v)\bigr), \\[2mm] w_t&= -\ve^{-1}\bigl(w-g(v)\bigr),\enda
\right. 
\eeq
where $g'(v)>0$ and the function $v=f(u)$ is the inverse of $u=v+g(v)$.
\item[(iii)] Approximations with nonlocal flux. More specifically, motivated by models of traffic flow
\cite{BS2, CCS, KP}, one can consider
\bel{nloc}u_t + F(u,v)_x~=~0,\qquad\qquad v(t,x)\doteq\, \int_0^{+\infty} \ve^{-1}e^{-y/\ve} \,u(t, x+y)\,dy,\eeq
where $F(u,u)=f(u)$.\item[(iv)] Various numerical approximations \cite{HR, LV}.
\endi
Let now $u=u(t,x)$ be a generic solution of (\ref{1}) with  $\C^3$ initial data.
In addition, let $u_{\ve}$ be a family of approximations of  one of the types (i)--(iv),
such that the initial data $u_{\ve}(0,\cdot)$ converge to $u(0,\cdot)$ in $\C^3$, as $\ve\to 0$.

If $(\tau, \xi)$ is a point where $u$ has a structurally stable singularity, as already remarked
the rescaled functions in (\ref{resc1}) or (\ref{resc2}) converge to a well defined limit as $\delta\to 0$.
In this same setting we conjecture that, by suitably choosing 
 points $(\tau_\ve, \xi_\ve)\to (\tau, \xi)$, for suitable exponents $\alpha,\beta,\gamma$ 
(not necessarily the same as in (\ref{resc2})), the rescaled
functions
\bel{Uep1}U^{\ve}(t,x)~\doteq~ u_\ve\bigl(\tau+ \ve^\alpha t,~ \xi+\ve^\beta x\bigr),\eeq
or, depending on the type of singularity,
\bel{Uep2}
U^{\ve}(t,x)~\doteq~ \ve^{-\gamma} \Big[  u_{\ve}\bigl(\tau_\ve+ \ve^\alpha t,\,\xi_\ve+ \ve^\beta x\bigr)- u(\tau, \xi)\Big],\eeq
converge to a unique limit $U=U(t,x)$, uniformly on compact sets.   This limit is defined for all $(t,x)\in \R^2$. It provides an ``eternal
solution" (i.e, with time ranging from $-\infty$ to $+\infty$)  to an equation of the same type as 
(\ref{vclaw}), (\ref{R1}) or (\ref{nloc}), but with $\ve=1$.  The set of all these limits  depends
on a finite number of parameters.

In the remainder of this paper we shall give a proof of the above conjecture in the basic case of 
viscous approximations to a scalar conservation law with strictly convex flux:
\bel{viscep}
u_t + f(u)_x~=~\ve\, u_{xx}\,.\eeq
The following assumptions will be used.
\begi
\item[{\bf (A1)}] 
{\it The flux  function $f$ is smooth and uniformly convex:
\bel{fuc}
0~<~c_1~\leq~f''(u)~\leq~c_2\qquad\qquad\forall u\in\R.\eeq
}
\item[{\bf (A2)}] 
{\it The initial data satisfies $\bar u\in \L^1(\R)$.} 
\endi

We recall that, for a generic initial data $\bar u\in  \U\subset \C^3(\R)$,  the solution of (\ref{1})-(\ref{2})
is piecewise smooth with a locally  finite number of shocks \cite{GS, Gk,S}. 
Moreover,
shocks interact at most two at a time.
Here $ \U$ is a dense $\G_\delta$-set, i.e., the intersection of countably 
many open dense subsets of $\C^3(\R)$.

\begin{figure}[ht]
\centerline{\hbox{\includegraphics[width=5cm]{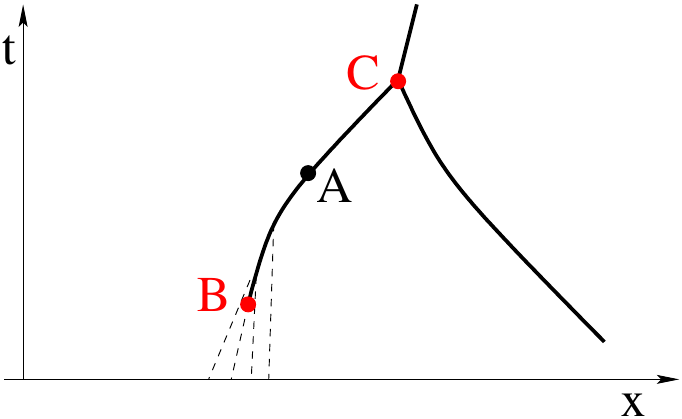}}}
\caption{\small  The three types of generic singularities for a solution to a conservation law.
A: a point along a shock.   B: a point where a new shock forms.  C: a point where two shocks merge. Here the dashed lines denote characteristics, while the solid curves denote shocks. }
\label{f:z200}
\end{figure}

For such generic solutions, singularities can only be of the above types (I)--(III).
For each type, there is a natural asymptotic rescaling that 
characterizes  the singularity up to leading order.
\v
{\bf CASE 1.}
If $(\tau,\xi)$ is a  point  along a shock, we consider the rescalings
\bel{rescd}
u^\delta(t,x)~\doteq~u\bigl(\tau+ \delta t,\, \xi+ \delta x\bigr).\eeq
As  $\delta\to 0+$, the solutions $u^\delta$ converge in $\L^1_{loc}(\R^2)$ to a
piecewise constant solution $U=U(t,x)$ of (\ref{1})
defined for all $(t,x)\in \R^2$ and containing a single shock:
\bel{sshock}U(t,x)~=~\left\{ \bega{rl} u^-\quad &\hbox{if}~~x<\lambda t,\\[1mm]
u^+\quad &\hbox{if}~~x>\lambda t.\enda\right.\eeq
Here
$$u^\pm\doteq u(\tau, \xi\pm),\qquad \qquad\lambda = {f(u^+)- f(u^-)\over u^+-u^-}\,.$$
\v
{\bf CASE 2.}
If $(\tau,\xi)$ is a  point  where two shocks interact,  the rescaled functions  $u^\delta$  defined at (\ref{rescd})
 now converge to a piecewise constant solution with two interacting shocks
\bel{2shock}U(t,x)~=~\left\{ \bega{cl} u^-\quad &\hbox{if}~~x<\lambda_1 t,\\[1mm]
u^*\quad &\hbox{if}~~\lambda_1t<x< \lambda_2 t,
\\[1mm]
u^+\quad &\hbox{if}~~\lambda_2t<x,
\enda\right.\qquad 
\qquad\hbox{for} ~t< 0,\eeq

\bel{3shock}U(t,x)~=~\left\{ \bega{rl} u^-\quad &\hbox{if}~~x<\lambda t,\\[1mm]
u^+\quad &\hbox{if}~~x>\lambda t,\enda\right.\qquad 
\qquad\hbox{for} ~t\geq 0.\eeq

Here $u^-, u^*, u^+$ are the limits of left, middle and right states before the interaction, while
$\lambda_1>\lambda_2$ and $\lambda$ are the Rankine-Hugoniot speeds of the shocks, 
before and after the interaction.
\v
{\bf CASE 3.} At the point $\ov P=(\tau,\xi)$ a new shock is formed. 
Assuming that the initial data $\bar u$ is $\C^3$, call 
$$x_0\,\doteq\,\xi- f'\bigl(u(\tau,\xi)\bigr)\, \tau,$$
 the initial point of the characteristic
through $\ov P$, so that
$$\bar u(x_0)~=~u(\tau,\xi),\qquad\qquad  \tau\,=\, {1\over f''\bigl( \bar u(x_0)\bigr)\bar u_x(x_0)}\,. $$
Following \cite{Gk}, we say that this singularity is {\bf structurally stable} if 
$${\partial\over\partial x}  f'\bigl(\bar u(x)\bigr)\bigg|_{x=x_0} =~0,
\qquad\qquad {\partial^2\over\partial x^2}  f'\bigl(\bar u(x)\bigr)\bigg|_{x=x_0} <~0.$$
In this case, the rescaled functions
\bel{ud3}u^\delta(t,x)~=~\delta^{-1} \Big[ u\bigl(\tau+ \delta^2 t, \,\xi+ \delta^3 x\bigr) - u(\tau,\xi) \Big],
\eeq
as $\delta\to 0$  converge to a global solution of Burgers' equation 
\bel{Bur} u_t+ \left({a u^2\over 2} + bu\right)_x~=~0,\eeq
with
\bel{Bid}u(0,x)~=~-c x^{1/3}.\eeq
\v
Here $a= f''(u(\tau,\xi))$, $b= f'(u(\tau,\xi))$, while $c$ can be recovered in terms of the third derivative $x_{uuu}$ 
of the inverse 
function
$u\mapsto x(\tau,u)$, at the point $u=u(\tau,\xi)$.
The proof is elementary, since for $t<\tau$ in a neighborhood of $(\tau,\xi)$ the solution of (\ref{1})
can be explicitly constructed by characteristics.
\begin{remark}\label{r:11}{\rm
All solutions to (\ref{Bur})-(\ref{Bid}) can be obtained from the single solution where
$a=c=1$ and $b=0$, after an affine change of variables. Indeed, setting
$$u(t,x)~=~\gamma\, w(\alpha t, \, x-\beta t),\qquad\qquad \alpha = ac, \quad\beta = b,\quad \gamma = c,$$
a direct computation shows that $w$ provides a solution to 
$$w_t + \left(w^2\over 2\right)_x\,=\,0,\qquad\qquad w(0,x)~=~-x^{1/3}.$$
}
\end{remark}
Next, we consider smooth functions which are obtained as limits of suitable asymptotic rescalings
of the viscous solutions (\ref{viscep}).
These can be of three types.

{\bf TYPE 1:}  A traveling viscous shock for the viscous equation
\bel{visc1} u_t + f(u)_x~=~u_{xx}\,,\eeq
with asymptotic states $u^-, u^+$ as $x\to\pm \infty$.
As it is well known, this has the form
\bel{UtwS}U(t,x)~=~S(x-\lambda t),\eeq   where $S$ provides a solution to
\bel{SODE}  -\lambda S+ f(S) ~=~S' + C,\eeq
\bel{lc}\lambda~=~{f(u^-)-f(u^+)\over u^--u^+} \,,\qquad\qquad 
C\,=\,f(u^+)-\lambda u^+\,=\, f(u^-)-\lambda u^-.\eeq
Notice that $S$ is determined up to a shift.   To remove this ambiguity, we 
impose the additional condition
\bel{S0}S(0)~=~{u^-+u^+\over 2}\,.\eeq
\v
{\bf TYPE 2:}  Two traveling viscous shocks merging together.
 Let $S_1, S_2$ be two viscous traveling profiles for 
the equation (\ref{visc1}), connecting the states $(u^-, u^*)$ and the states $(u^*, u^+)$, 
respectively (see Fig.~\ref{f:asy4}). These profiles are uniquely determined up to a shift.  
To remove this ambiguity we assume
$$S_1(0)\,=\, {u^-+u^*\over 2}\,,\qquad\qquad S_2(0)\,=\, {u^*+u^+\over 2}\,.$$
The corresponding Rankine-Hugoniot speeds are
$$\lambda_1~=~{f(u^-) - f(u^*)\over u^--u^*}~>~
 {f(u^*) - f(u^+)\over u^*-u^+}~=~\lambda_2\,.$$
 As it will be shown in Section~\ref{sec:8},
there exists a unique solution  $W=W(t,x)$ of  (\ref{visc1}),
defined for all 
$(t,x)\in \R^2$, with the following property. Choosing any intermediate speed 
$\lambda_2<\lambda<\lambda_1$,  as $\tau\to -\infty$ the profile $W(\tau,\cdot)$ satisfies
\bel{W12} \lim_{\tau\to -\infty} \left\{ \int_{-\infty}^{\lambda \tau} \Big|
W(\tau,x) - S_1(x - \lambda_1\tau)\Big| \, dx + \int^{+\infty}_{\lambda \tau} \Big|
W(\tau,x) - S_2(x - \lambda_2\tau)\Big| \, dx 
\right\}~=~0.\eeq
\v
{\bf TYPE 3:}  To describe the appropriate rescaled limit of viscous approximations, in a neighborhood where a new shock is formed, we proceed as follows.   
Let $z=z(t,x)$ be the solution to the inviscid Burgers' equation
\bel{Burv} z_t + zz_x~=~0\,,\qquad\qquad z(0,x)=-x^{1/3}.\eeq
This is well defined for all $(t,x)\in \R^2$, and smooth for $t<0$. Then we
define $Z^{(n)}:[-n, +\infty[\,\times \R\mapsto \R$ to be the solution of 
the Cauchy problem for the viscous Burgers' equation
\bel{CPn}Z_t + Z Z_x~=~Z_{xx}, \qquad \qquad Z(-n,x)\,=\,z(-n,x).\eeq
Finally, we take the limit 
\bel{Zlim}
Z(t,x)~=~\lim_{n\to \infty} Z^{(n)}(t,x).\eeq 
Details about the existence, uniqueness, and some properties of these limit solutions will be given
 in Section~\ref{sec:8}. 
 
\begin{figure}[ht]
\centerline{\hbox{\includegraphics[width=16cm]{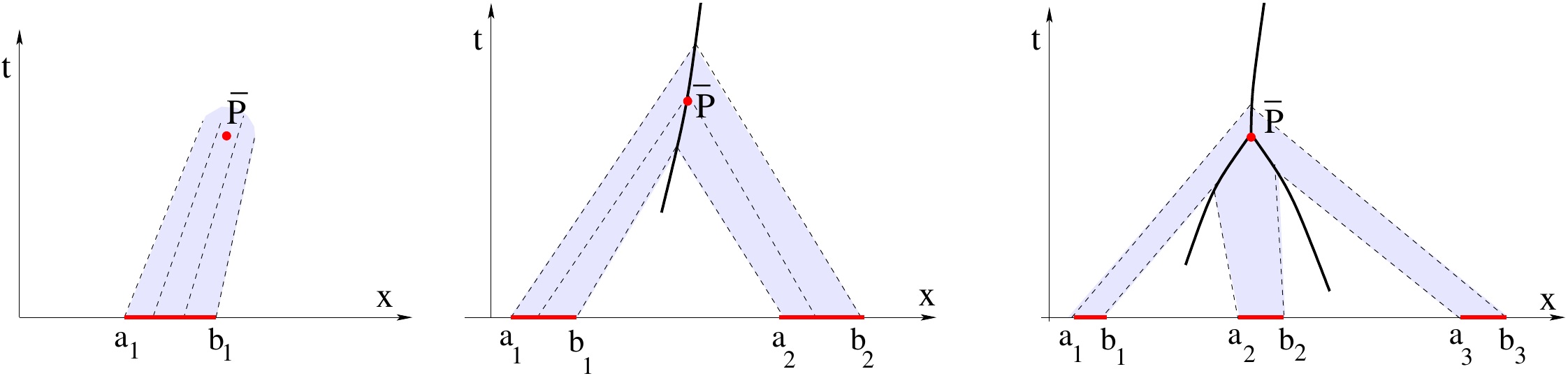}}}
\caption{\small  Three examples of backward domains for a point $\ov P$.  
Left: a point where the solution $u=u(t,x)$
is continuous. Center: a point where $u$ has a shock. Right: a point where two shocks merge.}
\label{f:asy55}
\end{figure}

\begin{definition}\label{backdom} Let $u=u(t,x)$ be a solution to the inviscid equation (\ref{1}).  We say that $\Omega\subset\R^2$
is a {\bf backward domain} for the point $\ov P=(\tau,\xi)$ if $\Omega$ contains all  backward characteristics starting from a neighborhood of $\ov P$.
\end{definition}
Depending on whether the solution $u$ either (i) is 
continuous at $\ov P$, or (ii) has a shock at $\ov P$, or (iii)  has two shocks merging at $\ov P$,
the set of initial points
\bel{Om0}\Omega_0~\doteq~\bigl\{ x\,;~~(0,x)\in \Omega\bigr\}\eeq
may consist of one, two, or three disjoint intervals $[a_i, b_i]$, see Fig.~\ref{f:asy55}.

In the above setting, our first main result is

\begin{theorem}\label{t:1} Let {\bf (A1)-(A2)} hold, and 
let $u=u(t,x)$ be an entropy weak solution of the conservation law (\ref{1})-(\ref{2}) which is piecewise
Lipschitz continuous in a neighborhood of a point $\ov P=(\tau,\xi)$. Let $\Omega$ be a backward domain for 
$\ov P$ and assume $\bar u\in \C^1(\Omega_0)$.

Let $(u_\ve)_{\ve>0}$ be a family of solutions to the viscous equation (\ref{viscep}), whose initial data satisfy
\bel{uelim1}
\bigl\| u_\ve(0,\cdot) - \bar u\bigr\|_{\L^1(\R)}~\to~0,\qquad\quad
\bigl\| u_\ve(0,\cdot) - \bar u\bigr\|_{\C^1(\Omega_0)}~\to~0\qquad\qquad\hbox{as}\quad \ve\to 0.\eeq
Then, as $\ve\to 0$ the following holds.
\begi
\item[(i)] If $(\tau,\xi)$ is a point where $u$ has a single shock, then there exist  points 
$(\tau_\ve,\xi_\ve)\to (\tau,\xi)$ such that the rescaled functions 
 \bel{resc11}U^\ve(t,x)~\doteq~u_\ve\bigl( \tau_\ve + \ve t,~\xi_\ve+ \ve x\bigr)\eeq
 converge to the 
 viscous traveling wave $U(t,x)=S(x-\lambda t)$ in (\ref{UtwS})--(\ref{S0}), uniformly on bounded subsets of $\R^2$.
 
 \item[(ii)] If $(\tau,\xi)$ is a point where $u$ has two interacting shocks, then there exist points 
 $(\tau_\ve,\xi_\ve)\to (\tau,\xi)$ such that the rescaled functions $U^\ve$ in (\ref{resc11}) converge  to the 
globally defined viscous solution $W$ satisfying (\ref{W12}),
uniformly on bounded sets.
\endi
\end{theorem}

Our second result describes the limit of viscous approximations in a neighborhood of a 
point where a new shock is formed.
\begin{theorem}\label{t:2}
 Let {\bf (A1)-(A2)} hold, and 
let $u=u(t,x)$ be an entropy weak solution of the conservation law (\ref{1})-(\ref{2}).
Let $\ov P=(\tau,\xi)$ be a point where a new shock is formed.  Let $\Omega$ be a backward domain for 
$\ov P$ and assume $\bar u\in \C^4(\Omega_0)$.

Let $(u_\ve)_{\ve>0}$ be a family of solutions to the viscous equation (\ref{viscep}), whose initial data satisfy
\bel{uelim2}
\bigl\| u_\ve(0,\cdot) - \bar u\bigr\|_{\L^1(\R)}~\to~0,\qquad\quad
\bigl\| u_\ve(0,\cdot) - \bar u\bigr\|_{\C^4(\Omega_0)}~\to~0\qquad\qquad\hbox{as}\quad \ve\to 0.\eeq
Then  there exist points 
 $(\tau_\ve,\xi_\ve)\to (\tau,\xi)$ and constants $c,\sigma,\lambda$ such that the rescaled functions
 \bel{resc22}U^\ve(t,x)~\doteq~\ve^{-1/4} \Big[u_\ve\bigl( \tau_\ve +\ve^{1/2} t,~
 \xi_\ve +\ve^{3/4}x\bigr)- u(\tau,\xi)\Big] \eeq
converge, uniformly on bounded sets, 
to the function
\bel{PZdef}\Phi(t,x) ~=~c\, Z(\sigma t, x-\lambda t).\eeq
Here $Z$ is 
the global solution to the viscous Burgers' equation  defined at (\ref{Zlim}).
\end{theorem}

\begin{remark}{\rm
Since all  the above functions $U^\ve$  are smooth and have derivatives which are uniformly bounded on compact sets,  the convergence in $\L^1_{loc}$ actually implies
convergence in $\C^k(\Omega)$ for every bounded open set $\Omega\subset\R^2$.
}\end{remark}

\begin{remark} {\rm Let $u=u(t,x)$ be a solution of (\ref{1}) which is $\C^3$ for $t\in [0, t_0]$.
For a fixed $\ve>0$,  consider a sequence of solutions $u_n$ to the viscous equation 
$$u_t+f(u)_x~=~\ve\,u_{xx}$$ 
such that 
\bel{C0lim}\bigl\|u_n(0,x) - \bar u\bigr\|_{\C^0}~\to~ 0.\eeq
Since viscous solutions regularize any discontinuous data, this would imply
\bel{conve}\bigl\|u_n(t,\cdot ) - u(t,\cdot)\bigr\|_{\C^3}\,\to\, 0\qquad \hbox{for } ~~0<t<t_0\,.\eeq
However, if the viscosity coefficient $\ve=\ve_n\to 0$ depends on $n$, the convergence (\ref{conve}) 
may no longer hold.     Indeed, the validity of (\ref{conve}) requires that the convergence (\ref{C0lim})
hold much faster than the rate at which $\ve_n\to 0$.   For this reason, the second assumption in (\ref{uelim2})
cannot be weakened to (\ref{C0lim}).
}
\end{remark}

For the basic theory of conservation laws we refer to \cite{Bbook, Dafermos, HR, Kr, Liubook}. 
In the case of a strictly convex flux,  solutions are entropy admissible iff they satisfy the Lax condition, i.e., 
if they only have downward jumps.

The remainder of this paper is organized as follows. Section~\ref{sec:2} contains a few preliminary computations,
motivating the choice of rescalings.  In Section~\ref{sec:3} we recall a classical transformation of variables,
useful in the analysis of traveling profiles for parabolic equations~\cite{BiB02, BD, F}. Section~\ref{sec:4}  
contains a detailed analysis on the formation of a viscous shock profile, from an initial data which is close to 
two limit values $u^-, u^+$ outside a bounded interval. The earlier results proved in \cite{BD}
are here refined in a crucial way, providing an explicit bound on the length of time needed for the solution
to become close to a viscous shock profile.
In Section~\ref{sec:5} we construct an ``eternal solution" which for $t\to-\infty$ contains two viscous shock profiles,
which merge into a single one as $t\to +\infty$. Section~\ref{sec:6} contains some preliminary lemmas.
These will allow us to conveniently modify the initial data outside a bounded interval $[a,b]$, without changing the 
asymptotic properties of rescaled viscous approximations at the point $(\tau, \xi)$.
After all these preliminaries,
the proof of Theorem~\ref{t:1} is given in Section~\ref{sec:7}.  The last two sections are concerned with the formation of a new shock.  In Section~\ref{sec:8} we derive a few properties of the global viscous solution $Z(t,x)$ 
considered at (\ref{Zlim}).    Finally, Sections~\ref{sec:9} and \ref{sec:10} contain a proof of Theorem~\ref{t:2}.

In the literature on hyperbolic systems of conservation laws,  a general theorem on the $\L^1$ convergence of 
vanishing viscosity approximations was proved in \cite{BiB05}. The local construction of viscous approximation
by matched asymptotic expansions was achieved in 
 \cite{GX} in a neighborhood of an isolated  shock, and more recently in 
recently in \cite{ACG, CG} near a point where a new shock is formed, up to the time of singularity formation.
Our  present point of view is somewhat different, since we assume that some
family of viscous approximations is already given, and try to understand their local behavior near a point of generic singularity. We believe this approach can be fruitful also in connection with different kinds of approximations.

\section{Variable rescaling at shock formation}
\label{sec:2}
\setcounter{equation}{0}
Let $u=u(t,x)$ be a solution to (\ref{1})-(\ref{2}), under the assumptions {\bf (A1)-(A2)}.

As long as the solution remains Lipschitz continuous, it is computed by the method of characteristics.
For every initial point $\xi\in \R$, along the characteristic $t\mapsto x(t,\xi)~=~\xi+ t f'(\bar u(\xi))$
one has
\bel{r3}u(t,x(t,\xi))~=~\bar u(\xi), \qquad u_x(t,x(t,\xi))~=~\left[ {1\over \bar u_x(\xi)} + t\,f''(\bar u(\xi)) \right]^{-1}\eeq
where the second identity is obtained by integrating 
\bel{r4} u_{xt} + f'(u)u_{xx} ~=~- f''(u) u_x^2\,,\eeq
along a characteristic.
If $\bar u_x(\xi)<0$, we consider the blow-up time 
\bel{r5} \T(\xi)~=~ \Big[-f''(\bar u(\xi))\cdot \bar u_x(\xi)\Big]^{-1}.\eeq
A new shock is formed at time $\T(\xi^*)$ along the characteristic starting at $\xi^*$ if the map
\bel{r6}\xi\,\mapsto\,\T(\xi) \eeq has a local minimum at $\xi=\xi^*$.
For generic initial data $\bar u\in \C^3(\R)$ with $\bar u_x(x)\to 0$ as $x\to\pm\infty$, there will be finitely many points $\xi_1<\xi_2<\cdots<\xi_N$ where the 
map (\ref{r6}) attains a local minimum, with $\T(\xi_k)\leq T$.  In this case, the solution to (\ref{1})-(\ref{2}) 
will contain finitely many shocks within the time interval $[0,T]$.
\v
On regions where the solution is smooth and monotonically decreasing, we can consider the inverse function 
$x=x(t,u)$.   This satisfies the evolution equation
\bel{r7} x_t(t,u)~=~f'(u).\eeq
New shocks originate at points where $x_u=0$.   Since
\bel{r8} x_{ut}~=~f''(u),\qquad\qquad x_u(0)~=~{1\over \bar u_x}\,,\eeq
the time where $x_u=0$ along a characteristic is again given  by (\ref{r5}).
\v
\subsection{Variable rescaling for an inviscid solution.}
To simplify our notation, we now assume that the solution $u=u(t,x)$ is defined for $(t,x)\in [-1, T]\times\R$,
and a new shock forms at $(t,x)=(0,0)$.   In view of Remark~\ref{r:11}, w.l.o.g.~we assume that 
\bel{u00} u(0,0)\,=\,0, \qquad f(0)\,=\,f'(0)\,=\,0,\qquad f''(0)=1,\eeq
\bel{111} f'(u)~=~u + \psi_1(u) u^2,\eeq
for some bounded function $\psi_1$.
To describe the process of shock formation, up to leading order, we perform a 
family of variable rescalings.   This amounts to ``zooming in" at the singular point, i.e.,
\bel{rs1}
t\,=\,\ve^\alpha \tau,\qquad x \,=\, \ve^\beta y,\qquad u \,=\, \ve^\gamma v.\eeq
Inserting (\ref{rs1}) in (\ref{1}), using (\ref{111}) one obtains
\bel{r9}\ve^{\gamma-\alpha} v_\tau + f'(\ve^\gamma v) \ve^{\gamma-\beta} v_y~=~\ve^{\gamma-\alpha} v_\tau + \Big(\ve^\gamma v + \psi_1 (\ve^\gamma v) \ve^{2\gamma}v^2\Big) \ve^{\gamma-\beta} v_y~=~0.\eeq
The limit $\ve\to 0$ will thus be a solution to
Burgers' equation
$$v_\tau + vv_y~=~0$$
provided that 
$$\gamma-\alpha~=~2\gamma-\beta,\qquad\qquad \gamma \,=\, \beta-\alpha\,>\,0.$$
In the limit, we wish to capture a solution where
$$x~=~-b u^3 + tu.$$
That is, the above solution should be invariant under rescaling.
This requires
$$\ve^\beta y~=~-b \,\ve^{3\gamma} v^3 + \ve^{\alpha+\gamma} \tau v,$$
hence
$\beta=3\gamma=\alpha+\gamma$.  The three exponents should thus satisfy the proportionality relation
$$\alpha : \beta : \gamma~=~2 : 3 : 1.$$
%

\subsection{Rescaling of viscous solutions.}
For solutions to the viscous equation (\ref{viscep}), the choice
 of the constants $\alpha, \beta,\gamma$ in the rescaling (\ref{rs1})  is dictated by two requirements:
\begi
\item[(i)] If $u(t,x)$ is a solution to Burgers' equation with $\ve$-viscosity, then $v(\tau,y)$ is a solution
of the same equation with unit viscosity:
\bel{v1e}
u_t+uu_x~=~\ve\, u_{xx}\qquad\implies\qquad v_\tau + vv_y~=~v_{yy}\,.\eeq
\item[(ii)] If the map $x\mapsto u(t,x)$ is monotone decreasing, and has a smooth inverse $u\mapsto x(t,u)$, then 
the same holds for $v$ and the map $v\mapsto y(\tau,v)$ satisfies
\bel{xuuu}
x_{uuu}~=~y_{vvv}\,.\eeq
\endi
Inserting (\ref{rs1}) in (\ref{v1e}) one obtains
$$\ve^{\gamma-\alpha} v_\tau + \ve^{2\gamma-\beta}v v_y~=~\ve^{1+\gamma-2\beta} v_{yy}\,.$$
Moreover, inserting (\ref{rs1}) in (\ref{xuuu})  we obtain
$$x_{uuu} ~=~\ve^{\beta-3\gamma} y_{vvv}~=~ y_{vvv}\,.$$
The requirements (i)-(ii) thus imply the three identities
$$\gamma-\alpha~=~2\gamma-\beta ~=~1+\gamma-2\beta,\qquad\qquad \beta-3\gamma\,=\,0.$$
This system has the unique solution
\bel{rv4}\alpha\,=\,{1\over 2},\qquad \beta\,=\, {3\over 4}\,,\qquad \gamma\,=\,{1\over 4}\,.\eeq
This indicates that the appropriate rescaling of viscous solutions, in a neighborhood of a point where a new shock is formed, is
\bel{rv5}u^\ve(t,x)~=~\ve^{-1/4} u(\ve^{1/2} t, \, \ve^{3/4} x).\eeq
%

\section{Viscous shock profiles, in transformed variables}
\label{sec:3}
\setcounter{equation}{0}
Starting with the viscous conservation law
\bel{21}
u_t + f(u)_x~=~u_{xx}\,,\eeq
following \cite{BiB02, BD} we
consider the variable 
\bel{22}w\,=\, f(u) -u_x\,.\eeq
Notice that $w$ corresponds to the full flux function in (\ref{21}). Indeed,
\bel{20}u_t + w_x~=~0.\eeq
Differentiating (\ref{22}) w.r.t.~time and using (\ref{20}), one obtains
\bel{23} w_t~=~f'(u) u_t  -  u_{tx}~=~-f'(u)  w_x + w_{xx}\,.
\eeq
In (\ref{23}) the variable $w=w(t,x)$ is regarded as a function of $t,x$.
We now regard  $w=w(\tau,u)$ as a function of the two variables 
$$\tau= t,\qquad\quad u=u(t,x).$$
Using the identities
$$w_t \,=\, w_\tau + w_u u_t,\qquad \quad w_x~=~w_u\, u_x\,,\quad\qquad w_{xx} \,=\, w_{uu} u_x^2 + w_u\, u_{xx}\,,$$
{}from (\ref{23}) one obtains 
$$\bega{rl} w_\tau(\tau,u)&=w_t - w_u u_t\\[4mm]
&\ds =~-f'(u)  u_x w_u+ \left[  w_{uu} u_x^2 + w_u\, u_{xx}
\right] -u_t  w_u \,,\\[4mm]
&=~\bigl[ - f'(u) u_x + u_{xx}\bigl] w_u + u_x^2 w_{uu}- u_t w_u\\[4mm]
&=~ u_x^2 w_{uu}\,.
\enda$$
Recalling (\ref{22}), we finally obtain
\bel{24} w_\tau~=~\bigl( w-f(u)\bigr)^2\, w_{uu}\,.\eeq
From (\ref{22}) it also follows
\bel{25} 
w_u~=~{w_x\over u_x}~=~{f'(u)u_x - u_{xx}\over u_x}~=~-{u_t\over u_x}\,.\eeq
\v
Next, we review some results on the stability of traveling profiles, for the conservation law with viscosity (\ref{21}).
Let
\bel{tv1}u(t,x)~=~ S(x-\lambda t)\eeq be a traveling wave solution of (\ref{21}), with asymptotic limits $S(\pm\infty) = u^\pm$, as in (\ref{SODE})--(\ref{S0}).
In terms of the variable $w\,=\, f(u) -u_x$ introduced at (\ref{22}),
in the plane with coordinates $(u,w)$ a traveling wave $u=S(x-\lambda t)$ 
corresponds to a stationary solution of (\ref{24}).
Its graph is a segment with endpoints $\bigl( u^+, f(u^+)\bigr)$, $\bigl( u^-, f(u^-)\bigr)$ and  slope
$w_u=\lambda$ equal to the Rankine-Hugoniot speed of the shock, as in (\ref{lc}).
 Indeed,
$$u_t~=~-f'(u) u_x +u_{xx}~=~\lambda u_x\,,$$
while
\bel{wu}w_u~=~f'(u) - {u_{xx}\over u_x}~=~-{u_t\over u_x}~=~\lambda.\eeq

The following lemma provides a quantitative converse to this statement. 
Namely, if the curve 
\bel{gamx}x~\mapsto ~\gamma(x)~\doteq~\Big( u(x), ~f\bigl(u(x)\bigr)-u_x(x)\Big)\eeq
remains within a small strip  around a segment 
with endpoints
$$ P^+\,=\, \bigl(u^+, f(u^+)\bigr),   \qquad P^-\,=\, \bigl(u^-, f(u^-)\bigr),$$
then $u(\cdot)$ is close to a viscous shock profile $S(\cdot)$ 
with asymptotic limits $S(\pm\infty) = u^\pm $.
In the following we denote by 
\bel{tilw}\Tilde w(u)~\doteq~f(u^+) + {f(u^-)-f(u^+)\over u^--u^+}\cdot (u-u^-),\qquad\qquad u\in [u^+, u^-],\eeq
the equation of the  straight line through $P^+, P^-$.

\begin{lemma} \label{l:31} Let $f$ be a smooth flux function, and assume that
the states $u^+<u^-$
are connected by a viscous shock profile $S(\cdot)$.
For any  given $\delta >0$, let $x\mapsto u(x)$ be a smooth profile
which satisfies 
\bel{3.3}\Big|  f\big(u(x)\big)-u_x(x) - \Tilde w\big(u(x)\big)\Big|~\leq~\delta
\qquad\qquad\forall x\in\R,
\eeq
 together with
\bel{3.15}\limsup_{x\to -\infty}~ u(x)>{u^++u^-\over 2}\,,\qquad\qquad \liminf_{x\to +\infty}
~u(x)<{u^++u^-\over 2}\,.\eeq
Then there exists a suitable shift $c\in\R$ such that
\bel{3.16}\big\| u-S(\cdot\,-c)\big\|_{\L^\infty(\R)}~\leq ~\O(1)\cdot \delta.
\eeq
Here $\O(1)$ denotes a quantity which is uniformly bounded
 as long as $u^-,u^+$ range within a given compact set and the distance $|u^+-u^-|$ remains bounded away from zero.
\end{lemma}

For a proof, see Lemma~2 in \cite{BD}.  
Notice that the assumption (\ref{3.15}) is needed to rule out trivial situations such as
$u(x)= u^+$  (or $u(x)= u^-$) for all $x\in\R$.   
Indeed, in this trivial case we would have 
\[ \Big|  f\big(u(x)\big)-u_x(x) - \Tilde w\big(u(x)\big)\Big|~=~\bigl| f(u^+) - \Tilde w(u^+)\bigr|~=~0\qquad\qquad\forall x\in\R.\]
However, $u(x)=u^+$ is a constant function, and cannot satisfy (\ref{3.16}) for any shift $c$.

%
%
%

\section{Review of viscous shock formation}
\label{sec:4}
\setcounter{equation}{0}
In this section we review the analysis in \cite{BD} on the formation of a viscous shock profile, refining some
of the results in the case where the flux $f$ is uniformly convex.
The setting is as follows. Let $u=u(t,x)$ be a solution to the viscous conservation law (\ref{21}).
As shown in Fig~\ref{f:d53}, assume that the initial data satisfy
\bel{idbox}\left\{\bega{cl}\big|\bar u(x)-u^-\big|\leq\delta_0\quad \qquad &\hbox{if}~~ x\leq a\,,
\\[2mm]
\big|\bar u(x)-u^+\big|\leq\delta_0\quad \qquad &\hbox{if}~~x\geq b\,,
\\[2mm]
\bar u(x)\in [m,M]\quad\qquad &\hbox{for all}~~x\in\R\,.\enda\right.\eeq
We seek an estimate of how long it takes for the solution to become close to a traveling profile.
Notice that (\ref{idbox}) does not require the existence of the limits $\lim_{x\to \pm\infty} \bar u(x)$.
Therefore, as $t\to +\infty$, the solution may not approach asymptotically any viscous shock profile. 
Differently from most literature on the topic, rather than on the asymptotic limit here we focus on the transient behavior of the solution. 

\begin{figure}[ht]
\centerline{\hbox{\includegraphics[width=8cm]{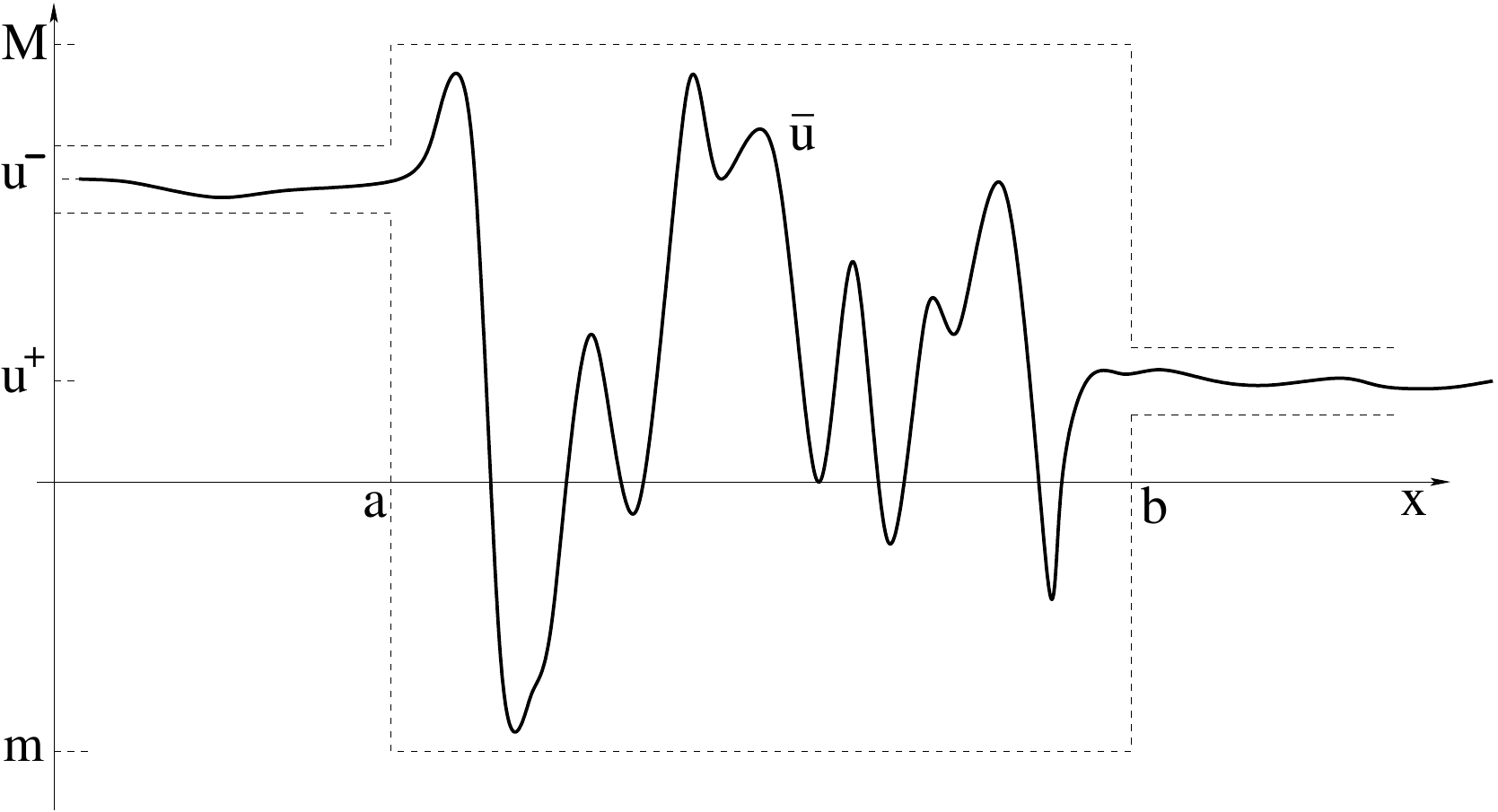}}}
\caption{\small  An initial data satisfying the conditions (\ref{idbox}).}
\label{f:d53}
\end{figure}

As time progresses, the evolution of the corresponding curve $x\mapsto \gamma(t,x)$ in (\ref{gamx})
is illustrated in Figures \ref{f:asy36} and \ref{f:asy37}. 
\begi
\item Given $\delta_1\geq 2\delta_0$, after a certain time $T_1$ the function $u(t,\cdot)$ takes values
inside the interval $[u^+-\delta_1,\, u^-+\delta_1]$. Hence the curve $\gamma$ 
lies inside the vertical strip where $u^+-\delta_1\leq u\leq \, u^-+\delta_1$, see Fig.~\ref{f:asy36}, right.
\item After a  time $T_2$, the curve $\gamma$ is entirely contained in a neighborhood 
$\Omega^{\delta_1} $  of the convex domain
\bel{Omdef} \Omega~\doteq~\Big\{ (u,w)\,;~~u\in [u^+, u^-],~f(u)\leq w\leq \Tilde w(u)\Big\},\eeq
where $\Tilde w$ is the affine function at (\ref{tilw}), see Fig.~\ref{f:asy37}, left.
\item After a time $T_3$, the curve $\gamma$ is
entirely contained in a neighborhood of the segment with endpoints $P^+, P^-$,
see Fig.~\ref{f:asy37}, right.
By Lemma~\ref{l:31}, this implies that for $t\geq T_3$ the function $u(t,\cdot)$ is well approximated
by a viscous traveling profile.
\endi

\begin{figure}[ht]
\centerline{\hbox{\includegraphics[width=0.60\textwidth]{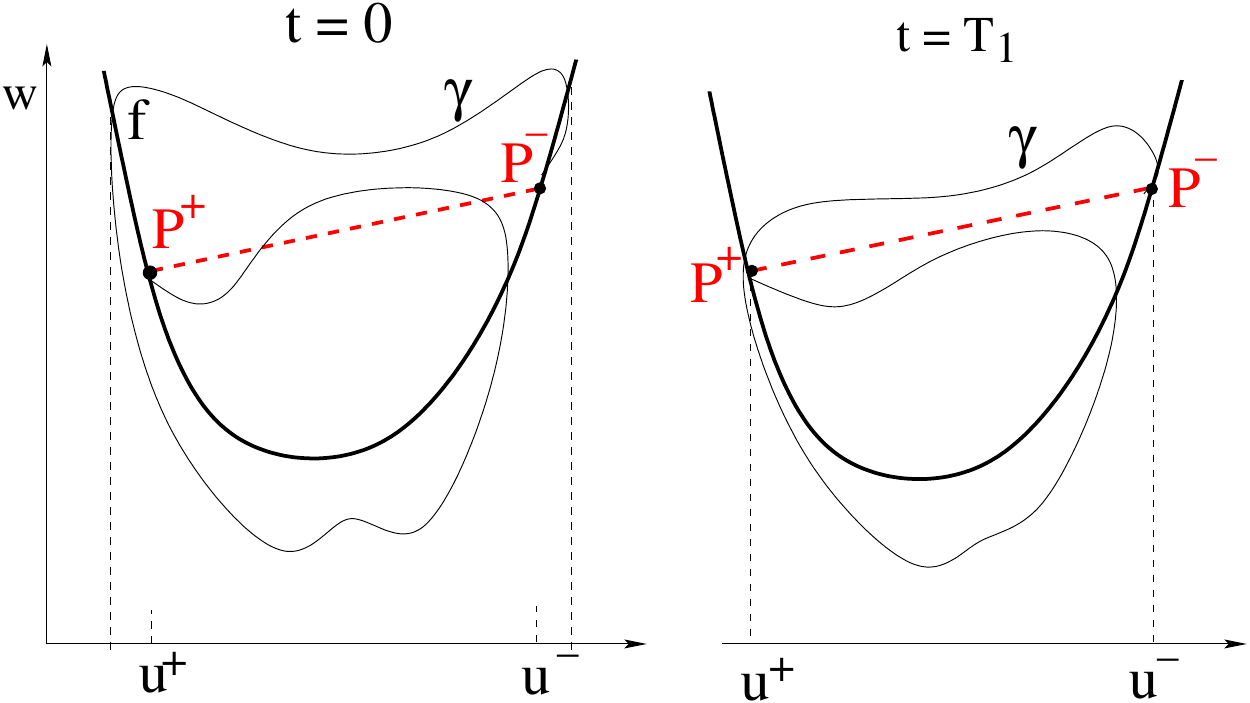}}}
\caption{\small The initial phase in which a viscous shock profile is approached. Left: a curve $\gamma$
corresponding to general initial data, as in 
(\ref{idbox}). Right: after a first phase, the curve $\gamma$ is entirely contained in a vertical strip, where
$u\in [u^+-\delta_1, ~u^-+\delta_1]$.
}
\label{f:asy36}
\end{figure}

\begin{figure}[ht]
\centerline{\hbox{\includegraphics[width=0.60\textwidth]{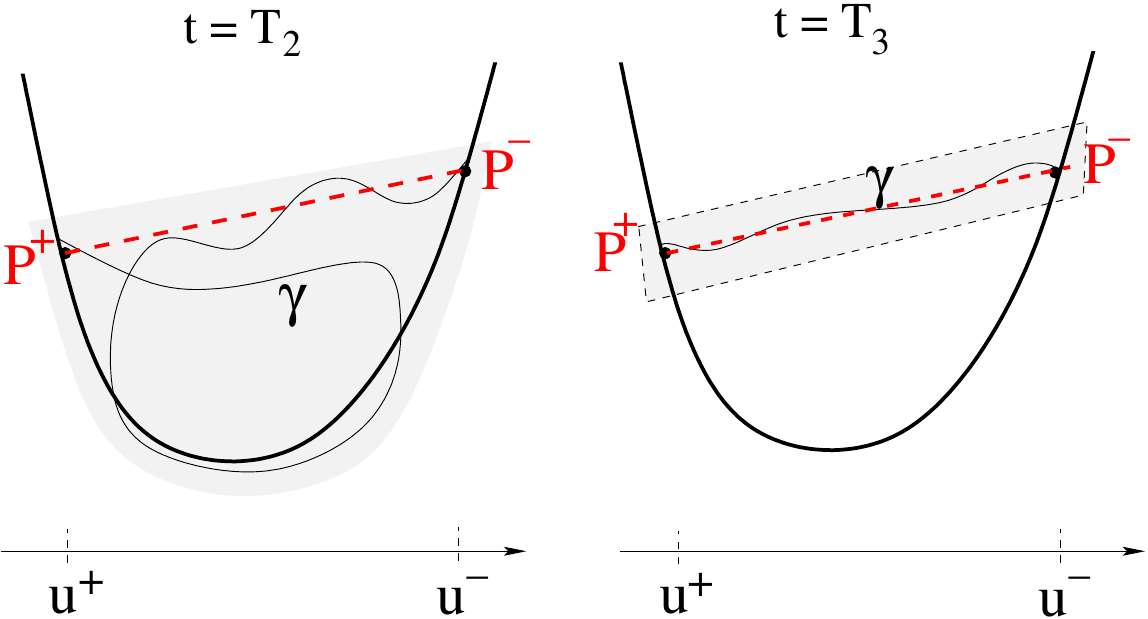}}}
\caption{\small The second and third phase in which a viscous shock profile is approached. Left: 
for $t\geq T_2$, the curve $\gamma$ is entirely contained in a neighborhood of
the convex closure of the graph of $f$, between $u^+$ and $u^-$. Right: at the end of a third phase,
the curve $\gamma$ is contained in a thin strip around the segment with endpoints $P^+, P^-$.
This implies that for $t\geq T_3$ the function $u(t,\cdot)$ is well approximated by a viscous 
traveling profile.}
\label{f:asy37}
\end{figure}

The next lemmas provide precise results in this direction. Together with the domain $\Omega$, for any $\delta_1>0$
we consider the convex neighborhood 
\bel{omd}
\Omega^{\delta_1}~\doteq~\ov{co} \Big\{ \bigl(u,f(u)+\xi\bigr)\,;~~u\in [u^+-\delta_1,~ u^-+\delta_1],~~
|\xi|\leq \delta_1\Big\},
\eeq
where $\ov{co}(S)$ denotes the closed convex hull of a set $S\subset\R^2$.
\begin{lemma}\label{l:41}
Let $u=u(t,x)$ be a solution of (\ref{21}), for a smooth flux function $f$ satisfying (\ref{fuc}) and   with initial data as in (\ref{idbox}).
Then
\begi
\item[(i)] Given $\delta_1\in [2\delta_0, 1]$, for every
\bel{T1}
t~\geq~T_1~\doteq~{8(M-m)(b-a)\over c_1 \delta_1^{2}}\,,
\qquad c_1\,\doteq\, \min_{u\in [m,M]} f''(u),\eeq
one has 
\bel{uin}
u(t,x)~\in~[ u^+-\delta_1\,,~u^-+\delta_1].\eeq
\item[(ii)] 
For all $x\in\R$ and
\bel{T2}
t~>~T_2~\doteq~\max\left\{{1\over c_{1}\delta_{1}}\,,\ {8(M-m)(b-a )+8c_{1}+\bigl(u^{-}-u^{+}+2\delta_{1}\bigr)^2c_{1}\over 2c_{1}\delta_{1}^{2}}\right\},
\eeq

 there holds
\bel{uomd} \Big( u(t,x),~ f\bigl(u(t,x)\bigr)-u_x(t,x)\Big)\,\in\,\Omega^{\delta_1}
.\eeq
\item[(iii)] In addition, for $t\geq T_1$ the solution $u(t,\cdot)$ satisfies the property 
\bel{adec}
x_1< x_2\qquad\implies\qquad u(t,x_1)~\geq~u(t, x_2) - 2\delta_1\,.\eeq
\endi
\end{lemma}

{\bf Proof.} 
{\bf 1.} Observe  that the area below  the graph of $u(t,\cdot)$ and above the line $u= u^-+\delta_0$.
 is non-increasing in time, by the entropy condition.  Hence, by the assumption~\eqref{idbox} on the initial datum $\overline u$, one has
$$A(t)~\doteq~\int \bigl[u(t,x)- u^- - \delta_0\bigr]_+\, dx~\leq~\int \bigl[\bar u(x)- u^- - \delta_0\bigr]_+\, dx
~\leq~(M-u^-)(b-a).$$

Moreover, by genuine nonlinearity we have Oleinik's inequality
\bel{Ole}u_x(t,x)~\leq~{1\over c_1 t}\qquad\forall t>0,~ x\in\R.\eeq
Set
$$h(t)~\doteq~\sup_x \bigl[ u(t,x) - u^- - \delta_0\bigr]_+\,.$$
For any $0<\eta<h(t)$, there is a point $x_{\eta}$ such that that a triangle with vertices $(x_{\eta},h(t)-\eta)$, $(x_{\eta}+h(t)c_{1}t,0)$, $(x_{\eta},0)$ lies entirely below the graph of $u$. This implies
$${1\over 2} h^2(t)\cdot c_1 t~\leq~\int \bigl[\bar u(x)- u^- - \delta_0\bigr]_+\, dx~\leq~(M-u^-)(b-a),$$
and hence
$$h(t)~\leq~\sqrt{2(M-u^-)(b-a)\over c_1 t}\,.$$
Similarly, calling 
$$k(t)~=~\sup_x \bigl[  u^+ - \delta_0-u(t,x) \bigr]_+\,,$$
one finds
$$k(t)~\leq~\sqrt{2(u^+-m)(b-a)\over c_1 t}\,.$$
We thus conclude that for all times $t>0$ and all $x\in\R$ one has
\bel{umax}
u^+-\delta_0 -\sqrt{2(u^+-m)(b-a)\over c_1 t}~\leq~
u(t,x)~\leq~u^- + \delta_0 + \sqrt{2(M-u^-)(b-a)\over c_1 t}\,.\eeq
The statement (i) is an immediate consequence of the estimate, for $t~\geq~T_1$, 
\[
\delta_{1}~\geq~ \delta_0 +\sqrt{2(M-m)(b-a)\over c_1 T_{1}}
~\geq~ \delta_0 +\sqrt{2\max\{u^{+}-m,M-u^{-}\}(b-a)\over c_1 t}\,,
\]
since by assumption $\delta_{1}/2\leq \delta_{1}-\delta_{0}$.
\v
{\bf 2.} To prove (ii), observe that, by Oleinik's inequality,
\bel{decay}u_x(t,x)~\leq~{1\over c_1 t}\,,\qquad\qquad w(t,u)~\geq~f(u) - {1\over c_1 t}\,.\eeq
This yields the lower bound on $w$, when $T_2> {1\over c_1\delta_1}$.
To prove an upper bound,
consider the polynomials 
$$p(u)=A+Bu - (u^2/2),\qquad\qquad q(u) = A'+B'u,$$ choosing $A,B,A',B'$
so that
$$p(u^+-\delta_1) ~=~p(u^-+\delta_1)~=~1, \qquad\qquad q(u^+-\delta_1) = f(u^+-\delta_1), \quad q(u^-+\delta_1) = f(u^-+\delta_1).$$
More explicitly:
\[
A~=~1- { (u^{-}+\delta_{1}) (u^{+}-\delta_{1})\over 2}\,, \qquad B~=~\frac{u^{-}+u^{+}}{2}\,,
\]
  \[A'~=~ -\frac{f(u^-+\delta_1)\left( u^+-\delta_1\right)-f(u^+-\delta_1)  \left( u^-+\delta_1\right) }{(u^{-}+\delta_1)-({u^{+}-\delta_1)}}\,,\qquad
   B'~=~-\frac{f(u^+-\delta_1)-f(u^-+\delta_1)}{(u^{-}+\delta_1)-({u^{+}-\delta_1)}}\,.
  \]

Notice that both $p$ and $q-f$ are nonnegative and
\[\max_{u\in [ u^+-\delta_1,\,u^-+\delta_1]}p(u)~=~p(B)~=~ 1+{(u^{-}-u^{+}+2\delta_{1})^{2}\over 8}~\geq~ 1+{\delta_{1}^{2}\over 2}\,.\]

By the definition of $p$ and $q$, the function
\bel{w+def}w^+(t,u)~=~{p(u)\over  t-T_1}+ q(u)+{\delta_{1}\over 2}\eeq
satisfies $w^+_{uu}= {-1\over t-T_1}$ with  boundary values
\[
w^+(t,u^+-\delta_1)~=~{\delta_{1}\over 2}+{1\over  t-T_1}+ f(u^+-\delta_1)\,, \qquad w^+(t,u^-+\delta_1)={\delta_1\over 2}+{1\over  t-T_1}+ f(u^-+\delta_1)\,.
\] 
Furthermore, for $t>T_1+{8+ (u^{-}-u^{+}+2\delta_{1})^2\over 2\delta_{1}^{2}}$ one has  $w^+_{uu}(t,u)<0$ and
\[
{w^+_t(t,u)\over w^+_{uu}(t,u)}~=~ {p(u)\over  t-T_1}
~\leq~ {\delta_{1}^{2}\over 4}~\leq~\left({p(u)\over  t-T_1}+q(u)-f(u)+{\delta_1\over 2}\right)^{2}.
\]
This shows that $w^{+}$ is an upper solution to (\ref{24}) (see Fig.~\ref{f:asy38}, left), 
and we thus obtain (\ref{uomd}). 

\v

{\bf 3.} To prove  (iii) assume, by contradiction, that $x_1<x_2$ but 
\bel{e:absurdLemma}u(t, x_1)+2\delta_1 ~<~u(t, x_2) .\eeq
 By part (i) there holds $$u(t, x_{1})< (u^{-}+\delta_{1})-2\delta_{1}=u^--\delta_1\,,$$ therefore $u(t, x_{1})<  u^--\delta_1$.

As shown in Fig.~\ref{f:asy29}, right,
consider the bounded region between the graph of $u(t,\cdot)$ and the horizontal line
$u= u(t, x_2)$.  
The fact that we are cutting the epigraph (or the ipograph, in another case)
with a horizontal line at level  $u\in [ u^++\delta_1,~u^--\delta_1]$ guarantees 
that at time $t=0$ the enclosed region is bounded, with finite area.
 This area  decreases in time, and cannot be greater than
$(M-m)(b-a)$.    On the other hand, by (\ref{Ole}) this region contains 
a triangle with 
height 
$h=u(t, {x_2})-u(t, x_1)>2\delta_{1}$ and area
$${1\over 2} h^2 c_1T_1~\leq~(M-m)(b-a).$$
The choice (\ref{T1}) thus implies  $h\leq \delta_1$, reaching a contradiction with 
(\ref{e:absurdLemma}).
\endproof

\begin{figure}[ht]
\centerline{\hbox{\includegraphics[width=0.6\textwidth]{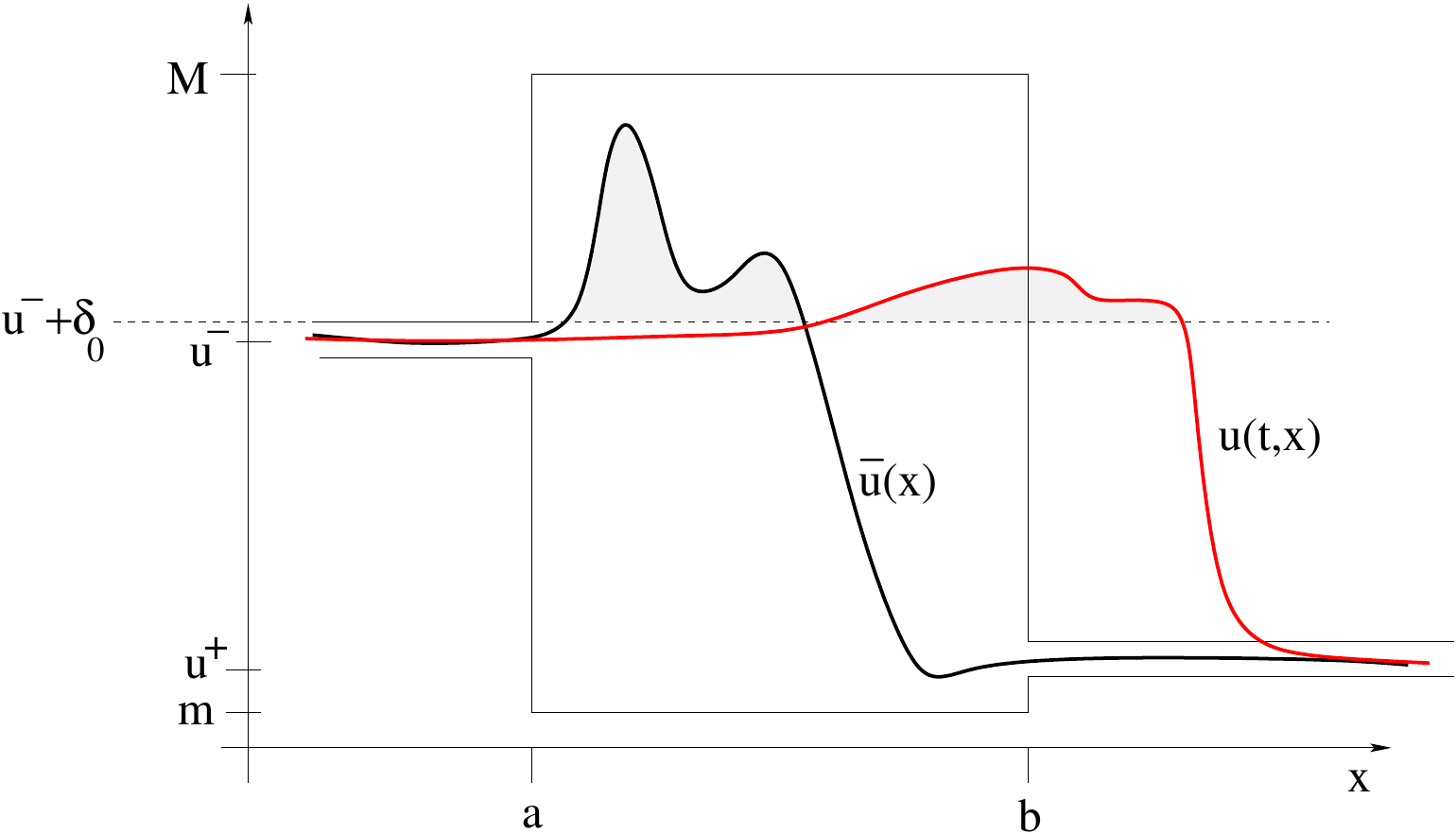}\qquad \includegraphics[width=0.45\textwidth]{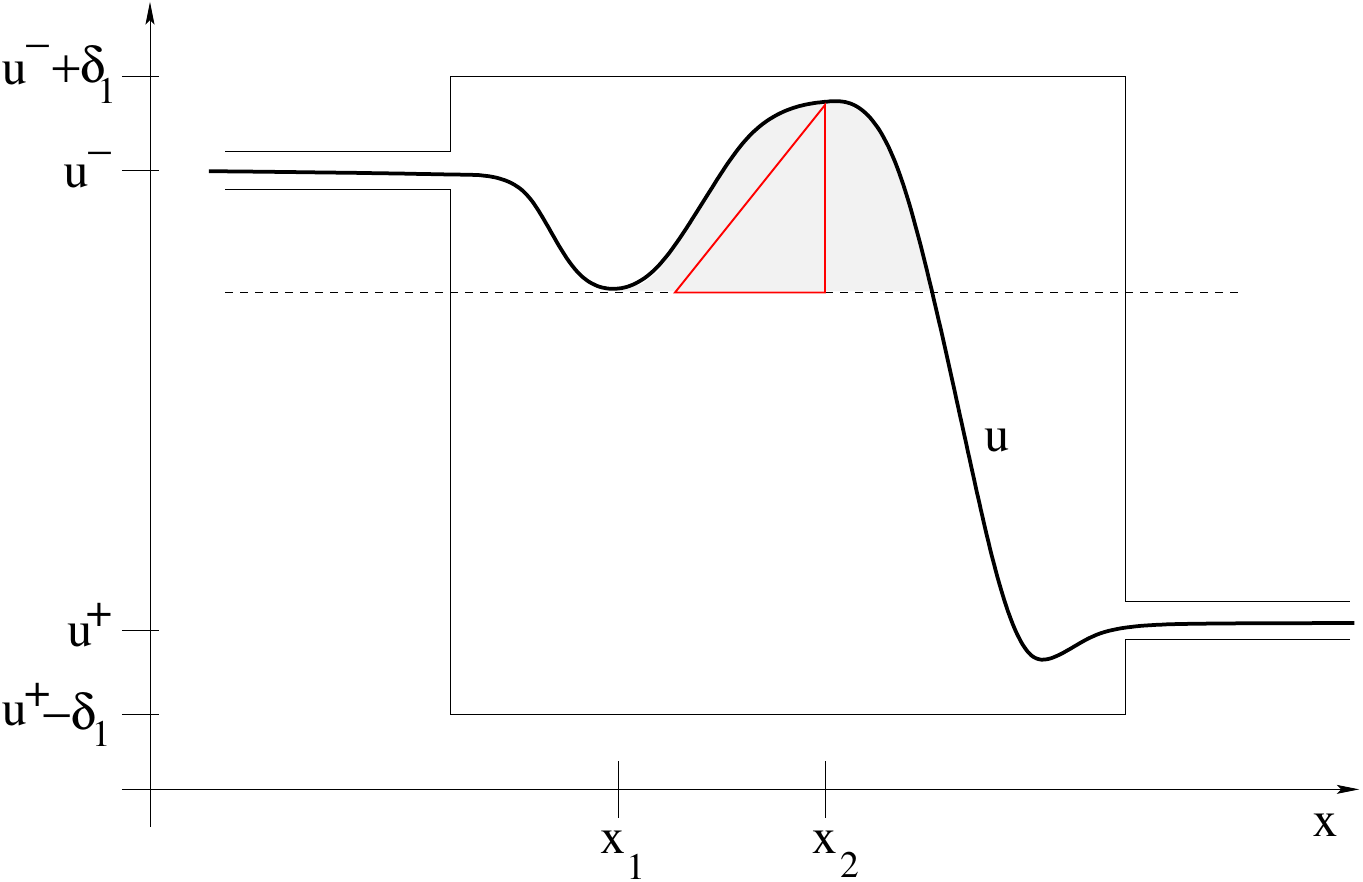}}}
\caption{\small Left: proof of Lemma~\ref{l:31}, part (i).   Right: the construction to prove part (iii).}
\label{f:asy29}
\end{figure}

\begin{figure}[ht]
\centerline{\hbox{\includegraphics[width=0.9\textwidth]{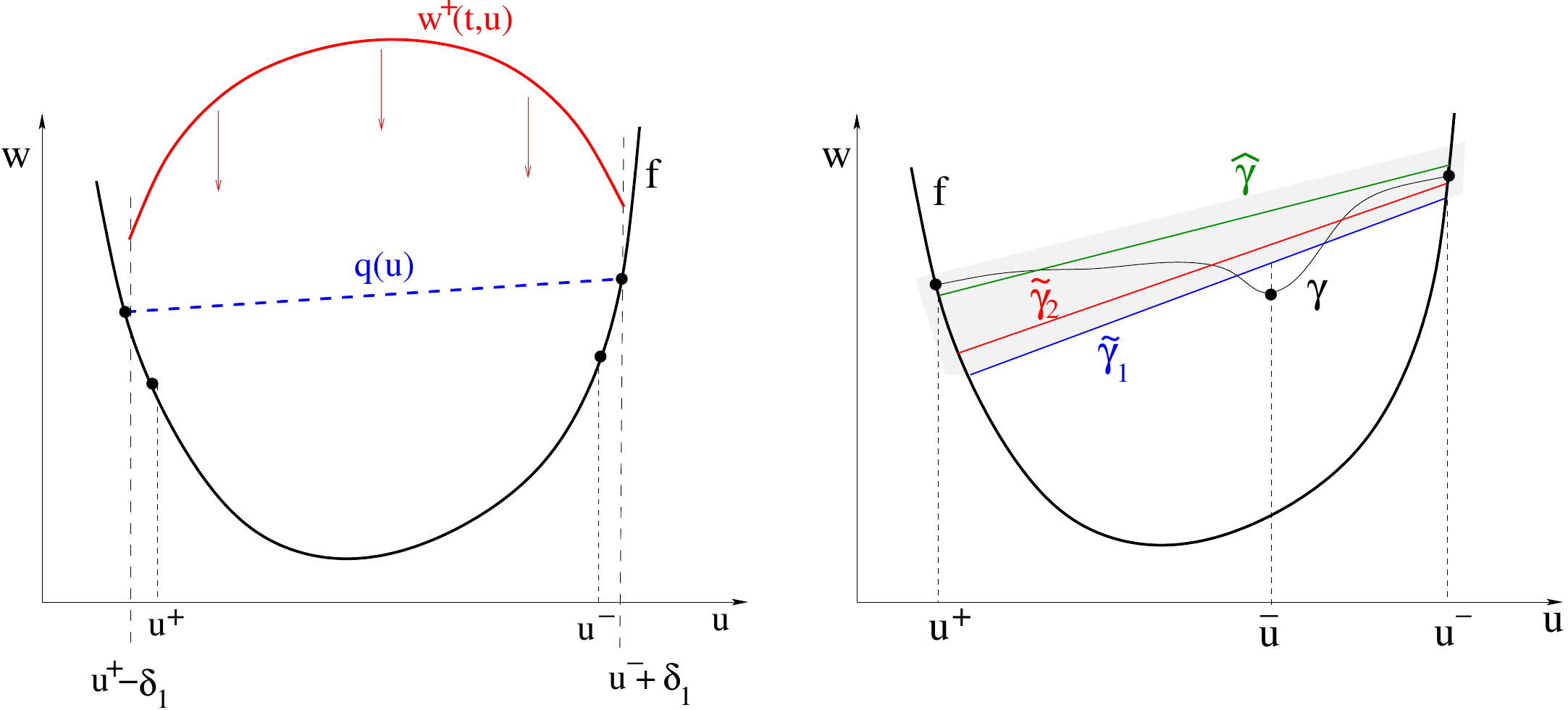}}}
\caption{\small Left: the upper solution $w^+(t,u)$ constructed at (\ref{w+def}).
Right: the construction used in the proof of Lemma~\ref{l:43}.
The curve $\gamma$ corresponding to the graph of $u(T_2,\cdot)$, as defined at (\ref{gamx}), and the 
segments $ \Tilde \gamma_1, \Tilde \gamma_2, \Hat\gamma$, corresponding to the traveling profiles 
$\Tilde \psi_1, \Tilde\psi_2, \Hat \psi$.  In (\ref{Tppmm}), the $u$-coordinates of the endpoints of the segments 
$ \Tilde \gamma_1, \Tilde \gamma_2, \Hat\gamma$ are denoted by $\Tilde \psi_1^\pm, \Tilde\psi_2^\pm, \Hat\psi^\pm$ respectively.}
\label{f:asy38}
\end{figure}

In the setting considered in Lemma~\ref{l:41}, for $t\geq T_2$ the  solution $u(t,\cdot)$ satisfies
all properties (\ref{uin}), (\ref{uomd})-(\ref{adec}).   However, it will no longer satisfy (\ref{idbox}), in general.
For a given $\delta^*>0$, we shall seek an interval $\bigl[a^*(t), b^*(t)\bigr]$ such that 
\bel{ab'}\left\{\bega{rl}&\big|\bar u(x)-u^-\big|\leq\delta_0+\delta^*\qquad \qquad \qquad\qquad x\leq a^*(t),
\\[1mm]
&\big|\bar u(x)-u^+\big|\leq\delta_0+\delta^*\qquad \qquad \qquad\qquad x\geq b^*(t).
\enda
\right.
\eeq
This will be achieved by constructing upper and lower solutions.

Consider an initial data $\bar u$ satisfying (\ref{idbox}).   An upper solution $\Hat u$ is obtained
by solving the linear Cauchy problem with constant coefficients
\bel{Usol}\Hat u_t + f'(M) \Hat u_x~=~\Hat u_{xx}\,,\qquad \qquad 
\Hat u(0,x)\,=\,\left\{ \bega{cl} M \qquad &\hbox{if}~~x<b,\\[1mm]
u^++\delta_0\qquad &\hbox{if}~~x>b.\enda\right.\eeq

Since $m\leq u\leq M$ and $f'$ increasing, while $\Hat u$ is decreasing, we have
$$\Hat u_t ~=~ -f'(M) \Hat u_x+\Hat u_{xx}~\geq~- f'(\Hat u) \Hat u_x+\Hat u_{xx} \,,$$
showing that $\Hat u$ is indeed an upper solution.  A direct computation yields
\bel{u2+}
\Hat u\bigl(t, x-f'(M) t\bigr)~=~u^+ + \delta_0 + 
{M-u^+-\delta_0\over 2\pi} \int_{-\infty}^{x/ \sqrt t} e^{-y^2}\, dy.\eeq
As proved in  \cite{CDS}, the complementary error function satisfies the bound
\bel{cerf}{2\over \sqrt \pi} \int_x^{+\infty}  e^{-y^2}\, dy~\leq~{e^{-x^2}+ e^{-2x^2}\over 2} ~\leq~  e^{-x^2} \qquad\qquad \forall x\geq 0.\eeq
Hence, for any given $\delta^*>0$, the last term in (\ref{u2+}) will be $\leq \delta^*$
provided that
$$e^{-x^2/t} ~\leq~{4\sqrt \pi \delta^*\over M-u^+-\delta_0}\,,\qquad\qquad x^2~\geq ~ - t \cdot \ln \left({4\sqrt \pi \delta^*\over M-u^+-\delta_0}\right).$$
In conclusion, by choosing
$$b^*(t)~=~b + f'(M) \, t  +\sqrt{t \cdot \left|\ln \left({4\sqrt \pi \delta^*\over M-u^+-\delta_0}\right)\right|} $$
we achieve
$$\Hat u(t,x)~\leq~u^+ + \delta_0  + \delta^*\qquad\forall x<a^*(t).$$
Similarly, a lower solution  can be constructed by solving
\bel{Lsol}u_t + f'(m) u_x~=~u_{xx}\,,\qquad u(0,x)\,=\,\left\{ \bega{cl} m \qquad &\hbox{if}~~x>a,\\[1mm]
u^--\delta_0\qquad &\hbox{if}~~x<a.\enda\right.\eeq
The above analysis can be summarized as

\begin{lemma}\label{l:42} Let $u=u(t,x)$ be a solution of (\ref{21}), for a smooth flux function $f$ satisfying (\ref{fuc}) and with initial data as in (\ref{idbox}).    Then for any $t>0$ and $\delta^*>0$ one has
\bel{ab2}\left\{\bega{rl}u(t,x) &\geq ~u^--\delta_0-\delta^* \qquad\qquad \forall x\leq a^*(t),
\\[1mm]
u(t,x)&\leq~ u^++\delta_0+\delta^*\qquad \qquad \forall x\geq b^*(t),
\enda
\right.\eeq
where
\bel{apbp}\bega{l}
\ds a^*(t)~\doteq~
a + f'(m) \, t  -\sqrt{t \cdot \left|\ln \left({4\sqrt \pi \delta^*\over M-m}\right)\right|}\,,
\\[6mm]
 \ds b^*(t)~\doteq~b + f'(M) \, t  +\sqrt{t \cdot \left|\ln \left({4\sqrt \pi \delta^*\over M-m}\right)\right|}\,.
\enda
\eeq
\end{lemma} 

When $t= T_2$ is the time in (\ref{T2}),
choosing $\delta^* = e^{-T_2}$, from (\ref{apbp}) it follows
$$a^*(T_2)~\geq~a- C_2T_2,\qquad b^*(t)\leq b+C_2T_2,$$
for a suitable constant $C_2$.   Notice that, without loss of generality,
by a coordinate shift we can assume $a^*= -b^*$.

To study what happens during the third phase, we consider an initial data $\bar u(x) = u(T_2, x)$ satisfying
\bel{5.1}\left\{\bega{cl}\big|\bar u(x)-u^-\big|\leq\delta_0+\delta^*\qquad  \qquad & x\leq -b^*\,,
\\[1mm]
\big|\bar u(x)-u^+\big|\leq\delta_0+\delta^*\qquad \qquad  & x\geq b^*\,,
\\[1mm]
\bar u(x)\in \big[u^+-\delta_1\,,~u^-+\delta_1\big]\qquad\qquad & x\in\R\,.\enda
\right.
\eeq
together with
\bel{5.2} 
 \bigl( \bar u, f(\bar u)-\bar u_x\bigr)\,\in\,\Omega^{\delta_1}\qquad\forall x\in\R,\eeq
\bel{5.3}
x_1< x_2\qquad\implies\qquad \bar u(x_1)~\geq~\bar u(x_2) - 2\delta_1\,.\eeq

\begin{lemma}\label{l:43} There exist constants $C, C'$ such that the following holds.
 Let $u=u(t,x)$ be a solution of (\ref{21}) whose initial data at time $T=T_2$ 
 satisfy (\ref{5.1})--(\ref{5.3}),
 with $\delta_0+\delta^*<\delta_1/2$. 
 Then for any 
$\delta\geq C' \delta_1$, the solution $u(t,\cdot)$ satisfies the inequality 
(\ref{3.3}) for every $x\in\R$, provided that
\bel{5.18}t~\geq~T_3~\doteq~T_2+ {C |\ln\delta_1|+2 b^*\over \delta}\,.\eeq
\end{lemma}
\v

\begin{figure}[ht]
\centerline{\hbox{\includegraphics[width=1.0\textwidth]{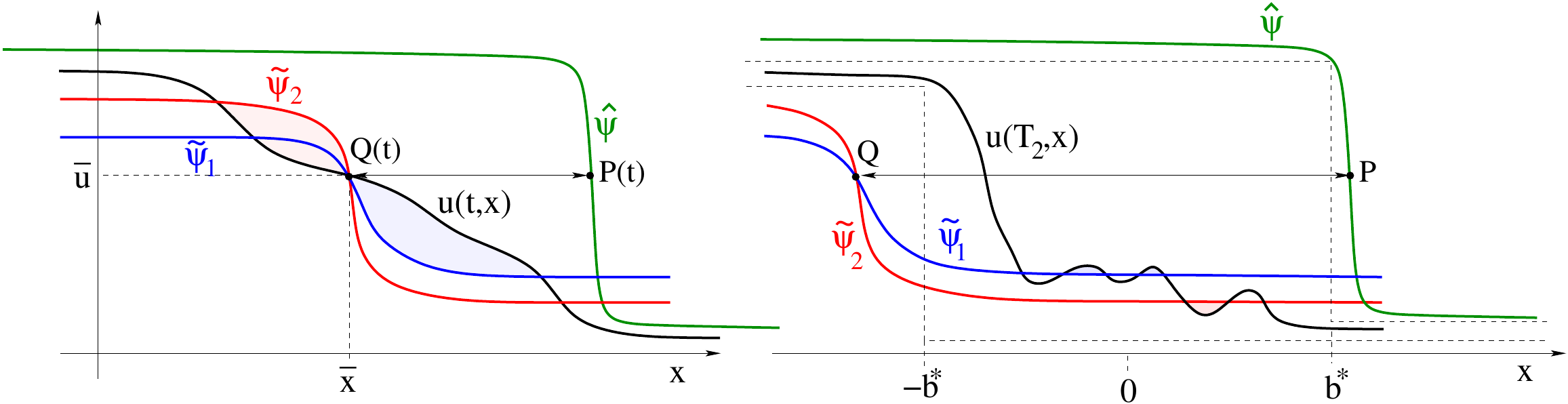}}}
\caption{\small  Left: in the setting shown in Fig.~\ref{f:asy38}, right, there exists a bounded region
where $u(t, x)<\Tilde\psi_2(x-\Tilde \lambda t)$, $x<\bar x$, and a region where
$u(T, x)>\Tilde \psi_1(x-\Tilde \lambda t)$, $x>\bar x$.
Right: if $t-T_2$ is sufficiently large, at the earlier  time $T_2$ the region where $u(T_2,x) >\Tilde\psi_1(x-\Tilde \lambda T_2)$
lies entirely to the left of the region where  $u(T_2,x) <\Tilde\psi_2(x-\Tilde \lambda T_2)$. 
This leads to a contradiction.
Here $\Hat \psi$ is a traveling profile that yields an upper bound for $u$ at all times $t\geq T_2$.
 }
\label{f:asy48}
\end{figure}

{\bf Proof.} 
{\bf 1.} The assumption (\ref{5.2}) implies that at the initial time 
the upper estimate \bel{33}  f\big(u(x)\big)-u_x(x) - \Tilde w\big(u(x)\big)~\leq~\delta
\qquad\qquad\forall x\in\R,
\eeq
is already achieved, with $\delta=\delta_1$.   
Since the domain $\Omega^{\delta_1}$ is convex, hence positively invariant, this upper estimate remains valid for all 
$t\geq T_2$.
%

{\bf 2.} 
To provide a lower bound, we shall compare the solution $u=u(t,s)$ with  three traveling wave profiles:
$\Hat \psi, \Tilde \psi_1,\Tilde\psi_2$.
(see Fig.~\ref{f:asy48}).
The asymptotic limits of these traveling profiles will be denoted by
\bel{limpsi}\Hat \psi^\pm~\doteq~\lim_{s\to\pm\infty} \Hat\psi(s), \qquad\qquad \Tilde \psi_i^\pm~\doteq~\lim_{s\to\pm\infty} \Tilde \psi_i(s),\quad i=1,2.\eeq
We denote by $\Hat \lambda$ the speed of the traveling profile $\Hat \psi$, and  by $\Tilde\lambda$ the common speed of $\Tilde\psi_1,\Tilde\psi_2$.
These values will be chosen so that 
\bel{Tpm} \Hat \psi^-~=~u^- + 3\delta_1, \qquad  \Hat \psi^+~=~u^+ + 3\delta_1,
\qquad\qquad     \Tilde \psi_2^-~=~u^- - 3\delta_1\,, \qquad \Tilde \psi_1^+ = \Tilde\psi_2^++ 3\delta_1\,.
\eeq
Calling $\lambda\doteq {f(u^-)-f(u^+)\over u^--u^+}$ the Rankine-Hugoniot speed of a shock with left and right states $u^->u^+$,  the first two equalities in (\ref{Tpm}) imply that the speed of the  first traveling profile satisfies
$$\Hat\lambda~\doteq~{f(\Hat \psi^-) - f(\Hat\psi^+)\over \Hat \psi^- - \Hat\psi^+}~>~\lambda.$$
We choose the common speed of the two remaining traveling profiles to be
\bel{Tspeed}\Tilde \lambda~=~\Hat\lambda +\delta\,.\eeq
Notice that $\Hat \lambda $ is the slope of the segment $\Hat \gamma$ in Fig.~\ref{f:asy38}, right, while
$\Tilde\lambda$ is the slope of the two parallel segments $\Tilde\gamma_1, \Tilde\gamma_2$.
Notice that, having assigned $\Tilde\psi_2^-$ as in (\ref{Tpm}), the limit $\Tilde\psi_2^+$ is uniquely determined by 
the second of the Rankine-Hugoniot conditions 
\bel{TRH}  {f(\Tilde \psi_1^-) - f(\Tilde \psi_1^+)\over \Tilde  \psi_1^- - \Tilde \psi_1^+}~=~\Tilde\lambda~\doteq~{f(\Tilde \psi_2^-) - f(\Tilde \psi_2^+)\over \Tilde  \psi_2^- - \Tilde \psi_2^+}\,.\eeq
In turn, the last identity in (\ref{Tpm}) determines $\Tilde \psi_1^+$. At last, $\Tilde\psi_1^-$
is uniquely determine by the second of the Rankine-Hugoniot conditions.
Summarizing the above construction, we have (see Fig.~\ref{f:asy38}, right)
\bel{Tppmm}
u^{+} 
~<~
\Hat \psi^+
~<~
\Tilde \psi_2^+
~<~
\Tilde \psi_1^+
~<~
\Tilde \psi_1^-
~<~
\Tilde \psi_2^-
~<~
u^{-}
~<~
\Hat \psi^-.
\eeq

\v
{\bf 3.} Next,
referring to Fig.~\ref{f:asy38}, right, assume that there exists a point $\bar u = u( t, \bar x)$ such that
\bel{cross1}f(u(t, \bar x)) - u_x ( t, \bar x)~<~\Tilde\gamma_1(\bar u).
\eeq
By performing a horizontal shift of the graphs of $\Tilde \psi_1,\Tilde \psi_2$, we can assume 
\bel{cross2}
\Tilde  \psi_1(\bar x-\lambda  t)~=~\Tilde  \psi_2(\bar x-\lambda  t)~=~u( t, \bar x)~=~\bar u\,,
\eeq
As shown in Fig.~\ref{f:asy48}, left, the inequality (\ref{cross1})  means that the graph of $\Tilde\psi_1$ crosses the graph of $u(t,\cdot)$ 
downwards, at the point $\bar x$. Hence the two graphs must also intersect at two additional  points, 
because $u^->\Tilde\psi_1^-> \Tilde\psi_1^+>u^+$.
\v
{\bf 4.}
We now go back in time, and look at the relative position of the graphs of $u, \Tilde\psi_1,\Tilde\psi_2,\Hat \psi$
at time $T_2$ (see Fig.~\ref{f:asy48}, right).  We choose the smallest possible 
shift in the profile $\Hat \psi$ so that
\bel{Hatp} u(T_2,x)~\leq~\Hat\psi(x-\Hat\lambda T_2)\qquad\qquad\forall x\in \R.\eeq
 By assumption, $u(T_2,\cdot)$ satisfies the bounds in (\ref{5.1}).
Consider the points $Q$ and  $P$ where 
$$\Tilde \psi_1(Q-\Tilde\lambda T_2)~=\Tilde \psi_2(Q-\Tilde\lambda T_2)~=~\bar u,
\qquad\qquad \Hat \psi(P-\Hat \lambda T_2)=\bar u.$$
Since the speeds of the traveling profiles satisfy $\Tilde \lambda =\Hat\lambda+\delta$, we have
$|P-Q|\geq (t-T_2)\delta$.   If $t$ is large enough (to be determined later), this implies that only possible 
intersections of the graph of $u(T_2,\cdot)$ with the graphs of $\Tilde\psi_1,\Tilde\psi_2$ occur in the region where  
\bel{nearlim}|\Tilde\psi_1- \Tilde \psi_1^+|<\delta_1, \qquad\qquad |\Tilde\psi_2- \Tilde \psi_1^+|<\delta_1\,.\eeq
By the second identity in (\ref{Tpm}), we have the implication
\bel{impl}\left\{ \bega{l} u(T_2,x)~\geq~ \Tilde \psi_1(x-\Tilde \lambda T_2),\\[1mm]
 u(T_2,y)~\leq ~ \Tilde\psi_2(y-\Tilde \lambda T_2),\enda\right.\qquad\implies\qquad x<y.\eeq
Indeed, since $\Tilde \psi_1^+ - \Tilde\psi_2^+ = 3\delta_1$, by (\ref{Tpm})  this follows from the fact that (\ref{nearlim}) implies $\bar u(x)\geq \bar u(y) -2\delta_1$ and the  ``almost monotonicity" property (\ref{5.3}).

\begin{figure}[ht]
\centerline{\hbox{\includegraphics[width=1.0\textwidth]{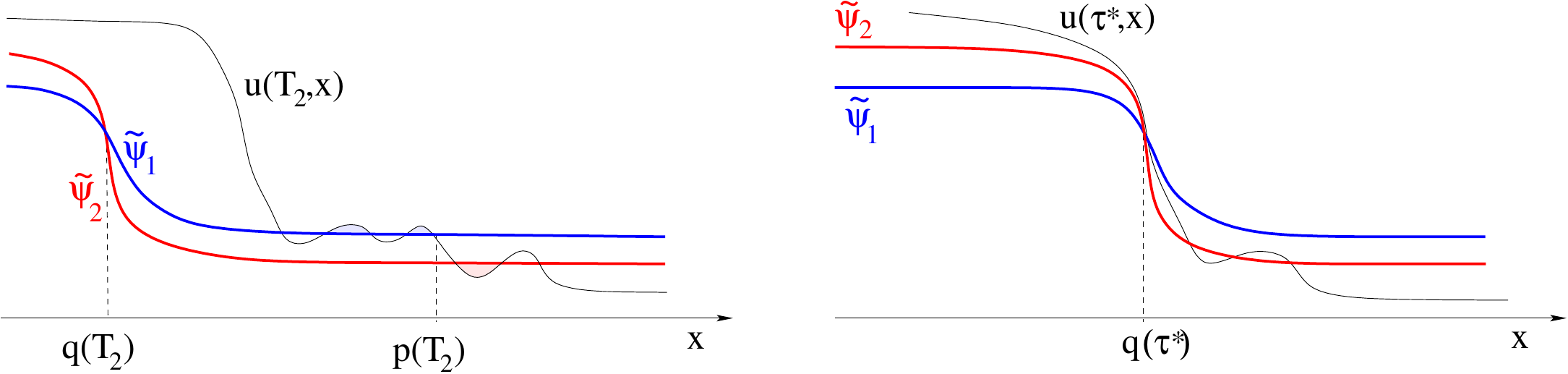}}}
\caption{\small  Left: at time $t= T_2$, the crossings between the graphs of $u$ and $\Tilde\psi_1$ remain to the left of the crossings
between the graphs of $u$ and $\Tilde \psi_2$.   Right: for $t-T_2$ large enough, the graphs of $u(t,\cdot)$ and 
$\Tilde\psi_1$ cross at a single point.}
\label{f:asy35}
\end{figure}

{\bf 5.} By \eqref{impl} 
at time $T_2$  the region where 
$u(T_2,x)> \Tilde \psi_1(x-\Tilde \lambda T_2)$
lies entirely to the left of the region where $u(T_2,x)<  \Tilde \psi_2(x-\Tilde \lambda T_2)$
(see Fig.~\ref{f:asy35}, left).
We will  show that this property is retained at all subsequent times. 
As a consequence, if (\ref{cross2}) holds,
this will lead to a contradiction.

Toward our goal, for $\tau\in [T_2, t]$, consider the 
 right-most intersection (see  Fig.~\ref{f:asy35}, left)
\bel{e:comparisonp}p(\tau)~\doteq~\sup\bigl\{ x\,;~u(\tau,x)>\Tilde\psi_1(x-\Tilde \lambda \tau )\bigr\}
\qquad \hbox{so that}\qquad u\bigl(\tau, p(\tau)\bigr) = \Tilde \psi_1(x-\Tilde\lambda \tau).\eeq

At time $T_{2}$, (\ref{impl}) implies
\bel{inCond}
u(T_2,y)~\geq~\Tilde\psi_2(y-\lambda T_2)\qquad\text{for} \qquad y\leq x=p(T_2).
\eeq
Moreover, let $q(\tau)$ be the point where the two traveling profiles cross, so that
$$\Tilde \psi_1\bigl(q(\tau)- \Tilde \lambda \tau\bigr)~=~\Tilde \psi_2\bigl(q(\tau)-\Tilde \lambda \tau\bigr),\qquad\qquad q(T_{2})=Q.$$
Notice that  the choice of the shift at (\ref{cross2}) yields $u\bigl(t,q(t)\bigr)=
\Tilde \psi\bigl(q(t)-\Tilde \lambda t \bigr)$.

By (\ref{inCond}), in view of \eqref{Tppmm}, \eqref{nearlim} and \eqref{e:comparisonp},  as long as $p(\tau)>q(\tau)$ a comparison argument on the domain 
$$\bigl\{ (\tau,x)\,;~\tau\geq T_2, ~x\leq p(\tau)\bigr\}$$
yields
\bel{ine2}u(\tau,x)~\geq~\Tilde\psi_2(x-\Tilde \lambda \tau)\qquad\forall x\leq p(\tau).\eeq

Next, consider the time
$$\tau^*~=~\sup\bigl\{ \tau\in [T_2,t]\,;~p(\tau)>q(\tau)\bigr\}.$$
At time $\tau^*$, since $p(\tau^{*})=q(\tau^{*})$, by~\eqref{e:comparisonp} and \eqref{ine2} we have
$$u(\tau^*,x) \,\geq\, \Tilde \psi_2(x-\Tilde \lambda \tau^*)\,\geq\,\Hat \psi_1(x-\Tilde \lambda \tau^*)\qquad\forall x<q(\tau^*),$$
$$u(\tau^*,x) \,\leq\, \Tilde \psi_1(x-\Tilde \lambda \tau^*-\xi)\qquad\forall x>q(\tau^*).$$
Therefore, the graph of $u(\tau^*,\cdot)$ crosses the graph of $\Tilde \psi_1(\cdot -\Tilde\lambda \tau^*)$ at the single point $q(\tau^*)$.
As a consequence, for all times $\tau\geq \tau^*$ only one crossing of these two graphs can occur.
In particular, a single crossing occurs at time $t$ . This rules out the situation shown in Fig.~\ref{f:asy48}, left,
reaching the desired contradiction.
\v
{\bf 6.} To complete the proof, it remains to estimate more precisely the time interval which is needed to 
carry out the above construction.

Since  the traveling profiles converge to their asymptotic limits (\ref{limpsi}) at exponential rate,
the condition (\ref{Hatp}) will be satisfied if $P-b^*= C_1| \log \delta_1|$ for some constant $C_1$ 
Similarly, the conditions 
\bel{nlim}\Big|\Tilde\psi_1(x-\Tilde\lambda T_2)- \Tilde \psi_1^+\Big|\,<\,\delta_1, \qquad 
\Big|\Tilde\psi_2(x-\Tilde\lambda T_2)- \Tilde \psi_2^+\Big|\,<\,\delta_1
\qquad\qquad\forall~ x\geq -b^*,\eeq
will be achieved provided that  $-b^*-Q\geq C_2| \log \delta_1|$, for some constant $C_2$.

Referring to Fig.~\ref{f:asy48}, left, call $Q(t), P(t)$ the points such that
$$\Tilde \psi_1\bigl(Q(t)-\Tilde\lambda t\bigr)~=\Tilde \psi_2\bigl(Q(t)-\Tilde\lambda t\bigr)~=~\bar u,
\qquad\qquad \Hat \psi\bigl(P(t)-\Hat \lambda t\bigr)=\bar u.$$
Since $u(t,x)\leq \Hat\psi(x-\Hat\lambda t)$, we trivially have $P(t)\geq Q(t)$.
This implies
$$P-Q~=~P(t)-Q(t) + \delta(t-T_2)~\geq~(C_1+C_2) |\log\delta_1 | + 2b^*$$
provided that 
$$t-T_2~\geq~{(C_1+C_2) |\log\delta_1 | + 2b^*\over\delta}\,.$$
This yields (\ref{5.18}) with $C=C_1+C_2$, completing the proof.
\endproof

\v

Combining the three Lemmas \ref{l:41}, \ref{l:42}, and \ref{l:43}, one obtains
\begin{theorem}\label{t:41}
Let $f$ be a smooth flux function satisfying (\ref{fuc}), and assume that
the states $u^+<u^-$
are connected by a viscous shock profile $\phi(\cdot)$.
Then there exists a constant $C$ such that the following holds.
For every $\delta_0>0$ and $\delta>4\delta_0$ sufficiently small, if $\bar u$ is an initial
condition satisfying (\ref{idbox}), then the corresponding solution
$u=u(t,x)$ of (\ref{21}) satisfies
\bel{1.11}\Big\|
u(t,\cdot)-\phi\bigl(t, \cdot-c(t)\bigr)\Big\|_{\L^\infty}~\leq ~C\delta\,,\eeq
for a suitable shift $c(t)$ and all times $t>T = \O(1)\cdot (1+b-a) \,\delta^{-2}$.

Here $\O(1)$ denotes a quantity depending only on $f, u^-, u^+, m,M$, 
which remains uniformly bounded as $u^-, u^+$ range over a compact set and $u^--u^+$ remains uniformly positive.
\end{theorem}

{\bf Proof.} {\bf 1.} 
Given $\delta>0$ we set $\delta_1 = \sqrt \delta$, $\delta_2=\delta$.    Applying Lemmas~\ref{l:41} and \ref{l:42}
we conclude that, at  a suitable time 
\[T_2 ~
\doteq~\max\left\{{1\over c_{1}\sqrt{\delta}}\,,\ {8(M-m)(b-a )+8c_{1}+\bigl(u^{-}-u^{+}+2\sqrt{\delta}\bigr)^2c_{1}\over 2c_{1}\delta }\right\}
~\leq ~ \O(1)\cdot {(b-a+1) \over\delta }
\]
 the solution 
satisfies 
\bel{td1}\left\{\bega{rl}&\big|\bar u(T_2,x)-u^-\big|\leq\delta_0+\delta^*\qquad  \qquad\qquad x\leq a^*\,,
\\[1mm]
&\big|\bar u(T_2,x)-u^+\big|\leq\delta_0+\delta^*\qquad \qquad \qquad x\geq b^*\,,
\\[1mm]
&\bar u(T_2,x)\in \big[u^+-\delta_1\,,~u^-+\delta_1\big]\qquad\qquad x\in\R\,.\enda
\right.
\eeq
where  $\delta^*=e^{-T_2}$ and $[a^*,b^*]$ is an interval such that $b^*-a^*= \O(1)\cdot T_2~=~\O(1) \cdot {1+b-a\over\delta}$.
\v
{\bf 2.} Next, for 
$$t~\geq~T_3 ~=~T_2 + \O(1)\cdot { |\ln \delta_1| + (1+b^*-a^*)\over\delta_2}
~=~ \O(1)\cdot { 1+b-a\over \delta^2}\,,$$
by Lemma~\ref{l:43}  the solution $u(t,\cdot)$ satisfies the inequality (\ref{3.3}).
The conclusion thus follows from Lemma~\ref{l:31}.
\endproof

\begin{remark}\label{r:41}{\rm Lemma~\ref{l:31} does not give information about the shift $c(t)$. 
For this purpose, one can use the inequality
$$\left| u\bigl(t, c(t)\bigr)- {u^-+u^+\over 2}\right|~\leq~C\delta.$$
If we choose
\bel{dsmall}\delta_0 +\delta^* + C\delta~<~{u^--u^+\over 2}\,,\eeq
by Lemma~\ref{l:42} the estimates in (\ref{ab2})-(\ref{apbp}) hold. 
This implies $c(t)\in \bigl[ a^*(t), b^*(t)\bigr]$.  In particular, for $t>\!>1$ we have
\bel{cest}
c(t)~\in~\bigl[ a-\hat\lambda t, ~b+\hat \lambda t\bigr],\eeq
where $\hat \lambda> \max\big\{ - f'(m), f'(M)\bigr\}$ is an upper bound for all characteristic speeds.
}
\end{remark}

\section{Interaction of viscous shocks}
\label{sec:5}
\setcounter{equation}{0}

Aim of this section is to construct a viscous solution $u^\ve$ of (\ref{viscep}) corresponding to the interaction 
of two shocks.
Toward this goal, consider three states
$u^->u^*>u^+$. Let $S_1, S_2$ be two viscous traveling profiles for 
the equation 
\bel{W1} u_t + f(u)_x~=~u_{xx}\,,\eeq
 connecting the states $(u^-, u^*)$ and the states $(u^*, u^+)$, 
respectively (see Fig.~\ref{f:asy4}). These profiles are uniquely determined up to a shift.  
To remove this ambiguity we assume
$$S_1(0)\,=\, {u^-+u^*\over 2}\,,\qquad\qquad S_2(0)\,=\, {u^*+u^+\over 2}\,.$$
The corresponding Rankine-Hugoniot speeds are
$$\lambda_1~=~{f(u^-) - f(u^*)\over u^--u^*}~>~
 {f(u^*) - f(u^+)\over u^*-u^+}~=~\lambda_2\,.$$

We will construct a solution $W=W(t,x)$ of  (\ref{W1}),
defined for all 
$(t,x)\in \R^2$, with the following property. As $\tau\to -\infty$, the profile $W(\tau,\cdot)$ 
approaches the superposition of
two traveling wave profiles, centered at the points
$$\lambda_1\tau~<~\lambda_2 \tau.$$
Moreover, as $\tau\to +\infty$, $W(\tau,\cdot)$  approaches a viscous traveling profile $S$
connecting the states $u^-, u^+$, with Rankine-Hugoniot speed
\bel{RHL}\lambda~=~{f(u^-) - f(u^+)\over u^--u^+}\,.\eeq

\begin{figure}[ht]
\centerline{\hbox{\includegraphics[width=16cm]{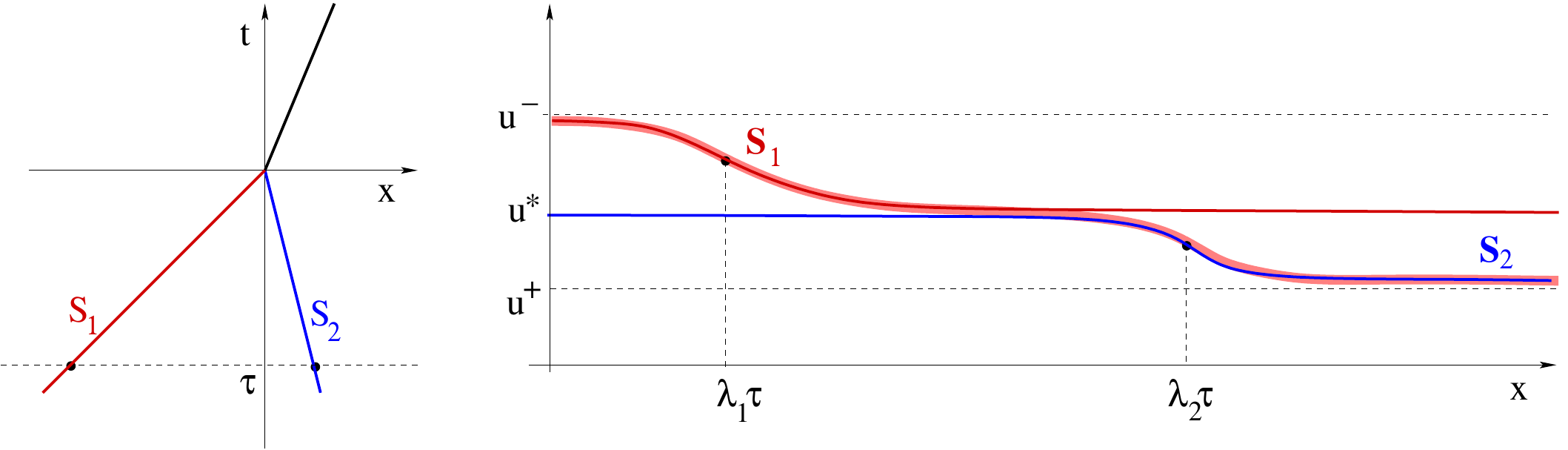}}}
\caption{\small An initial data at $\tau<\!< 0$, obtained by 
interpolating two viscous traveling profiles.}
\label{f:asy4}
\end{figure}

\begin{lemma}\label{l:51}
There exists a unique solution $W=W(t,x)$ of (\ref{W1}), defined for all $(t,x)\in\R^2$,
such that 
\bel{45} \lim_{\tau\to -\infty} \left\{ \int_{-\infty}^{\lambda^* \tau} \Big|
W(\tau,x) - \phi_1(x - \lambda_1\tau)\Big| \, dx + \int^{+\infty}_{\lambda^* \tau} \Big|
W(\tau,x) - \phi_2(x - \lambda_2\tau)\Big| \, dx 
\right\}~=~0\eeq
for any  $\lambda^*$ with $\lambda_2<\lambda^*<\lambda_1$.
\end{lemma} 
%
\v
{\bf Proof.} {\bf 1.} Let $\theta:\R\mapsto [0,1]$ be a smooth, nondecreasing function
such that 
$$\theta(s)~=~\left\{ \bega{rl} 0\quad &\hbox{if}~~x\leq 0,\cr
1\quad &\hbox{if}~~x\geq  1.\enda\right.$$

To construct the solution $W$, 
for a fixed $\tau<\!<0$ we let  $W^{(\tau)}(t,x)  $ be the solution 
to (\ref{W1}) defined for  $t\in [\tau, +\infty[\,$, with initial data at $t=\tau$ given by 
\bel{47}  W^{(\tau)}(\tau, x) ~=~\theta(x-\lambda^* \tau) S_2(x-\lambda^*\tau) + 
\bigl(1-\theta(x-\lambda^* \tau)\bigr)  S_1(x-\lambda^*\tau) .\eeq
\v
{\bf 2.} We claim that, as $\tau\to -\infty$, the solutions $W^{(\tau)}$  defined on $[\tau, +\infty[\,\times \R$
converge to a unique limit solution of (\ref{W1}), defined on the whole plane $\R^2$.
Indeed, consider any two times $\tau'<\tau<\!< 0$.
Since the $\L^1$ distance between any two solutions to (\ref{W1}) 
does not increase in time, we have the error estimate
\bel{48}
\bigl\|  W^{(\tau)}(\tau, \cdot) - W^{(\tau')}(\tau, \cdot) \bigr\|_{\L^1}~
\leq~\int_{\tau'}^\tau E(s)\, ds,\eeq
where the instantaneous error rate is
\bel{49}\bega{rl}E(s)&\ds=~\int_{-\infty}^{+\infty} \left| 
{d\over ds} \Big[W^{(s)} (s, x) \Big]+ f\bigl( W^{(s)}\bigr)_x - W^{(s)}_{xx} 
\right|\, dx\\[4mm]
&\ds =~\int_{\lambda^* s}^{\lambda^* s + 1} \left| 
{d\over ds}\Big[ W^{(s)} (s, x)\Big] + f\bigl( W^{(s)}\bigr)_x - W^{(s)}_{xx} 
\right|\, dx\\[4mm]
&=~\O(1)\cdot e^{ cs},
\enda\eeq
for some constant $c>0$.   Indeed, the  profiles $S_1, S_2$ converge exponentially to the same limit $u^*$ as $x\to +\infty$ and $x\to -\infty$, respectively.

This proves that  any sequence $W^{(\tau_n)}$ is Cauchy, as $\tau_n\to -\infty$.
Hence a solution exists.   
%
\v
{\bf 3.} 
To prove uniqueness, let $W,\Tilde W$ be any two solutions. 
Fix any time $t\in\R$ and let $\tau<t$.  Then
$$\bigl\| W(t)-\Tilde W(t)\bigr\|_{\L^1} ~\leq~\bigl\| W(\tau)-\Tilde W(\tau)\bigr\|_{\L^1} .$$
Letting $\tau\to -\infty$, by (\ref{45}) the right hand side approaches zero.
Hence the left hand side vanishes as well.
\endproof

\section{Viscous approximations on the smooth region}
\label{sec:6}
\setcounter{equation}{0}

As a preliminary,
we recall that, by Kuznetsov's error estimate (see \cite{HR}), the $\L^1$ distance between the 
 solution $u_\ve$ of the viscous equation  (\ref{viscep}) and the original solution $u$ of (\ref{1}), with the same initial data (\ref{2}), is bounded by
 \bel{kerr}
 \bigl\| u_\ve(t,\cdot) - u(t,\cdot)\bigr\|_{\L^1}~\leq~C_*\,\ve^{1/2},\qquad\qquad t\in [0,T],\eeq
for some constant $C_*$ depending only on the flux $f$ and on the total variation of the initial data $\bar u$.
{}From (\ref{kerr}) an elementary argument yields a pointwise error bound on  regions where 
the inviscid solution $u$ is 
Lipschitz continuous.
%

\begin{lemma} \label{l:61}  Let the assumption {\bf (A1)} hold, and consider any initial data $\bar u\in BV$.
Assume that, for some $t>0$, on the interval $[x_1, x_2]$ the solution $u(t,\cdot)$ of (\ref{1})-(\ref{2}) 
is Lipschitz continuous. 
Then,  for  all $\ve>0$ sufficiently small, the solution $u_\ve$ of (\ref{viscep}) 
with the same initial data satisfies
\bel{err2}\bigl| u_\ve(t,x)-u(t,x)\bigr|~\leq~2C_*\ve^{1/6}\quad\forall \quad x\in \left[x_1+\ve^{1/3}, \, x_2-\ve^{1/3}\right].\eeq
\end{lemma}
{\bf Proof.}
By the convexity assumption (\ref{fuc}), the gradient $u_{\ve,x}$ satisfies the  bound 
\bel{supgrad} u_{\ve,x}(t,x)~\leq ~{1\over c_1 t}\,.\eeq
Let $L$ be a Lipschitz constant for $u(t,\cdot)$ on the interval $[x_1, x_2]$.
Assume that, at some point $y\in [x_1+\ve^{1/3}, x_2]$ one has
$u_\ve(t,y)-u(t,y)\doteq\delta>0$.
Then, for $x\in [y- \ve^{1/3}, y]$ one has
$$u_\ve(t,x)-u(t,x)~\geq ~\delta - \left({1\over c_1 t} + L\right) (y - x).$$
Integrating, and recalling (\ref{kerr}), we obtain
$$C_* \ve^{1/2} ~\ge~\int_{y-\ve^{1/3}}^y \max\bigl\{ 0, u_\ve(t,x)- u(t,x)\bigr\} \, dx~\geq~\min
\left\{ {\delta\over 2} \,\ve^{1/3},~{\delta^2\over 2} \left({1\over c_1 t} + L\right)^{-1}\right\}.$$
For $\ve>0$ small, this implies $C_* \ve^{1/2}\geq \delta \ve^{1/3}/2$, hence
 the inequality in (\ref{err2}).
The same inequality is obtained if $u_\ve(t,y)-u(t,y)<0$
at some point $y\in \left[x_1, ~x_2-\ve^{1/3}\right]$.
\endproof

The next lemma yields a kind of ``localization principle".
For a solution of the conservation law (\ref{1}), assume that all backward characteristics
from a neighborhood of the point $(\tau,\xi)$ start inside an interval $[a,b]$.  
Then (assuming that shocks are not present), by perturbing the initial data outside
$[a,b]$ the values of the solution near $(\tau,\xi)$ are not affected.
Of course, the same is no longer true for solutions of the viscous equation (\ref{viscep}).
Yet, as the viscosity coefficient $\ve\to 0$, in a neighborhood of $(\tau,\xi)$ these perturbations vanish to high order.  Therefore, they do not affect the convergence 
properties of the local asymptotic rescalings.  Results of this kind are well known in the literature, see for example \cite{SX}.

For convenience, the following lemma refers to solutions obtained after rescaling, which satisfy the equation (\ref{visc1}) with unit viscosity.

\begin{lemma}\label{l:62} 
Let {\bf (A1)} hold and let $U,V$ be two bounded solutions of the viscous conservation law (\ref{visc1}), 
defined for $t\in [0,T]$, whose initial data
satisfy 
\bel{UV0}U(0,x) = V(0,x)\qquad\forall x\geq 0.\eeq
\begi
\item[(i)]
Assume that, for some speed $\lambda\in\R$ and some constants $\kappa>0$ and $r>0$, the following transversality condition holds:
\bel{trasv0}\max \Big\{ f'\bigl(U(t,x)\bigr), ~f'\bigl(V(t,x)\bigr)\Big\}\,\leq\,
\lambda - \kappa\qquad
\forall t\in [0,T], ~~\lambda t\leq x\leq \lambda t + \ve^{-r}.\eeq
Then, for any $q>0$ and $n\geq 1$,  there exists a constant $C$ independent of $\ve$ such that, 
\bel{uvec}
\bigl|U(t,x) -V(t,x)\bigr| \,\leq\,C \ve^q\qquad\qquad\forall t\in [0,T], ~~x\geq \lambda t + \ve^{-r},\eeq
as long as $T\leq \ve^{-n}$.
\item[(ii)]
The same bound (\ref{uvec}) holds if (\ref{trasv0}) is replaced by
\bel{trasv}\max \Big\{ f'\bigl(U(t,x)\bigr), ~f'\bigl(V(t,x)\bigr)\Big\}\,\leq\,\lambda - \ve^\sigma\qquad
\forall t\in [0,T], ~~\lambda t\leq x\leq \lambda t + \ve^{-r},\eeq
for some constant $0<\sigma<r$.
\item[(iii)] The same conclusion (\ref{uvec}) remains valid if the boundedness of $U,V$ is replaced by the assumption
\bel{boUV} \bigl\| U(t,\cdot)\bigr\|_{\C^1([\lambda t + \ve^{-r}, +\infty[)}~\leq~C_0\ve^{-m},
\qquad\qquad \bigl\| V(t,\cdot)\bigr\|_{\C^1([\lambda t + \ve^{-r}, +\infty[)}~\leq~C_0\ve^{-m},
\eeq
for some constant $C_0$ and any $ m\geq 1$.
\endi
\end{lemma}

\v
\begin{remark} {\rm Referring to Fig.~\ref{f:asy53}, right, the key assumption (\ref{trasv}) requires that, as $\ve\to 0$,
$$\hbox{[width of the boundary layer] $\times$ [difference in speed] }~\geq ~\ve^{-r}\cdot \ve^\sigma~\to~+\infty$$
}
\end{remark}
 {\bf Proof of the lemma.} 
 
{\bf 1.} 
The difference $W=U-V$ satisfies the viscous conservation law 
\bel{weq1}
W_t + \bigl[ A(U, V) W\bigr]_x~=~W_{xx}\,,\eeq
with
\bel{Auv}  A(U,V)\,\doteq\,\int_0^1 f'(\theta U + (1-\theta) V\bigr)\, d\theta~\leq~\lambda-\kappa
\qquad\qquad \hbox{for}~~x\in [\lambda t,~ \lambda t+ \ve^{-r}].\eeq
Moreover, the initial data satisfies
\bel{W0}W(0,x)~=~U(0,x)-V(0,x) ~=~0\qquad \hbox{for}~~x\geq 0.\eeq
\v
\begin{figure}[ht]
\centerline{\hbox{\includegraphics[width=8cm]{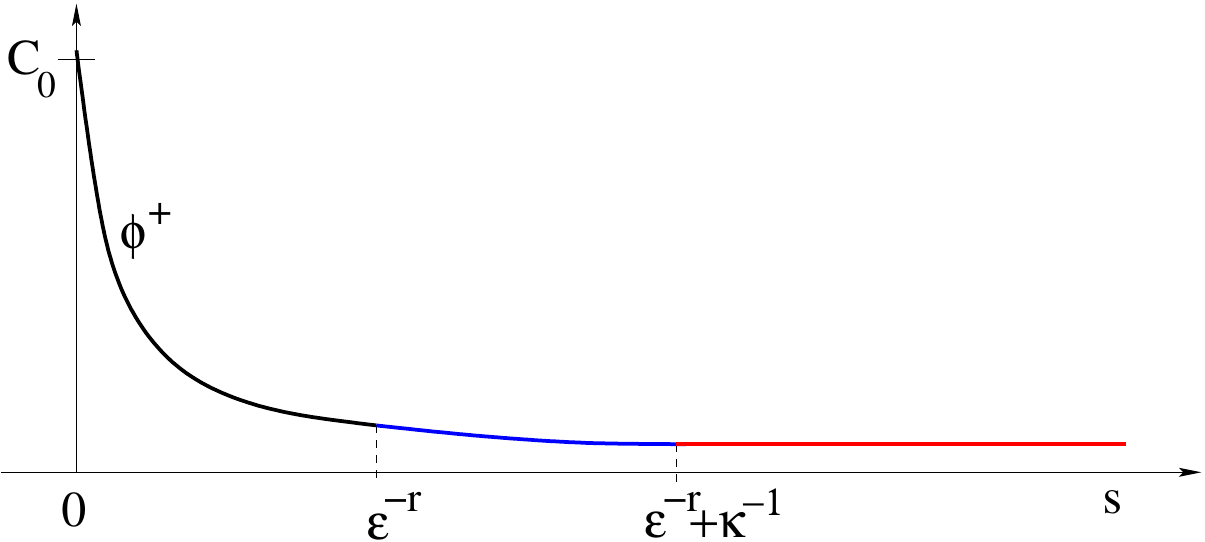}}
\qquad \hbox{\includegraphics[width=6cm]{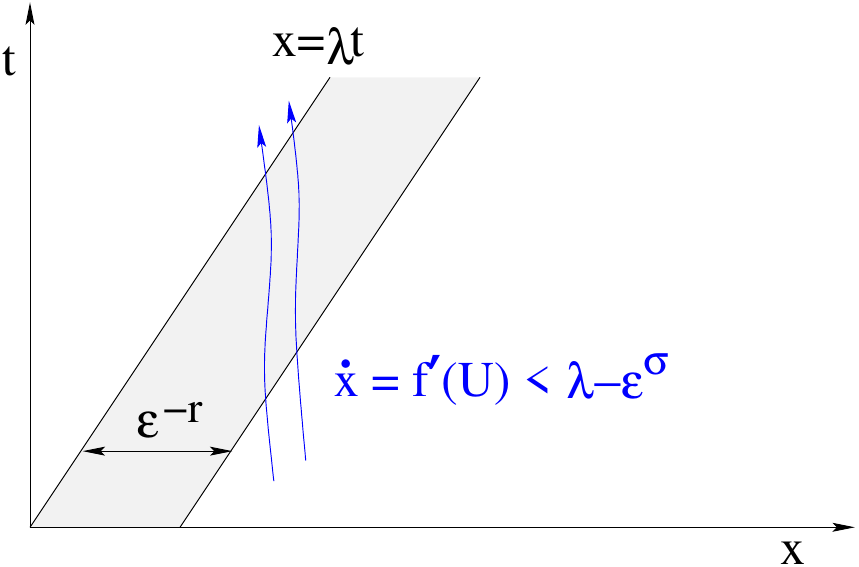}}}
\caption{\small Left: the function $\phi^+(s)$ defined as (\ref{phi+}). Right: the assumption (\ref{trasv})  requires that characteristic curves cross the shaded domain transversally. }
\label{f:asy53}
\end{figure}

{\bf 2.}
We shall estimate $W$ on the region where $x>\lambda t$ by establishing a bound on the integrated function
$$\Phi(t,x)~\doteq~\int_{x}^{+\infty} W(t,y)\, dy.$$
By (\ref{weq1}), this function satisfies
\bel{we2}
\Phi_t - A(U,V) \Phi_x~=~\Phi_{xx}\,,\eeq
with initial and boundary conditions
\bel{Pib}\Phi(0,x)\,=\,0\quad\hbox{for }~~x\geq 0,\qquad\quad \Phi_x(t,  \lambda t)~<~
C_0~\doteq~1+\|U\|_{\L^\infty} + \|V\|_{\L^\infty}\,.\eeq
An upper solution to (\ref{we2})-(\ref{Pib}) will be constructed by piecing together an exponential and a quadratic function, as shown in Fig.~\ref{f:asy53}.
We first define
\bel{phi+}\phi^+(s)~\doteq~\left\{ \bega{cl} C_0e^{-\kappa s} \quad &\hbox{if}\quad s\in [0, \ve^{-r}],\\[2mm]
P(s)\quad &\hbox{if}\quad s\in [\ve^{-r}, \ve^{-r}+\kappa^{-1}],\\[2mm]
P(\ve^{-r}+\kappa^{-1})\quad &\hbox{if}\quad s > \ve^{-r}+\kappa^{-1}.\enda\right.\eeq
Here $P(s)$ is the second order Taylor polynomial of the function $C_0e^{-\kappa s}$ computed at the point
$s_\ve=\ve^{-r}$.   This quadratic function attains its global minimum at $s= \ve^{-r}+\kappa^{-1}$.
We then construct an upper solution by setting
\bel{Phidef}\Phi^+(t,x)~=~\phi^+(x-\lambda t) + C_0\kappa^2 \exp\{ -\kappa \ve^{-r}\} t.\eeq
A lower solution can be constructed in the same way.

\begin{figure}[ht]
\centerline{\hbox{\includegraphics[width=10cm]{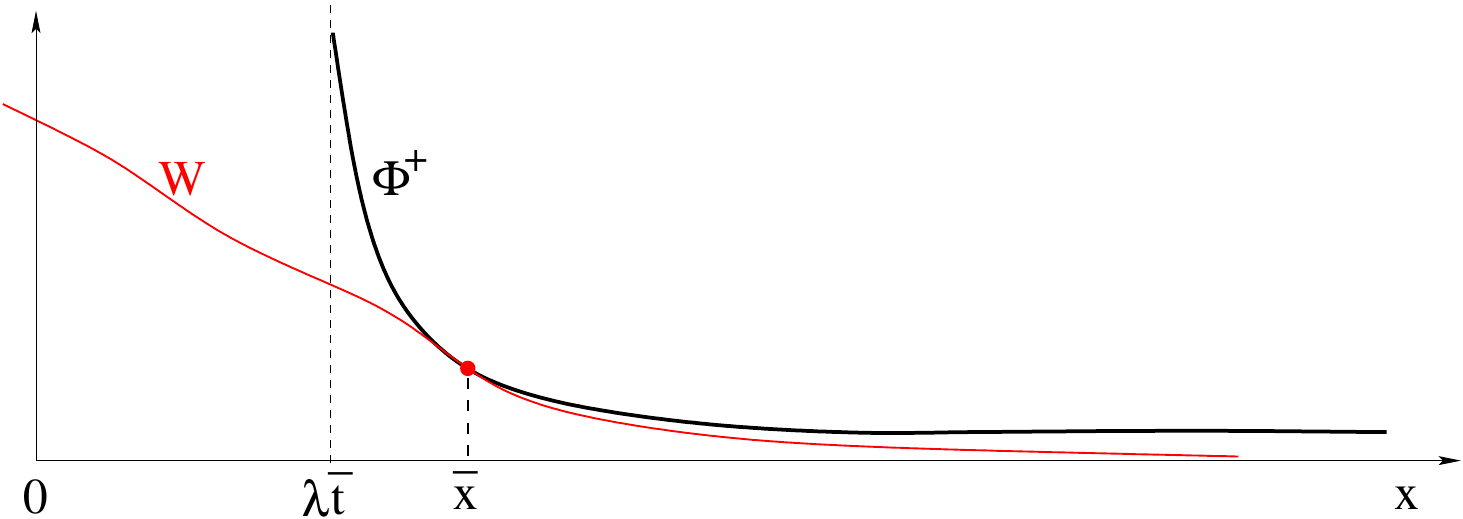}}}
\caption{\small The comparison argument used in the proof of (\ref{Intineq}). }
\label{f:asy54}
\end{figure}
\v
{\bf 3.} 
By a comparison argument, we claim that 
\bel{Intineq}
-\Phi^+(t,x)~\leq~\int_x^{+\infty} \bigl[U(t,y) - V(t,y)\bigr]\, dy~\leq~\Phi^+(t,x)\qquad 
\forall x\geq \lambda t, ~~t\in [0, T].\eeq
Indeed,  
consider a first time $\bar t$ such that, for some $\bar x\geq \lambda \bar t$, one has
$$W(\bar t, \bar x)= \Phi^+(\bar t,\bar x).$$
As shown in Fig.~\ref{f:asy54}, the boundary condition (\ref{Pib}) rules out the possibility that $\bar x= \lambda \bar t$.
Hence $\bar x>\lambda\bar t$ and at the point $(\bar t, \bar x)$ the functions $W$ and $\Phi^+$ satisfy
$$W_x = \Phi^+_x,\qquad W_{xx} \leq~\Phi^+_{xx}\,.$$
This implies $W_t<\Phi^+_t$, reaching a contradiction. 
In a similar way, a contradiction is also reached  if $W(\bar t, \bar x)= -\Phi^+(\bar t,\bar x)$.
\v
{\bf 4.} Assume that $T\leq \ve^{-n}$ for some $n\geq 1$.  On the region 
$$\Gamma\,\doteq\,\Big\{ (t,x)\,;~~x\geq \lambda t + \ve^{-r},\quad t\in [0, T]\Big\},$$
the previous analysis yields
\bel{ineq3} \bega{l}\ds  \left|\int_x^{+\infty} \bigl[U(t,y) - V(t,y)\bigr]\, dy\right|~\leq~C_0\exp\{ - \kappa \ve^{-r}\} +C_0
\kappa^2 \exp\{ -\kappa \ve^{-r}\} \ve^{-n}~
\leq~C' \ve^{2q},\enda
\eeq
for some constant $C'$ and all $\ve>0$ small enough.  
Notice that this constant $C'$ depends on $\kappa, r, n$, but not on $\ve$.
Since $U, V$ are uniformly Lipschitz continuous, from the integral bound (\ref{ineq3}) we recover a pointwise bound on $U-V$, namely
$$\bigl|U(t,x)- V(t,x)\bigr| ~\leq~C \ve^q \qquad \forall  (t,x)\in \Gamma.$$
This yields (\ref{uvec}). 
\v
{\bf 5.} To prove (ii),  assume that (\ref{trasv}) holds. Replacing $\kappa$ with $\ve^\sigma$ in (\ref{ineq3}) we  obtain 
\bel{ineq33} \bega{l}\ds  \left|\int_x^{+\infty} \bigl[U(t,y) - V(t,y)\bigr]\, dy\right|~\leq~C_0\exp\{ - \ve^\sigma \ve^{-r}\} +
\ve^{2\sigma} \exp\{ -\ve^\sigma \ve^{-r}\} \ve^{-n}~
\leq~C' \ve^{2q},\enda
\eeq
for some constant $C'$ and all $\ve>0$ small enough. Hence the same conclusion holds, provided that $0<\sigma<r$.
\v
{\bf 6.} Finally, to prove (iii) it suffices to replace $C_0$ with $\ve^{-m} C_0$ in the previous computations.
The inequality (\ref{ineq33}) can now be replaced by
\bel{ineq34} \bega{l}\ds  \left|\int_x^{+\infty} \bigl[U(t,y) - V(t,y)\bigr]\, dy\right|~\leq~C_0\ve^{-m}\exp\{ - \ve^\sigma \ve^{-r}\} +
\ve^{2\sigma} \exp\{ -\ve^\sigma \ve^{-r}\} \ve^{-n}~
\leq~C' \ve^{2q+m},\enda
\eeq
for all $\ve>0$ small enough. Since by assumption both $U$ and $V$ are Lipschitz continuous
with constant $C_0 \ve^{-m}$, this implies (\ref{uvec}), for a suitable constant $C$.
\endproof

\begin{remark}\label{r:62} {\rm By the same arguments, assuming that 
$U(0,x)=V(0,x)$ for $x\in [a,b]$, one can obtain an estimate of the form
\bel{uvec2}\bigl|U(t,x) -V(t,x)\bigr| \,\leq\,C \ve^q\quad\qquad\forall t\in [0,T], ~~x\notin \bigl[ a+\lambda t + \ve^{-r},~b-\lambda l - \ve^{-r}\bigr].\eeq
Indeed, one can introduce the intermediate function
$$\Tilde U(0,x)~=~\left\{ \bega{rl} U(0,x)\quad &\hbox{if}~~x<b,\\[1mm]
V(0,x)\quad &\hbox{if}~~x\geq b,\enda\right.$$
and compare the corresponding solutions $U(t,x)$ with $\Tilde U(t,x)$, 
then $\Tilde U(t,x)$ with $V(t,x)$.
}
\end{remark}

\section{Proof of Theorem~\ref{t:1}}
\label{sec:7}
\setcounter{equation}{0}
We observe that all the rescaled solutions $U^\ve$ in (\ref{resc11}) are globally bounded and satisfy the parabolic equation (\ref{visc1}) with unit viscosity.
Hence their Lipschitz constant is also uniformly bounded. The 
existence of a subsequence that converges uniformly on compact sets of $\R^2$ 
is thus trivial.   The main issue is to  identify the nontrivial limit which is obtained after a suitable shift.

For any $\ve>0$, let us denote by $u^{(\ve)}(t,x)$ the solution to the conservation law (\ref{1}) without
viscosity, and with initial data
\bel{epind}u^{(\ve)}(0,x)~=~\bar u_\ve(x).\eeq
As $\ve\to 0$, the convergence (\ref{uelim1})
 implies that
\begi
\item In case (i) the solution $u^{(\ve)}$ contains a shock at a point $(\tau, x^\ve)$, with 
$x^\ve\to \xi$ and left and right states $u_\ve^\pm\to u^\pm$.
\item In case (ii), the solution  $u^{(\ve)}$ contains two shocks which merge together at a point
$(t_\ve, x_\ve)\to (\tau,\xi)$, with left, middle and right states $u^-_\ve\to u^-$, $u^*_\ve\to u^*$, 
$u^+_\ve\to u^+$.
\endi
Based on the above remarks,  we shall first work out a  proof in the case where
the initial conditions satisfy
\bel{inv}u_\ve(0,x)~=~\bar u(x)\eeq
for all $\ve>0$.
The general case, where the initial data also depend on $\ve$, will then follow.
%
\v

 In both cases (i)--(ii), we construct
  a family of rescaled solutions $(U^\ve)_{\ve>0}$ to the equation (\ref{21}) with unit viscosity.  We want to show that, as 
  $\ve \to 0$,
  they converge to a specific solution $V$ of the same equation, uniformly for $(t,x)$ in compact subsets. 

The proof relies on two main ingredients.
\begi
\item As $\ve\to 0$, one has the convergence $\bigl\|U^\ve(t_\ve, \cdot)- V(t_\ve,\cdot)\bigr\|_{\L^1([-r_\ve, r_\ve])}\to 0$, 
for a suitable choice of $t_\ve\to -\infty$, $r_\ve \to +\infty$.
\item For the rescaled solution to the inviscid equation, the domains of determinacy $\D_\ve$ of the intervals $J_\ve= \bigl\{ (t_\ve, x)\,;~~|x|\leq r_\ve\bigr\}$ 
cover the entire plane $\R^2$.
\endi

\v

\begin{figure}[ht]
\centerline{\hbox{\includegraphics[width=9cm]{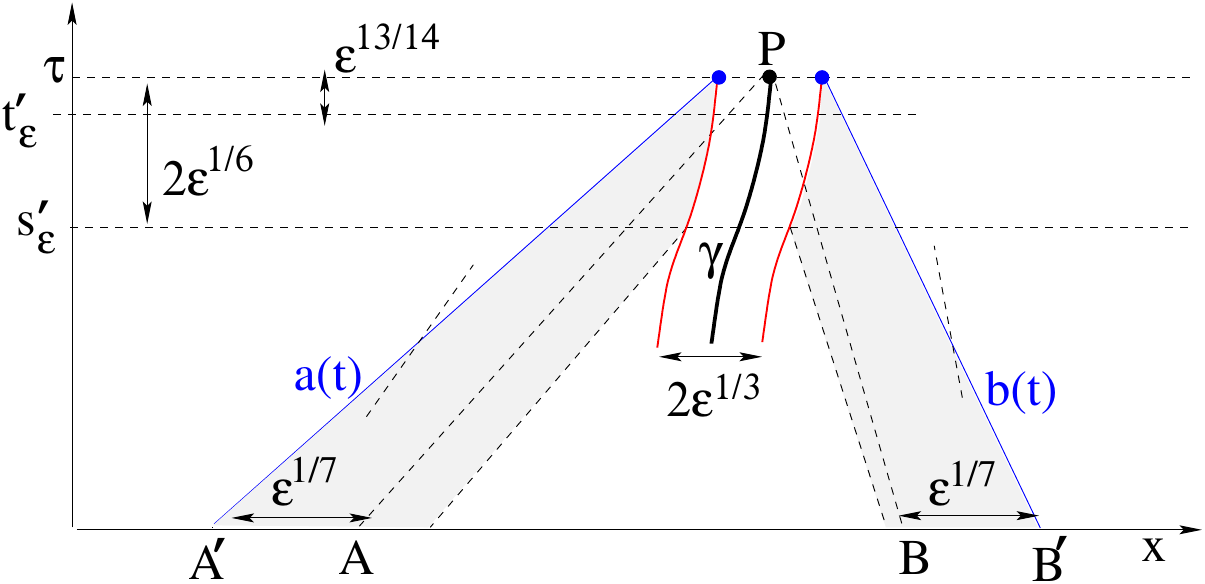}}\qquad \hbox{\includegraphics[width=6cm]{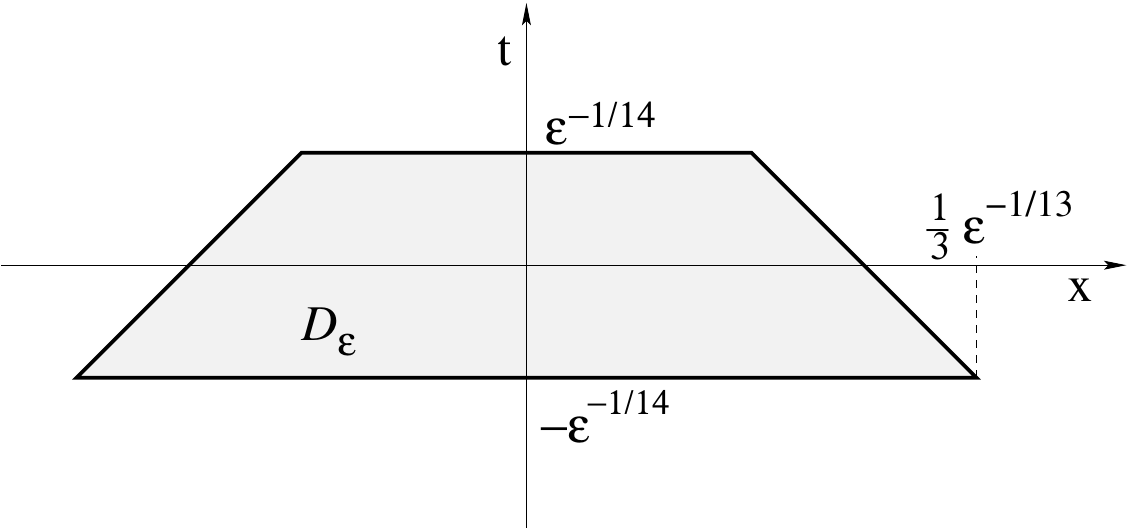}}}
\caption{\small Left: the construction used in the proof of Theorem~\ref{t:1}.  For $|x-\gamma(t)|>\ve^{1/3}$, by 
(\ref{err1}) the viscous approximation $u_\ve$ is close to the inviscid solution $u$. 
  Right: the domain $\D_\ve$  in (\ref{Dep1}).}
\label{f:asy56}
\end{figure}

\subsection{A point along a single shock.}\label{s:71}

{\bf 1.}
Assume that the inviscid solution has a shock  along the curve $x=\gamma(t)$, passing 
through the point $P=(\tau,\xi)$. 
Call 
$u^\pm\doteq u(\tau, \xi\pm)$ the left and right states at time $t=\tau$. W.l.o.g.~we shall assume $\xi=\gamma(\tau)=0$.

By the $\L^1$ error bound (\ref{err2}), on a neighborhood of the point $(\tau,\xi)$ one has the implication
\bel{err1}
\bigl|x-\gamma(t)\bigr|\geq\ve^{1/3}\qquad \implies \qquad \bigl| u_\ve(t,x)-u(t,x)\bigr| \,<\,2C_*\,\ve^{1/6},\eeq
valid for all $\ve>0$ sufficiently small.
As shown in Fig.~\ref{f:asy56}, left, define
$$C\,=\,-3\ve^{1/3},\qquad D\,=\, 3\ve^{1/3},$$
and consider the initial points of the backward characteristics through $P=(\tau,0)$, namely
$$A~=~- \tau\, f'(u^-) ,\qquad B~=~-\tau\,f'(u^+).$$
Consider the additional points
$$A'~=~- \tau\, f'(u^-)- \ve^{1/7} ,\qquad B'~=~
- \tau\,f'(u^+)+\ve^{1/7}.$$
We claim that 
the lines $$x= a(t)\quad\hbox{through ~$(0,A')$,~$(\tau,C)$ \qquad and}\qquad 
 x= b(t)\quad\hbox{through ~$(0,B')$,~$(\tau,D)$} $$
  have slopes $\lambda_a>\lambda_b$ which satisfy the transversality relations
\bel{trans1}\bega{rl}
\lambda_a~>~f'\bigl( u(t,x)\bigr) +2\kappa_1\ve^{1/7}\qquad &\hbox{if}~~ x\in \bigl[a(t)-2\ve^{1/3}, \, a(t)\bigr],\\[2mm]
\lambda_b~<~f'\bigl( u(t,x)\bigr) -2\kappa_1\ve^{1/7}\qquad &\hbox{if}~~ x\in \bigl[b(t), \, b(t)+2\ve^{1/3}\bigr],\enda\eeq
for some $\kappa_1>0$, independent of $\ve$.  Indeed, this follows from the assumption
that the solution $u(\tau,\cdot)$ is $\C^1$ to the left and to the right of the shock.
In particular, calling 
$$x(t,y) ~\doteq ~y+t f'\bigl(\bar u(y)\bigr)$$
 the position of the characteristic starting at $y$, at time $\tau$  we have
\bel{nondeg}{\partial \over\partial y} x(\tau, y)~\geq ~c_0~>~0,\eeq
for all characteristics reaching a  neighborhood of $P$.   Indeed, such a quantity vanishes only at the blow-up times (\ref{r5}).
Based on (\ref{nondeg}), an elementary computation 
leads to (\ref{trans1}).

By (\ref{err1}), for the corresponding  viscous solution $u_\ve$ with $\ve>0$ small enough, we thus have
\bel{trans2}\bega{rl}
\lambda_a~>~f'\bigl( u_\ve(t,x)\bigr) +\kappa_1\ve^{1/7}\qquad &\hbox{if}~~ x\in \bigl[a(t)-2\ve^{1/3}, \, a(t)\bigr],\\[2mm]
\lambda_b~<~f'\bigl( u_\ve(t,x)\bigr) -\kappa_1\ve^{1/7}\qquad &\hbox{if}~~ x\in \bigl[b(t), \, b(t)+2\ve^{1/3}\bigr].\enda\eeq

\v
{\bf 2.} We now construct a second initial data $\bar v$ such that 
\bel{barv1}\bar v(x)~=~\left\{ \bega{cl} \bar u(x)\quad &\hbox{if} \quad x\in \bigl[A', B'\bigr],\\[2mm]
\bar u(A')\quad &\hbox{if} \quad x\leq  A',\\[2mm] \bar u(B')\quad &\hbox{if} \quad x\geq  B'.
\enda\right.
\eeq
Let $v_\ve$ be the solution to (\ref{viscep}) with initial data $\bar v$.  
Consider the rescaled functions
\bel{uvres1}
U^\ve(t,x)\,\doteq\,u_\ve\bigl(\tau+\ve t, \,\ve x\bigr),\qquad\qquad V^\ve(t,x)\,=\, v_\ve\bigl(\tau+\ve t,\, \ve x\bigr),\eeq
and let 
$$\Gamma_\ve(t)\,\doteq\,\ve^{-1} \gamma(\tau+\ve t),\qquad\qquad
\bigl[a_\ve(t), \, b_\ve(t)\bigr] \,\doteq\,\bigl[ \ve^{-1} a(\tau+\ve t), ~\ve^{-1} b(\tau+\ve t)\bigr],$$
be  the shock curve,  and the interval corresponding to $[a, b]$ in the rescaled coordinates.
Referring to Fig.~\ref{f:asy56}, consider the intervals
\bel{step}[s'_\ve, t'_\ve]\,\doteq\,\bigl[ \tau - 2\ve^{1/6},\tau-\ve^{13/14}\bigr],
\qquad [s_\ve, t_\ve]\,\doteq\,\bigl[\ve^{-1} (s'_\ve-\tau), \,\ve^{-1}(t'_\ve-\tau)\bigr]~=~\bigl[ - 2\ve^{-5/6},
-\ve^{-1/14}\bigr].
\eeq
By Lemma~\ref{l:61} applied to the solution $v$ of (\ref{1}) with initial data $\bar v$, and by piecewise Lipschitz continuity, at time $t=s_\ve= -2 \ve^{-5/6}$ we have
\bel{Vep1}\bega{rl} 
\bigl|V^\ve(t,x) - u^-\bigr|~=~\O(1)\cdot \ve^{1/7}\qquad \hbox{for}~~ x< \Gamma_\ve(t) - \ve^{-2/3},\\[2mm]
\bigl|V^\ve(t,x) - u^+\bigr|~=~\O(1)\cdot \ve^{1/7}\qquad \hbox{for} ~~x> \Gamma_\ve(t) + \ve^{-2/3}.\enda\eeq
We can thus apply Theorem~\ref{t:41}, where:
\begi
\item  The time interval $[0,T]$ is replaced by 
$  [s_\ve, t_\ve]$ as in (\ref{step}).  Notice that  $t_\ve - s_\ve>\ve^{-5/6}$.
\item The interval 
$[a,b]$ is replaced by $[a_\ve, b_\ve],\doteq\,\bigl[ \Gamma_\ve(s_\ve) - \ve^{-2/3}, ~ \Gamma_\ve(s_\ve) + \ve^{-2/3}\bigr]$.
\item At time $t=s_\ve$, for $x\notin [a_\ve,b_\ve]$, we have 
\bel{Vepm}\bigl|V^\ve(s_\ve,x) - u^\pm\bigr| \leq \delta_0 ~=~ \O(1)\cdot \ve^{1/7}.\eeq
 \endi
At times  $t>\!>  s_\ve$,  by Theorem~\ref{t:41} the solution $V^\ve(t,\cdot) $ is well approximated by a traveling wave profile $\phi$ with left and right asymptotic states $u^-, u^+$. 
Indeed, for a suitable shift $c_\ve$,
\bel{neartv}\sup_x
\bigl|V^\ve(t,x)-\phi(x-\lambda t- c_\ve)\bigr|\,=\,\O(1)\cdot  \left( 1+ b_\ve-a_\ve \over t_\ve-s_\ve\right)^{1/2}\!=~
\O(1)\cdot \left({\ve^{-2/3} \over \ve^{-5/6}}\right)^{1/2}\!=\,\O(1)\cdot \ve^{1/12}.\eeq
We observe that, by Remark~\ref{r:41}, the shift $c_\ve$ can be bounded by
\bel{ceb}
c_\ve~=~\O(1)\cdot \bigl( (b_\ve-a_\ve) + (t_\ve-s_\ve)\bigr)~=~\O(1) (\ve^{-2/3} + \ve^{-5/6})~=~\O(1)\cdot \ve^{-5/6}.\eeq
Hence, in the original coordinates, this corresponds to choosing
\bel{sh2}\xi_\ve~=~\O(1)\cdot (\ve^{1/3} + \ve^{1/6})~=~\O(1)\cdot\ve^{1/6},\eeq
which approaches zero as $\ve\to 0$.
\v

In particular, at time $t=t_\ve= -\ve^{-1/14}$, integrating over the intervals where $|x- c_\ve|<\ve^{-1/13}$, as $\ve\to 0$ 
we have the convergence
\bel{int0}\int_{|x|<\ve^{-1/13}} \Big|V^\ve(t_\ve,x)-\phi(x-\lambda t_\ve - c_\ve)\Big|\, dx~=~\O(1)\cdot \ve^{{1\over 12}- {1\over 13}}~\to ~0.\eeq
\v
{\bf 3.} Thanks to the transversality condition (\ref{trans1}), we can now apply 
 Lemma~\ref{l:62} with $\sigma= {1\over 7}$ and $r= {1\over 6}$, $q=1$.
 First with $\lambda=\lambda_a$, then with $\lambda=\lambda_b$.
 This yields the bound
 \bel{VUed}\bigl|V^\ve(t,x)-U^\ve(t,x)\bigr|~=~\O(1)\cdot \ve ~<~\ve^{1/2}\qquad\qquad\forall x\in \bigl[a_\ve(t), b_\ve(t)\bigr].\eeq
In view of (\ref{int0}), we obtain
 \bel{int10}\int_{|x|<\ve^{-1/13}} \Big|U^\ve(t_\ve,x)-\phi(x-\lambda t_\ve - c_\ve)\Big|\, dx~=~\O(1)\cdot \ve^{{1\over 156}}~\to ~0.\eeq
\v
{\bf 4.} Finally, we study the solution $U^\ve$ for $t>t_\ve$.
Introduce the upper bound on all characteristic speeds:
\bel{hatl} \hat\lambda~\doteq~1+ \sup\Big\{ \bigl|f'(\omega)\bigr|\,;~~|\omega|\leq \|\bar u\|_{\L^\infty}\Big\}, 
\eeq
and, motivated by (\ref{int0}), define the domain (see Fig.~\ref{f:asy56}, right)
\bel{Dep1}
\D_\ve~\doteq~\Big\{ (t,x)\,;~~|t|\leq  \ve^{-1/14},~~~|x|\leq {1\over 3} \ve^{-1/13}- (t+\ve^{-1/14})\hat\lambda\Big\}.
\eeq
Let $\Hat U^\ve$ be a second solution to (\ref{visc1}) with initial data  at $t=t_\ve = -\ve^{-1/14}$ 
obtained by modifying $U^\ve$,
as follows.
\bel{HEep}
\Hat U^\ve(t_\ve,x)~=~\left\{ \bega{cl} U^\ve(t_\ve,x)\quad &\hbox{if} \quad |x|  \leq {2\over 3}\ve^{-1/13},\\[3mm]
\phi(x-\lambda t_\ve-c_\ve)\quad &\hbox{if} \quad |x|  \geq \ve^{-1/13}.
\enda\right.
\eeq
Recalling (\ref{int10}), we can define $\Hat U^\ve(t_\ve,x) $ in the intermediate region where ${2\over 3}\ve^{-1/13}<|x|< 
\ve^{1/13}$ in such a way that 
\bel{int12} \int_{-\infty}^{+\infty} \Big|\Hat U^\ve(t_\ve,x)-\phi(x-\lambda t_\ve-c_\ve) \Big|\, dx~=~\O(1)\cdot \ve^{{1\over 156}}.\eeq
Since the semigroup generated by the viscous conservation law (\ref{visc1}) is contractive w.r.t.~the $\L^1$ distance, for every $t>t_\ve$ we have
\bel{int3} \int_{-\infty}^{+\infty}  \Big|\Hat U^\ve(t,x)-\phi(x-\lambda t - c_\ve)\Big|\, dx~\le ~ \int_{-\infty}^{+\infty}  \Big|\Hat U^\ve(t_\ve,x)-\phi(x-\lambda t_\ve - c_\ve)\Big|\, dx~=~\O(1)\cdot \ve^{{1\over 156}}.\eeq
Next, to estimate the difference $\Hat U^\ve-U^\ve$ on the domain $\D_\ve$, we apply Lemma~\ref{l:62}
with $\kappa=1$, $\lambda= \pm \hat \lambda$, $r= 1/14$, $q=1$.   By (\ref{uvec}) it follows
\bel{UHUerr}\Big|\Hat U^\ve(t,x)- U^\ve(t,x)\bigr|~=~\O(1)\cdot\ve\qquad\qquad\forall (t,x)\in \D_\ve\,.\eeq
Together with (\ref{int3}), for every $t\in [-\ve^{-1/14}, \ve^{-1/14}]$, this implies
\bel{int34} \int_{|x|\leq {1\over 3} \ve^{-1/13}- (t+\ve^{-1/14})\hat\lambda}
 \Big|U^\ve(t,x)-\phi(x-\lambda t - c_\ve)\Big|\, dx~=~\O(1)\cdot \ve^{{1\over 156}}~\to~0\eeq
as $\ve\to 0$.

Since the viscous solutions $U^\ve$ are uniformly Lipschitz continuous and the domains $\D_\ve$ invade the entire plane $\R^2$, arguing as in the proof of Lemma~\ref{l:61}, 
by (\ref{int34}) we conclude the convergence of the rescaled functions $U^\ve$ to the traveling wave profile $\phi(\cdot)$, uniformly on bounded sets.

\subsection{A point where two shocks merge.}\label{s:72} We now consider the case where the inviscid solution 
$U$ has two shocks, say located at 
$\gamma_1(t)<\gamma_2(t)$,
interacting at time $t=\tau$ at the point $x=\xi$.    W.l.o.g.~we can again assume $\xi=0$.
Let $u^->u^*>u^+$ be the left, middle, and right state at the point of interaction.
As long as the two discontinuities remain sufficiently far apart, the same analysis as in the case of a single shock
remains valid.  More precisely, let $\phi_1$ be a viscous traveling profile with speed $\lambda_1$ connecting the asymptotic states $u^-, u^*$, and let $\phi_2$ be a traveling profile with speed $\lambda_2$ connecting  the states $u^*, u^+$.  
\v
{\bf 1.}
We perform the same analysis as in Section~\ref{s:71}, separately for the two shocks $\gamma_1$, $\gamma_2$.
More precisely, consider the initial points of the three backward characteristics through $P=(\tau,0)$:
\bel{ABC} A=-\tau f'(u^-),\qquad B= -\tau f'(u^*),\qquad C= - \tau f'(u^+).\eeq
Then define
\bel{ABC'} \bega{l}A'=-\tau f'(u^-)-\ve^{1/7},\qquad B'= -\tau f'(u^*)+\ve^{1/7},\\[2mm]B''= -\tau f'(u^*)-\ve^{1/7},
\qquad 
C''= - \tau f'(u^+)+ \ve^{1/7}.\enda\eeq

As in (\ref{barv1}),
by replacing the initial data $\bar u$ with  functions $\bar v_1$ or $\bar v_2$  such that
\bel{barv11}\bar v_1(x)~=~\left\{ \bega{cl} \bar u(x)\quad &\hbox{if} \quad x\in \bigl[A, B],\\[2mm]
\bar u(A')\quad &\hbox{if} \quad x\leq  A',\\[2mm] \bar u(B')\quad &\hbox{if} \quad x\geq  B',
\enda\right.\qquad\qquad 
\bar v_2(x)~=~\left\{ \bega{cl} \bar u(x)\quad &\hbox{if} \quad x\in \bigl[B'', C''\bigr],\\[2mm]
\bar u(B'')\quad &\hbox{if} \quad x\leq  B'',\\[2mm] \bar u(C'')\quad &\hbox{if} \quad x\geq C'',
\enda\right.
\eeq
 we obtain 
two viscous rescaled solutions $V^\ve_1$, $V^\ve_2$ that satisfy the same estimates as in (\ref{neartv}), say
\bel{neartw2}\sup_{-\infty<x<\infty}
\Big|V_i^\ve(t,x)-\phi_i\bigl(x-\lambda_i t- c_{i,\ve}(t)\bigr)\Big|~=~\O(1)\cdot \ve^{1/12},\qquad\qquad i=1,2.\eeq
Notice that here the shift $c_{i,\ve}$ may depend on time.
The bound (\ref{neartw2}) is uniformly valid for $t\in [\ve^{-5/6},0]$.

Denote by $\Gamma_{1\ve},\Gamma_{2\ve}$ the locations of the shocks in rescaled coordinates:
$$\Gamma_{i\ve}(t)~=~\ve^{-1} \gamma_i(\tau+\ve t),\qquad\qquad i=1,2.$$

Let $t_\ve^*=\O(1)\cdot \ve^{-2/3}$ be the time when
$$\Gamma_2(t^*_\ve)- \Gamma_1(t^*_\ve)~=~2\ve^{-2/3}.$$
Using Lemma~\ref{l:62}, 
 as in (\ref{VUed}) we obtain
\bel{VU12}\sup_{|x-\Gamma_i(t_\ve^*)|\leq \ve^{-2/3}} \bigl| U^\ve(t^*_\ve,x) - V_i^\ve(t^*_\ve,x) \bigr|~<~ \ve^{1/2},
\qquad\qquad i=1,2.\eeq
%
\v
{\bf 2.}
In the present setting, a new difficulty arises.  Namely, as shown in Fig.~\ref{f:asy65}, left,
using the previous arguments based on Lemma~\ref{l:62},  the difference $U^\ve-V^\ve_1$ can only be estimated for 
$x<\Gamma_{2\ve }(t)-\ve^{-2/3}$, and similarly the difference $U^\ve-V^\ve_2$ 
can only be estimated for $x>\Gamma_{1\ve}(t)+\ve^{-2/3}$.

\v

\begin{figure}[ht]
\centerline{\hbox{\includegraphics[width=14cm]{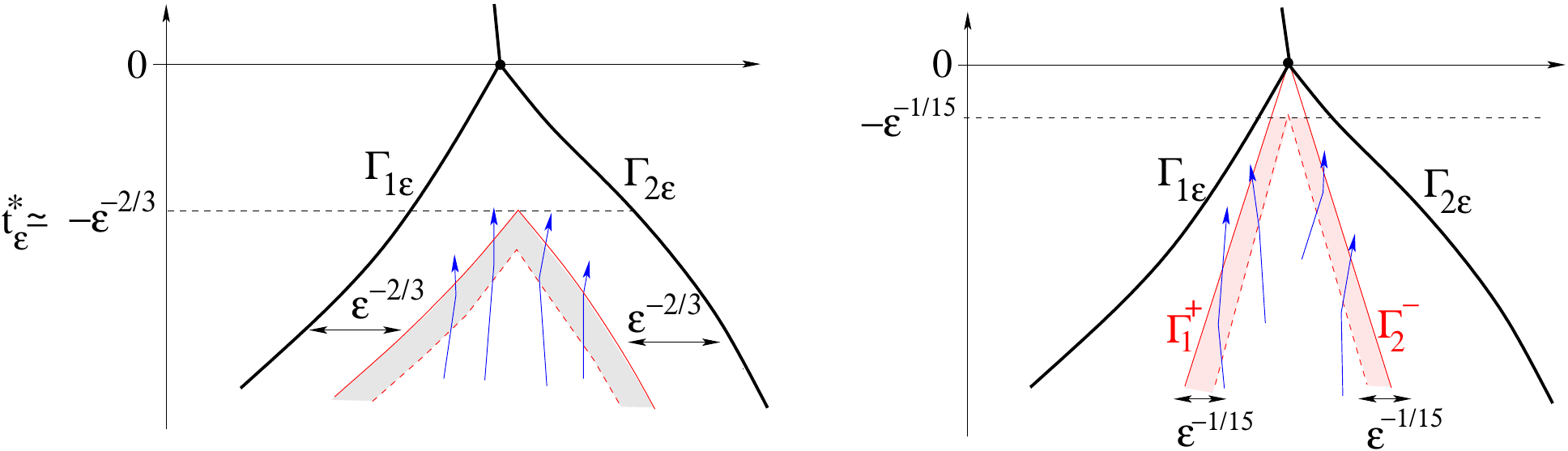}}}
\caption{\small Left: the shock curves $\Gamma_{1\ve},\Gamma_{2\ve}$, in the rescaled coordinates.
Based on the bound (\ref{err1}), transversality of characteristics is ensured only at a distance $>\ve^{-2/3}$.
Right: the argument at (\ref{supvu}) yields the transversality of characteristics up to time $\ve^{-1/15}$.}
\label{f:asy65}
\end{figure}

To extend the bounds (\ref{VU12}) at times $t>t_\ve^*$, we set $\lambda^*=f'(u^*)$ and 
 choose intermediate 
speeds
\bel{l123}\lambda_1~>~\mu_1~>~\lambda^*~>~\mu_2~>~\lambda_2\,,\eeq
and consider the straight lines (see Fig.~\ref{f:asy65}, right)
$$\Gamma_1^+(t) \,=\,  \mu_1t,\qquad\qquad \Gamma_2^-(t) \,=\,  \mu_2 t,\qquad\qquad t<0.$$
Let $t^\sharp$ be the first time where one of the inequalities 
\bel{supvu}\sup_{x\leq \Gamma_2^-(t)}  \Big|V_1^\ve(t,x)-U^\ve(t_\ve,x)\Big|~ \leq ~\ve^{1/2}
\qquad \sup_{x\geq \Gamma_1^+(t)}  \Big|V_2^\ve(t,x)-U^\ve(t_\ve,x)\Big|~ \leq ~ \ve^{1/2},
\eeq is satisfied as an equality.   As long as $t^\sharp<-\ve^{-1/15}$,
for $t\leq t^\sharp$ by (\ref{neartw2}) and (\ref{supvu}) we obtain
$$\sup_{\Gamma_1^+(t)<x<\Gamma_2^-(t)}\Big( \bigl| U^\ve_i(t,x) - u^*\bigr|+ \bigl| V^\ve_i(t,x) - u^*\bigr|\Big)~=~\O(1)\cdot \ve^{1/12}. $$
Hence (see Fig.~\ref{f:asy65}, right), always for $t< t^\sharp <-\ve^{-1/15}$
the characteristic speeds satisfy 
$$\bigl| f'(U^\ve_i(t,x)) - f'(u^*) \bigr|~=~\O(1)\cdot \ve^{1/12},\qquad\qquad \bigl| f'(V^\ve_i(t,x)) - f'(u^*) \bigr|~=~\O(1)\cdot \ve^{1/12}.$$
By (\ref{l123}) these characteristics
are strictly transversal to both $\Gamma_1^+$ and $\Gamma_2^-$.  For every $\ve>0$ small enough, we now apply Lemma~\ref{l:62} (i)
with $r=1/15$, $q=p$, choosing first $\lambda=\mu_1$, then $\lambda=\mu_2$
(the arbitrary constants in (\ref{l123})), and $\kappa = (\lambda_i-\mu_i)/2$.
This yields
$$\sup_{x\leq \gamma_2^-(t)}  \Big|V_1^\ve(t,x)-U^\ve(t_\ve,x)\Big|~ =~\O(1)\cdot \ve^p~<~\ve^{1/2},
$$
reaching a contradiction.
\v
{\bf 3.}  Now consider the rescaled solution $U^\ve$ at time $t_\ve =-\ve^{-1/14}$.  By the previous analysis, 
we have
\bel{supUep}\bega{l}\ds
\sup_{|x|<\ve^{-1/13}, \, x<\lambda^* t_\ve}  \Big|U^\ve(t_\ve,x)- \phi_1(x-\lambda_1 t_\ve -c_{1,\ve})\Big|~=~\O(1)
\cdot \ve^{1/12},\\[3mm]
\ds
\sup_{|x|<\ve^{-1/13}, \, x>\lambda^* t_\ve}  \Big|U^\ve(t_\ve,x)- \phi_2(x-\lambda_2 t_\ve -c_{2,\ve})\Big|~=~\O(1)
\cdot \ve^{1/12}.\enda\eeq
Calling $W=W(t,x)$ the eternal solution constructed at (\ref{45}), and setting 
$$W^\ve(t,x)~\doteq~W(t-\tau_\ve', x-\xi'_\ve)$$ for a suitable shift $(\tau'_\ve, \xi'_\ve)$ depending on $c_{1,\ve},c_{2,\ve}$, as $\ve\to 0$ 
we obtain
\bel{int123}\bega{l}\ds \int_{|x|<\ve^{-1/13}} \Big|U^\ve(t_\ve,x)-W^\ve(t,x)\Big|\, dx\
\\[4mm]
 \ds~=~\int_{|x|<\ve^{-1/13}, \, x<\lambda^* t_\ve}  \Big|U^\ve(t_\ve,x)- \phi_1(x-\lambda_1 t_\ve -c_{1,\ve})\Big| \, dx\\[4mm]
\qquad \ds +
\int_{|x|<\ve^{-1/13}, \, x>\lambda^* t_\ve} \Big|U^\ve(t_\ve,x)- \phi_2(x-\lambda_2 t_\ve -c_{2,\ve})\Big| \, dx~=~\O(1)\cdot \ve^{{1\over 13} - {1\over 12}}~\to ~0.
\enda\eeq
Note that, by Remark~\ref{r:41}, as in (\ref{ceb}) we have
\bel{ce12}
c_{1,\ve}, c_{2,\ve}~=~\O(1)\cdot (\ve^{-2/3} + \ve^{-5/6}).\eeq
In turn, this yields
$ \tau_\ve ', ~\xi_\ve ' ~=~\O(1)\cdot (\ve^{-2/3} + \ve^{-5/6})$.
Hence, in the original coordinates, the shifts approach zero:
$$\tau_\ve-\tau, ~ \xi_\ve ~=~\O(1)\cdot  \ve^{1/6}~\to~0\qquad\hbox{as} ~\ve\to 0.$$
\v
{\bf 4.} The last step of the proof is  the same as for the case of a single shock. 
We define the upper bound $\hat \lambda$ for all characteristic speeds as in (\ref{hatl}), and consider the domain 
$\D_\ve$ as in (\ref{Dep1}).  We then  modify $U^\ve$ by setting

\bel{Hep}
\Hat U^\ve(t_\ve,x)~=~\left\{ \bega{cl} U^\ve(t_\ve,x)\quad &\hbox{if} \quad |x|  \leq {2\over 3}\ve^{-1/13},\\[3mm]
W^\ve(t,x)\quad &\hbox{if} \quad |x|  \geq \ve^{-1/13}.
\enda\right.
\eeq
Recalling (\ref{int123}), we can define $\Hat U^\ve(t_\ve,x) $ in the intermediate region where ${2\over 3}\ve^{-1/13}<|x|< 
\ve^{1/13}$ in such a way that 
\bel{int2} \int_{-\infty}^{+\infty} \Big|\Hat U^\ve(t_\ve,x)-W^\ve(t,x) \Big|\, dx~=~\O(1)\cdot \ve^{{1\over 156}}.\eeq
Since the semigroup generated by the viscous conservation law (\ref{visc1}) is contractive w.r.t.~the $\L^1$ distance, for every $t>t_\ve$ we have
\bel{int13} \int_{-\infty}^{+\infty}  \Big|\Hat U^\ve(t,x)-W^\ve(t,x)\Big|\, dx~\le ~ \int_{-\infty}^{+\infty} 
\Big|\Hat U^\ve(t_\ve,x)-W^\ve(t,x)\Big|, dx~=~\O(1)\cdot \ve^{{1\over 156}}.\eeq
The difference $\Hat U^\ve-U^\ve$ on the domain $\D_\ve$ can be estimated as in (\ref{UHUerr}),
with $\kappa=1$, $\lambda= \pm \hat \lambda$, $r= 1/14$, $q=1$.   By (\ref{uvec}) it follows
\bel{UHUer}\Big|\Hat U^\ve(t,x)- U^\ve(t,x)\bigr|~=~\O(1)\cdot\ve\qquad\qquad\forall (t,x)\in \D_\ve\,.\eeq
Together with (\ref{int13}), for every $t\in [-\ve^{-1/14}, \ve^{-1/14}]$, this implies
\bel{int14} \int_{|x|\leq {1\over 3} \ve^{-1/13}- (t+\ve^{-1/14})\hat\lambda}
 \Big|U^\ve(t,x)-W^\ve(t,x)\Big|\, dx~=~\O(1)\cdot \ve^{{1\over 156}}~\to~0\eeq
as $\ve\to 0$.

Since the viscous solutions $U^\ve$ are uniformly Lipschitz continuous and the domains $\D_\ve$ invade the entire plane $\R^2$, arguing as in the proof of Lemma~\ref{l:61}, by (\ref{int14}) we conclude the convergence of the rescaled functions $U^\ve$ to a suitably shifted 
eternal solution $W$, uniformly on bounded sets.

\subsection{The case of variable initial data.}
Finally, we consider the case of a family of initial data $\bar u_\ve$ converging to $\bar u$ as $\ve\to 0$, as in 
(\ref{uelim1}).  For every $\ve>0$ small enough, at time $\tau$  the corresponding inviscid solution $u^{(\ve)}(\tau,\cdot)$ will have a shock at some  point $x_\ve(\tau)$, say with left and right  states $(u^-_\ve, u^+_\ve)$.
Here $x_\ve(\tau)\to \xi$, while $(u^-_\ve, u^+_\ve)\to u(\tau, \xi-), u(\tau, \xi+)$ as $\ve\to 0$.
Retracing all steps in the previous proof, we obtain the existence of domains $\Gamma_\ve\subset\R^2$,
traveling wave profiles $\phi_\ve\to \phi$ and constants $\delta_\ve\to 0$ such that the rescaled viscous
solutions satisfy
\bel{supep}
\sup_{(t,x)\in \Gamma_\ve} \bigl| U^\ve(t,x) - \phi_\ve(x-\lambda_\ve t- c_\ve)\bigr|~\leq~\delta_\ve\,.\eeq
We now observe that the domains $\Gamma_\ve$ and the constants $\delta_\ve$ remain valid for 
all initial data $\bar v$ sufficiently close to $\bar u$. More precisely, for all initial data $\bar v$ such that
\bel{uvclose}
\bigl\| \bar v - \bar u\bigr\|_{\L^1(\R)}~\leq ~\rho,\qquad\quad
\bigl\| \bar v - \bar u\bigr\|_{\C^0(\Omega_0)}~\leq \rho\eeq
for some $\rho>0$ sufficiently small.
Since $\|\phi_\ve- \phi\|_{\L^\infty} \to 0$  and $\lambda_\ve\to \lambda$, this yields the convergence 
$U^{\ve}(t,x)\to \phi(x-\lambda t)$ uniformly on bounded sets of $\R^2$.

An entirely similar argument applies to the case (ii), describing the interaction of two viscous shocks.

\section{A global viscous solution describing shock formation}
\label{sec:8}
\setcounter{equation}{0}

\begin{figure}[ht]
\centerline{\hbox{\includegraphics[width=9cm]{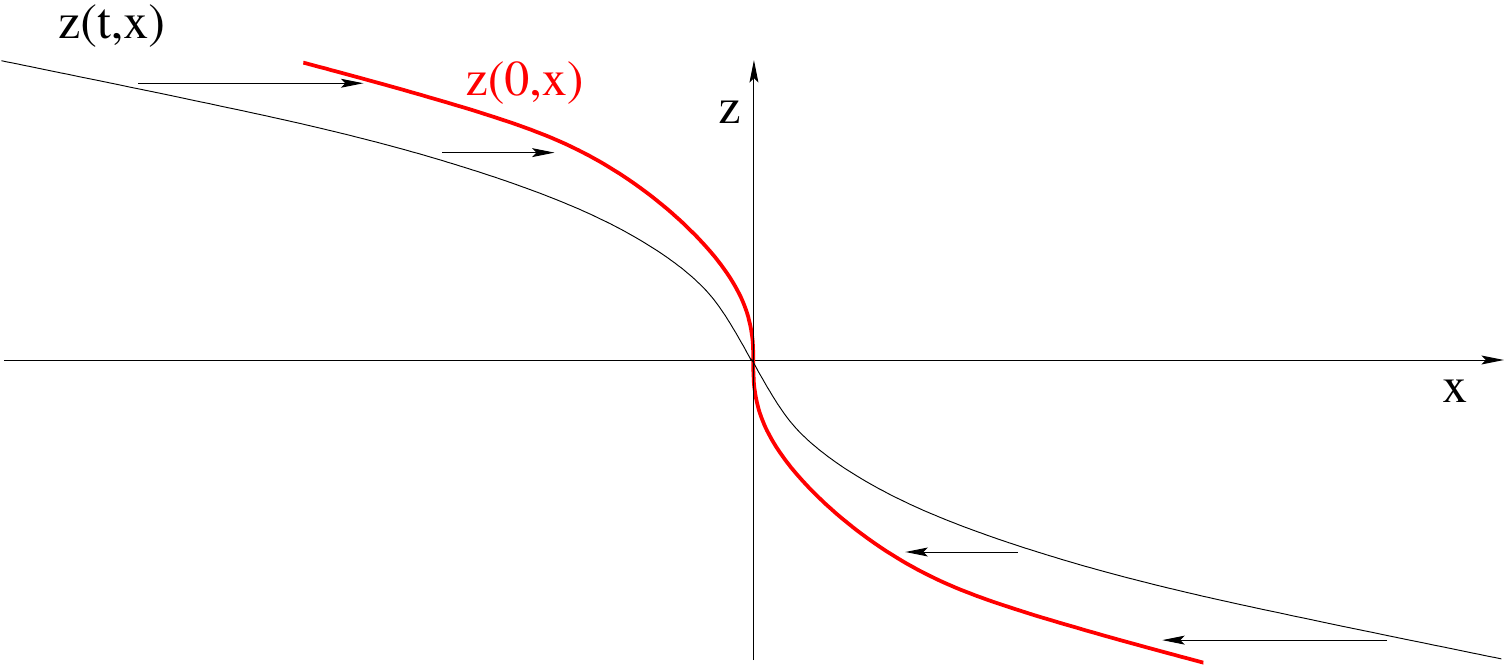}}}
\caption{\small  The solution to Burgers' equation implicitly defined by (\ref{xzc}), for $t\leq 0$.}
\label{f:asy7}
\end{figure}

Let $z=z(t,x)$ be the backward solution to Burgers' equation
\bel{zbur} z_t + zz_x~=~0,\qquad\qquad t\leq 0,\eeq
with terminal data
\bel{ztd} z(0,x)~=~-x^{1/3}.\eeq
For $t<0$ this solution is smooth (see Fig.~\ref{f:asy7}). Indeed, by the method of characteristics, 
the value $z=z(t,x)$ is implicitly defined by
\bel{xzc}
x~=~tz -z^3\qquad\qquad t\leq 0, ~x, z\in \R.\eeq
For  $t\leq 0, ~x\in \R$, the unique real root of~\eqref{xzc} has the same sign as $-x$. It is explicitly computed by
a classical formula due to Scipione del Ferro ($\sim$1540)
\bel{xzcsol} 
 z(t,x)~=~\sqrt[3]{-\frac x2+\sqrt{\frac{x^{2}}{4}-\frac{t^{3}}{27}}}+\sqrt[3]{-\frac x2-\sqrt{\frac{x^{2}}{4}-\frac{t^{3}}{27}}} 
  \,.\eeq
\begin{remark}{\rm For $t<0$ the two terms $-tz$ and $z^3$ have the same sign. Therefore
$$\max\bigl\{ |tz|, |z|^3\bigr\}~\leq~|x|~ \leq~2\max\bigl\{ |tz|, |z|^3\bigr\}$$
 This yields
\bel{zbo}
\min\left\{ \Big|{x\over 2t}\Big|\,,\, \Big|{x\over 2}\Big|^{1/3}\right\}~\leq~
\bigl|z(t,x)\bigr|~\leq~\min\left\{ \Big|{x\over t}\Big|\,,\, |x|^{1/3}\right\}.\eeq 
}
\end{remark}
For future use, we derive the implicit relations
\bel{xzz} x_z~=~t-3z^2, \qquad x_{zz}\,=\, -6z,\eeq
\bel{zxx} 
 z_x~=~{1\over x_z}~=~{1\over t-3z^2}\,,\qquad
 z_{xx}~=~{6zz_x\over (t-3z^2)^2}~=~{6z\over (t-3z^2)^3}\,,
 \eeq
 \bel{zxxx} \bega{rl}
 z_{xxx}&=~\ds {6z_{x}\over (t-3z^2)^3}+{108z^{2}z_{x}\over (t-3z^2)^4}\\
 &\ds=~{6 \over (t-3z^2)^4}-{108  \over (t-3z^2)^4}\cdot \frac{z^{2}}{3z^{2}-t}
 ~\in~\left[-{36 \over (t-3z^2)^4}\,,~{6 \over (t-3z^2)^4}\right]\,.
 \enda\eeq
We used in the last line that $\frac{z^{2}}{3z^{2}-t}=\frac{1}{3-t/z^{2}}\leq \frac13$ is bounded since we consider 
times $t<0$.
Notice that $z, z_{xx}$ are odd functions, while
$z_x, z_{xxx}$ are even functions.

 \begin{lemma}\label{l:81}
Let $t\leq 0, ~x\in \R$ and let $z(t,x)$ be the real root of~\eqref{xzc}.
\begi
\item[(i)] 
The first derivative $z_{x}=1/x_z$ satisfies $t^{-1}\leq z_x<0$.
\item[(ii)]
The second derivative satisfies $- |t|^{-5/2}<z_{xx}\leq 0$ for $x\leq0$, and $0\leq z_{xx}< |t|^{-5/2}$ for $x\geq 0$.
\item[(iii)]
The third derivative satisfies $\abs{z_{xxx}}\leq 36\cdot t^{-4}$.
\endi
\end{lemma}
\v
{\bf Proof.} 
From (\ref{zxx}) it immediately follows
\bel{zxbound}{1\over t} \leq z_x < 0\qquad\qquad\forall t<0, \, x\in\R.\eeq
Note that $z_x(t,x)\to 0$ as  $z\to\pm\infty$, since we are considering the case $t\leq 0$.

Differentiating the second identity in (\ref{zxx}) we find
$${d\over dz} {6z\over (t-3z^2)^3}~=~{6\over (t-3z^2)^3} +6z\cdot {(-3)\cdot (-6z)\over (t-3z^2)^4}
~=~{6(t+15 z^2)\over   (t-3z^2)^4}\,.  $$
This vanishes when 
$t+15 z^2=0$.   Setting $z^2 = -t/15$ in the last expression of (\ref{zxx}) we obtain
\bel{zxxb}|z_{xx}|~\leq~{6\sqrt{-t/15}\over \bigl| t-3(-t/15)\bigr|^3}~=~{1\over\sqrt 3\cdot 36} \left(-t\over 5\right)^{-5/2}
~<~|t|^{-5/2}.\eeq

Finally, the bound on $|z_{xxx}|$ follows immediately from (\ref{zxxx}).  \endproof

\v
Next, we seek a solution to the viscous Burgers' equation
\bel{Z0}Z_t + ZZ_x~=~Z_{xx}\,,\eeq
which, as $t\to -\infty$, asymptotically coincides with the backward solution $z$
of (\ref{zbur})-(\ref{ztd}). 

Toward this goal, we construct the sequence of functions $Z^{(n)}:[-n,+\infty [\,\times\R\mapsto \R$, 
where each $Z^{(n)}$ provides the solution to (\ref{Z0})
with initial data at $t=-n$ given by
\bel{Zn}Z(-n,x)~=~z(-n,x).\eeq

\begin{lemma}\label{l:82}
As $n\to\infty$, the sequence $(Z^{(n)})_{n\geq 1}$ converges to a function $Z :\,\R\times \R\mapsto\R$
uniformly on bounded sets.   The sequence is monotone increasing for $x\geq 0$ and
monotone decreasing for $x\leq 0$.
\end{lemma}

{\bf Proof.} {\bf 1.} We observe that all functions $Z^{(n)}, Z^{(n)}_{xx}$ are odd, while  $Z^{(n)}_x, Z^{(n)}_{xxx}$
are even.

We claim that, for all $t\in [-n,0]$, 
\bel{Zxx}\bega{rl}Z^{(n)}_{xx}(t,x)  \geq 0\qquad\hbox{for}~~x\geq 0,\\[2mm]
Z^{(n)}_{xx}(t,x)  \leq 0\qquad\hbox{for}~~x\leq 0.\enda\eeq
Indeed, differentiating (\ref{Z0}) one finds
\bel{Ztx}Z_{xt} + ZZ_{xx}~=~- (Z_x)^2 +Z_{xxx}\,,\eeq
\bel{Ztxx}Z_{xxt} + ZZ_{xxx} + 3 Z_x Z_{xx}~=~Z_{xxxx}\,,\eeq
Observing that
$$Z^{(n)}_{x}(-n,x)~=~z_x(-n,x)~<~0,$$
by (\ref{Ztx}) a comparison argument yields 
$$Z^{(n)}_{x}(t,x)~<~0\qquad\qquad\forall t\geq -n, ~x\in \R.$$
Moreover, since 
$$Z^{(n)}_{xx}(-n,x)~=~z_{xx}(-n,x)~\geq ~0\qquad\forall x\geq 0,$$
$$Z^{(n)}(t,0)~=~Z^{(n)}_{xx}(t,0)~=~0\qquad\forall t\geq -n,$$
by  (\ref{Ztxx}) a comparison argument yields 
\bel{Z+} Z^{(n)}_{xx}(t,x)~\geq ~0\qquad\qquad\forall t\geq -n, ~x\geq 0.\eeq
\v
{\bf 2.} Next, assume $1\leq m<n$.   We claim that 
\bel{Zmn} Z^{(n)}(t,x)~\geq~Z^{(m)}(t,x)\qquad\forall  t\geq -m, ~x\geq 0.\eeq
Indeed, by (\ref{Z+}) it follows 
$$Z^{(n)} + Z^{(n)}Z^{(n)}_x~\geq~0\qquad \forall t\geq -n,~x\ge 0\,.$$
Hence $Z^{(n)}$ is a supersolution to (\ref{zbur}) on $[-n, -m]\times \R_+$,
taking the same values as $z$ at $t=-n$ and at $x=0$.  A comparison argument yields
$$Z^{(n)}(t,x)\geq z(t,x)\qquad \forall (t,x)\in [-n, -m]\times \R_+\,.$$
On the common domain where $ (t,x)\in [-m,+\infty [\,\times \R_+$, the functions $Z^{(m)}, Z^{(n)}$ satisfy the same 
equation (\ref{Z0}), with boundary values
$$Z^{(m)}(-m, x)~=~z(-m,x)~\leq~  Z^{(n)}(-m, x), \qquad\qquad Z^{(m)}(t,0)\,=\,Z^{(n)}\,=\,0.$$
By comparison, this implies
$$Z^{(m)}(t, x)~\leq~  Z^{(n)}(t, x)\qquad\qquad\forall (t,x)\in  [-m,+\infty[\,\times \R_+\,.$$
Therefore, the sequence  $Z^{(n)}(t, x)$ is monotone increasing, for all $x\geq 0$.
Since all these functions are odd, the sequence  $Z^{(n)}(t, x)$ is monotone decreasing, for all $x\leq 0$.
This proves convergence. \endproof

\begin{remark}\label{r:82}{\rm  By (\ref{Ztx}),  another comparison argument 
shows that the first derivatives $Z^{(n)}_x$ satisfy
$${1\over t}~\leq~Z^{(n)}_x(t,x)~\leq~0\qquad\qquad\forall t<0, ~x\in \R.$$

Next, for any $t<0$, define
$$\sigma(t)~\doteq~\int_{-\infty}^t {2\over 3}  |s|^{-3/2}\, ds~=~{4\over 3} |t|^{-1/2}.$$
By another comparison argument, in view of (\ref{zxxb})  we claim that
\bel{Zbb} \left\{\bega{rl} 
z(t,x) ~\leq~Z(t,x)~\leq z\bigl(t,x-\sigma(t)\bigr) +  {2\over 3} |t|^{-3/2}\qquad\qquad  &\hbox{for}~~x\geq 0,\\[2mm]
z\bigl(t,x+\sigma(t)\bigr) -  {2\over 3} |t|^{-3/2}~\leq~Z(t,x) ~\leq~z(t,x)\qquad\qquad &\hbox{for}~~x\leq 0.\enda\right.\eeq
Indeed, to prove the inequality for $x\geq 0$, set $\Tilde Z(t,x)= z\bigl(t,x-\sigma(t)\bigr) +  {2\over 3} |t|^{-3/2}$. Then
$$\bega{rl} \Tilde Z_t&=~z_t -\dot \sigma(t) z_x +  |t|^{-5/2} ~=\ds~- zz_x - {2\over 3} |t|^{-3/2} z_x +  |t|^{-5/2} \\[2mm]
&=~ - \Big(\Tilde Z -   {2\over 3} |t|^{-3/2}\Big) \Tilde Z_x- {2\over 3} |t|^{-3/2} \Tilde Z_x+ |t|^{-5/2} ~\geq~-\Tilde Z\Tilde Z_x  + \Tilde Z_{xx}\,,
\enda
$$
showing that $\Tilde Z$ is an upper solution to (\ref{Z0}), with $\Tilde Z(t,0)\geq z(t,0)=0$.  The case $x\leq 0$ is entirely similar.

Since $t^{-1}\leq z_x\leq 0$, we have 
$$\bigl| z(t,x) - z(t, x\pm\sigma(t)\bigr| ~\leq~ t^{-1} \sigma(t)~\leq~2t^{-3/2}.$$
Hence the above implies
\bel{Zbo} \left\{\bega{rl} 
z(t,x) ~\leq~Z(t,x)~\leq z(t,x) + {2} |t|^{-3/2}\qquad\qquad  &\hbox{for}~~x\geq 0,\\[2mm]
z(t,x) - {2} |t|^{-3/2}~\leq~Z(t,x) ~\leq~z(t,x)\qquad\qquad &\hbox{for}~~x\leq 0.\enda\right.\eeq
}
\end{remark}

\section{Proof of Theorem~\ref{t:2} - preliminaries}
\label{sec:9}
\setcounter{equation}{0}
Toward a proof of Theorem~\ref{t:2}, we first consider the case where all solutions $u_\ve$ have the same initial data:
\bel{uepzero}
u_\ve(0,x)~=~\bar u(x)\qquad\qquad x\in \R.\eeq
Assume that, in the inviscid solution, the shock forms
at the point $P=(\tau,\xi)$. 
To simplify the notation, by performing an affine transformation
of the $(t,x,u)$ coordinates as in Remark~\ref{r:11}, we assume
\bel{ufa}\xi=0,\qquad u(\tau,\xi)=0,\qquad\qquad f(0)= f'(0)=0,\qquad f''(0)=1.\eeq
Moreover, at the point of blow up, we can also assume that the inverse function $x=x(\tau,u)$ satisfies
\bel{ifu} x_{uuu}(\tau,0)\,=\,-6,\qquad\qquad x(\tau,u)\,= - u^3 +\O(1) \cdot u^4.\eeq
Note that the three parameters $c,\sigma,\lambda$ in (\ref{PZdef}) must be inserted 
precisely to keep track of this affine transformation of variables.

The following notation will be used.
\begi
\item The rescaling of the inviscid solution $u$ is denoted by
\bel{U(e)} U^{(\ve)}(t,x)~=~\ve^{-1/4}\Big[ u\bigl(\tau+  \ve^{1/2}t, ~\ve^{3/4}x\bigr)\Big].\eeq
\item The rescaling of the viscous solution $u_\ve$ is denoted by
\bel{Ue} U^\ve(t,x)~=~\ve^{-1/4}\Big[ u_\ve\bigl( \tau+ \ve^{1/2}t, ~\ve^{3/4}x\bigr)\Big].\eeq
\item $z(t,x)$ is the backward solution of Burgers' equation considered at (\ref{zbur})-(\ref{ztd}) . It is invariant under rescaling.
\item $Z(t,x)$ is the solution of  the viscous Burgers' equation (\ref{Z0}) constructed in Lemma~\ref{l:82}.
\endi
To help the reader, we first outline the main steps of the proof, with the aid of Fig.~\ref{f:asy72}.

\begin{figure}[ht]
\centerline{\hbox{\includegraphics[width=5cm]{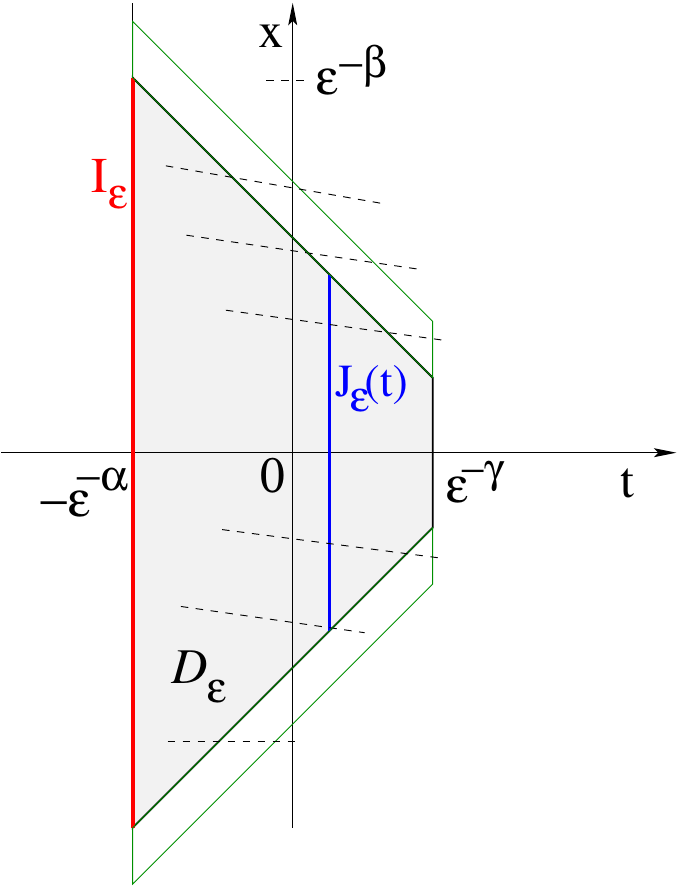}}
}
\caption{\small The construction used in the proof of Theorem~\ref{t:2}.
Here the trapezoidal domain has the form $\D_\ve = \bigl\{ (t,x)\,;~t\in [ -\ve^{-\alpha}, \ve^{-\gamma}],~~|x|\leq r_\ve(t)\bigr\}$.}
\label{f:asy72}
\end{figure}

\begi
\item[{\bf (I)}] For any small $\ve>0$, we consider a segment in the $t$-$x$ plane:
\bel{Iepdef}I_\ve = \{-\ve^{-\alpha}\}\times \bigl[-\ve^{-\beta}, \, \ve^{-\beta}\bigr]~\subset~\R^2.\eeq
Call $t_\ve\doteq -\ve^{-\alpha}$. By suitably choosing the exponents $0<\beta<\alpha$, as $\ve\to 0$ we achieve
the $\L^1$ convergence of the inviscid solutions
\bel{conv1}
\int_{-\ve^{-\beta}}^{\ve^{-\beta}} \bigl| U^{(\ve)}(t_\ve,x) - z(t_\ve,x)\bigr|\, dx~\to~0.\eeq
\item[{\bf (II)}]    We then show that, as $\ve\to 0$, the difference between viscous and inviscid solutions
also approaches zero in $\L^1(I_\ve)$. Namely
\bel{conv2}
\int_{-\ve^{-\beta}}^{\ve^{-\beta}} \bigl| Z(t_\ve,x) - z(t_\ve,x)\bigr|\, dx~\to~0.\eeq
\bel{conv3}
\int_{-\ve^{-\beta}}^{\ve^{-\beta}} \bigl| U^\ve(t_\ve,x) - U^{(\ve)} (t_\ve,x)\bigr|\, dx~\to~0.\eeq
In turn, this implies
\bel{conv4}
\int_{-\ve^{-\beta}}^{\ve^{-\beta}} \bigl| U^\ve(t_\ve,x) - Z (t_\ve,x)\bigr|\, dx~\to~0.\eeq
\item[{\bf (III)}]  As shown in Fig.~\ref{f:asy72}, we then consider a domain of dependence $\D_\ve$, so that all characteristics cross the lateral boundary of $\D_\ve$ 
strictly in the outward direction.
Introducing the segments
$$J_\ve(t)~\doteq~\{ (s,x)\in \D_\ve\,;~~s=t\bigr\}~\doteq~\bigl\{ (t,x)\,;~~|x|\leq r_\ve(t)\bigr\}$$
for a suitable affine function $r_\ve(\cdot)$, 
this will imply
\bel{conv5} \int_{-r_\ve(t)}^{r_\ve(t)} \bigl| U^\ve(t,x) - Z (t,x)\bigr|\, dx
~=~\O(1)\cdot \int_{-\ve^{-\beta}}^{\ve^{-\beta}} \bigl| U^\ve(t_\ve,x) - Z (t_\ve,x)\bigr|\, dx
~\to ~0\eeq
as $\ve\to 0$.
\item[{\bf (IV)}]  As $\ve\to 0$, the domains $\D_\ve$ invade the entire plane $\R^2$.
Since the functions $U^\ve, Z$ are solutions to a conservation law with unit viscosity,
they are uniformly Lipschitz continuous on bounded sets of the plane.
Therefore, the $\L^1$ convergence  (\ref{conv5})  implies the pointwise convergence $U^\ve(t,x) \to Z (t,x)$,
uniformly on bounded sets. This will achieve the proof.
\endi

\begin{figure}[ht]
\centerline{\hbox{\includegraphics[width=13cm]{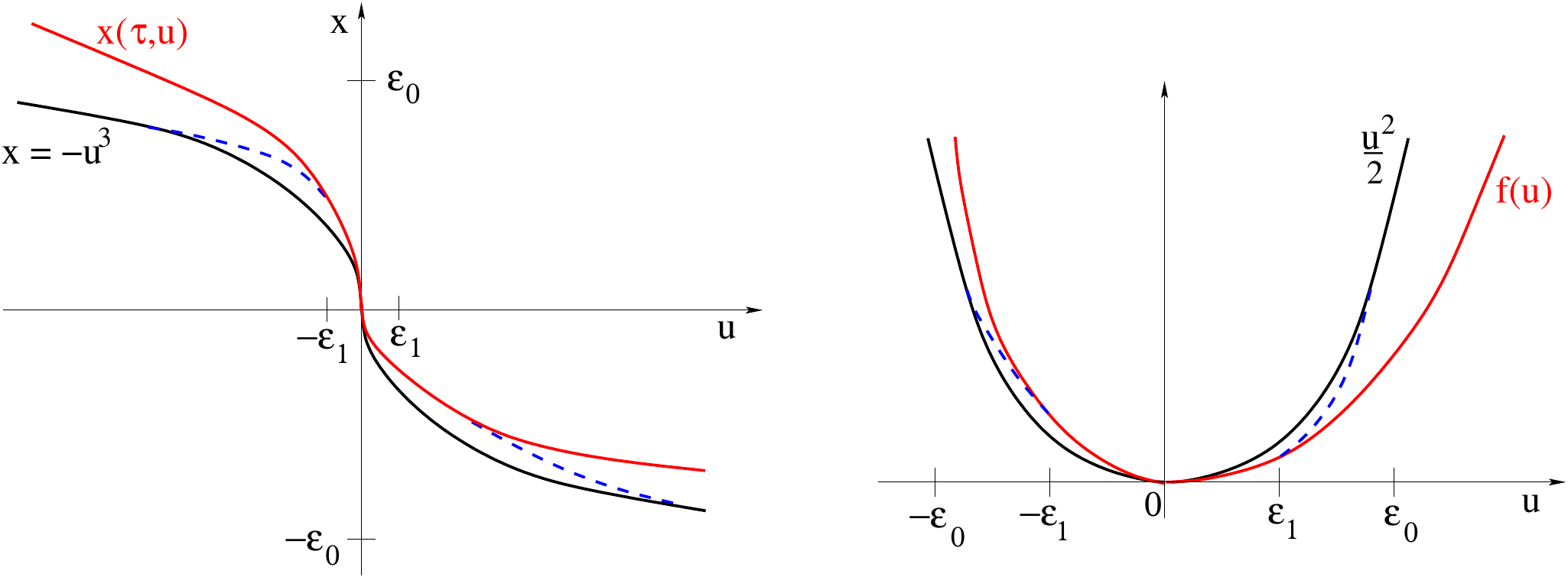}}}
\caption{\small   Left: the function $\Tilde x(\tau,\cdot)$ with the properties {\bf (P1)} is obtained by interpolating the graphs of $x(\tau,\cdot)$ and $Y(u)=-u^3$ (dashed line).   Right:   the flux function $\Tilde f$ with the properties {\bf (P2)} is obtained by interpolating
the graphs of $f$ and $B(u)=u^2/2$ (dashed line).
}\label{f:asy73}
\end{figure}

As a further preliminary, we  observe that the asymptotic properties of 
the rescaled solutions $U^{(\ve)}, U^\ve$ are not affected if the initial data $\bar u$ or the flux function $f$ are 
modified outside a neighborhood of the origin. 
 
More precisely, consider the function $Y(u)\doteq - u^3$ and Burgers' flux function
$B(u) \doteq u^2/2$. Given any small $\epsilon_0>0$, consider a terminal condition
$\Tilde u(\tau,\cdot)$ and a flux function $\Tilde f$ with the following properties:

\begi
\item[{\bf (P1)}] {\it The profile $x\mapsto \Tilde u(\tau,x)$ is monotone decreasing.  Moreover, its inverse function $u\mapsto \Tilde x(\tau,u)$ satisfies\label{P1}
 (see Fig.~\ref{f:asy73}, left)
\bel{ibu2}
 \Tilde x(\tau,u)\,=\,Y(u) \quad \hbox{whenever}~~|u|\geq \epsilon_0\,,
 \qquad \qquad   \bigl\| \Tilde x(\tau,\cdot) -Y(\cdot) \|_{\C^3(\R)}\,\leq\,\epsilon_0\,.\eeq
At $u=0$, the first three derivatives of $\Tilde x(\tau,\cdot)$ and $Y$ coincide:
 \bel{ibu3}
\Tilde x(\tau,0)\,=\,\Tilde x_u(\tau,0)\,=\, \Tilde x_{uu}(\tau,0)\,=\,0, \qquad \Tilde x_{uuu}(\tau,0)\, = \, -6.\eeq
 Moreover, for some $\epsilon_1>0$ there holds
 \bel{utu} \Tilde u(\tau,x)~=~u(\tau,x)\qquad\forall |x|\leq \epsilon_1\,.\eeq

 }
 \item[{\bf (P2)}] {\it One has  (see Fig.~\ref{f:asy73}, right)
\bel{fB} 
\Tilde f(u)~=~B(u)\quad \hbox{whenever} ~~|u|\geq\epsilon_0\,,
\qquad\qquad  \|\Tilde f-B\|_{\C^2(\R)}\,\leq\,\epsilon_0\,,\eeq
\bel{f00}
\Tilde f(0)\,=\,\Tilde f'(0)\,=\,0,\qquad \Tilde f''(0)\,=\,1.\eeq
Moreover, for some $\epsilon_1>0$ there holds
\bel{ftf} \Tilde f(u) = f(u)\qquad \hbox{whenever} \quad |u|\leq \epsilon_1\,.\eeq
}
\endi

\begin{figure}[ht]
\centerline{\hbox{\includegraphics[width=9cm]{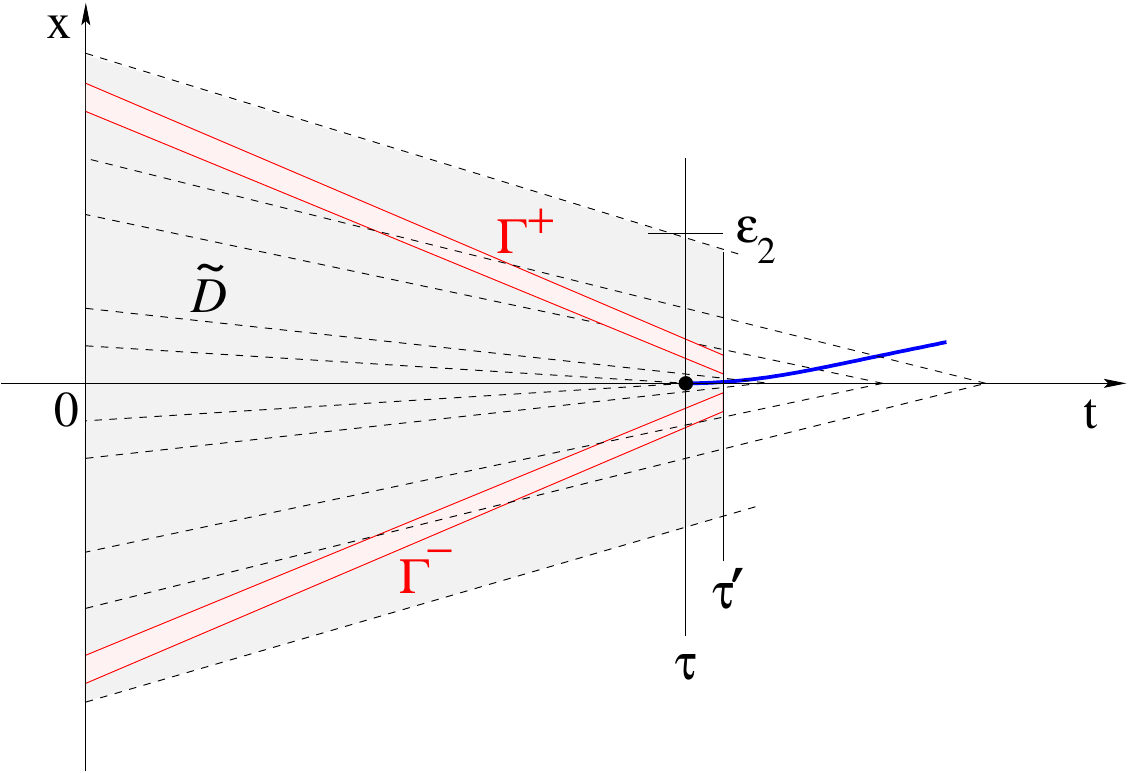}}}
\caption{\small  The construction showing that the modifications of the data 
in Fig.~\ref{f:asy73} do not change the asymptotic limit of rescaled solutions in a neighborhood of the point $(\tau,0)$ where a new shock is formed.}
\label{f:asy74}
\end{figure}

Let $u=u(t,x)$ and $\Tilde u=\Tilde u(t,x)$ be the solutions to
\bel{utilde2}u_t + f(u)_x=0,\qquad\qquad \Tilde u_t + \Tilde f(\Tilde u)_x =0,\eeq
respectively, with data $u(\tau,\cdot)$ and $\Tilde u(\tau,\cdot)$ as in {\bf (P1)}, 
assigned at time $t=\tau$.
Notice that, because of (\ref{utu}) and (\ref{ftf}), these two solutions coincide on a domain of the form
\bel{TDdef}\Tilde \D~\doteq~\Big\{ (t,x)\,;~~t\in [0, \tau']\,,~~-\epsilon_2 +(t-\tau) 
f'\bigl(u(\tau,-\epsilon_2)\bigr) \,\leq \,x \,\leq \,\epsilon_2 +(t-\tau) f'\bigl(u(\tau,\epsilon_2)\bigr)\Big\},\eeq
for some $\tau'>\tau$ and some $\epsilon_2\in \,]0, \epsilon_1[\,$.

As shown in Fig.~\ref{f:asy74}, we now construct two disjoint domains
\bel{Gpmor}\bega{l}
\Gamma^+ ~=~\Big\{ (t,x)\,;~~t\in [0, \tau'], ~~x\in \bigl[ \delta- \lambda(t-\tau'), ~4\delta -  \lambda(t-\tau')\bigr] \Big\},\\[2mm]
\Gamma^- ~=~\Big\{ (t,x)\,;~~t\in [0, \tau'], ~~x\in \bigl[ -4\delta+ \lambda(t-\tau'), ~-\delta +  \lambda(t-\tau')\bigr] \Big\},\enda
\eeq
choosing $\delta>0$ and $\lambda>0$ so that 
\bel{Gint} \Gamma^+\cup\Gamma^-~\subset~\Tilde \D,\eeq
and moreover, for some $\delta_0>0$, strict transversality holds:
\bel{cross}\left\{ \bega{cl} f'\bigl(u(t,x)\bigr) > - \lambda +\delta_0 
\qquad&\forall (t,x)\in \Gamma^+,\\[2mm]
 f'\bigl(u(t,x)\bigr) <  \lambda -\delta_0 
\qquad&\forall (t,x)\in \Gamma^-.\enda\right.\eeq
We also consider the subdomain
\bel{D*def}\D^*~\doteq~\Big\{ (t,x)\,;~~t\in [0, \tau']\,,~~
x\in \bigl[ -\delta+ \lambda(t-\tau'), ~\delta -  \lambda(t-\tau')\bigr]\Big\}.\eeq
By (\ref{Gint}) it follows $\D^*\subset\Tilde \D$.   Hence the two inviscid 
solutions $u, \Tilde u$ trivially coincide on $\D^*$.   Thanks to the transversality
property (\ref{cross}), we will show that the corresponding viscous solutions $u_\ve, \Tilde u_\ve$ are
also extremely close on $\D^*$.

More precisely, call $u_\ve, \Tilde u_\ve$ respectively the solutions to 
\bel{utuep}
\bega{rl}
(u_\ve)_t + f(u_\ve)_x~=~\ve (u_\ve)_{xx}\,,\qquad\qquad &u_\ve(0,x) = u(0,x),
\\[3mm]
(\Tilde u_\ve)_t + \Tilde f(\Tilde u_\ve)_x~=~\ve (\Tilde u_\ve)_{xx}\,,\qquad\qquad &\Tilde u_\ve(0,x) = \Tilde u(0,x).\enda\eeq
Recalling (\ref{Ue}), consider the corresponding rescaled solutions:
\bel{RUTU} 
\bega{l}U^\ve(t,x)~\doteq~\ve^{-1/4}\Big[ u_\ve\bigl( \tau+ \ve^{1/2}t, ~\ve^{3/4}x\bigr)\Big],
\\[3mm]
\Tilde U^\ve(t,x)~\doteq~\ve^{-1/4}\Big[ \Tilde u_\ve\bigl( \tau+ \ve^{1/2}t, ~\ve^{3/4}x\bigr)\Big].
\enda
\eeq
Observe that, in the rescaled coordinates, the domain $D^*$ is transformed into
\bel{D*ep}\D^*_\ve~\doteq~\Big\{ (t,x)\,;~~(\tau+\ve^{1/2} t,~\ve^{3/4} x) \in \D^*\Big\} .\eeq
The next lemma shows that the difference between 
the rescaled viscous solutions $U^\ve$ and $\Tilde U^\ve$ becomes arbitrarily small as 
$\ve\to 0$.  A very similar result can be found in Lemma 1 of the Appendix in \cite{SX}.

\begin{lemma}\label{l:91} In the above setting, as $\ve\to 0$ there holds
\bel{visconv}\lim_{\ve\to 0} ~\sup_{(t,x)\in \D^*_\ve} \Big| U^\ve(t,x) - \Tilde U^\ve(t,x)
\Big|~=~0.\eeq
\end{lemma}

{\bf Proof.} {\bf 1.} We first construct a homotopy between the inviscid solutions $u$ and $\Tilde u$, then a homotopy between the rescaled viscous solutions $U^\ve$ and $\Tilde U^\ve$.
Set 
\bel{ftheta} 
f^\theta(u)~\doteq~\ \theta \Tilde f(u) + (1-\theta) f(u),\eeq
and let $u^\theta$ be the solution to 
\bel{fthe}
u^\theta_t + f^\theta(u^\theta)_x
~=~0\qquad\qquad  u^\theta (0,x) ~=~\theta \Tilde u(0,x) 
+(1-\theta) u (0,x).  \eeq
Note that, for every $\theta\in [0,1]$, we have
\bel{utud*}
u^\theta(t,x)~=~u(t,x)~=~\Tilde u(t,x) \qquad\forall \quad (t,x)\in \Tilde \D.\eeq
In particular, equality holds on $\Gamma^+\cup\Gamma^-$. 
\v
{\bf 2.} Next, consider the corresponding viscous solutions $u_{\ve,\theta}$, which solve
\bel{fthep}
(u_{\ve,\theta})_t + f^\theta(u_{\ve,\theta})_x
~=~\ve(u_{\ve,\theta})_{xx}\qquad\qquad  u_{\ve,\theta} (0,x) ~=~u^\theta (0,x).  \eeq
By (\ref{utud*}), $u^\theta(t,x)$ is uniformly Lipschitz continuous on $\Gamma^+\cup\Gamma^-$.
Therefore, as $\ve\to 0$, by Lemma~\ref{l:61} the viscous approximations 
$u_{\ve,\theta}$  converge to $u^\theta$ uniformly on the subdomains
\bel{Gpm}\bega{l}
\Tilde \Gamma^+ ~=~\Big\{ (t,x)\,;~~t\in [0, \tau'], ~~x\in \bigl[ 2\delta- \lambda(t-\tau'), ~3\delta -  \lambda(t-\tau')\bigr] \Big\},\\[2mm]
\Tilde \Gamma^- ~=~\Big\{ (t,x)\,;~~t\in [0, \tau'], ~~x\in \bigl[ -3\delta+ \lambda(t-\tau'), ~-2\delta +  \lambda(t-\tau')\bigr] \Big\}.\enda
\eeq
As a consequence, for all $\ve>0$ small enough, a transversality condition similar to (\ref{cross}) remains valid
on these subdomains:
\bel{cross6}\left\{ \bega{cl} f'\bigl(u_{\ve,\theta}(t,x)\bigr) > - \lambda +\delta_0/2 
\qquad&\forall (t,x)\in \Tilde \Gamma^+,\\[2mm]
 f'\bigl(u_{\ve,\theta}(t,x)\bigr) <  \lambda -\delta_0/2 
\qquad&\forall (t,x)\in \Tilde\Gamma^-.\enda\right.\eeq
\v
{\bf 3.} We now consider the rescaled viscous solutions. Namely, we call 
$\Tilde U^\ve$ the function obtained by rescaling $\Tilde u_{\ve}$ and 
$U^{\ve,\theta}$ the function obtained as a rescaling of $u_{\ve,\theta}$.  Notice that 
$$U^{\ve,\theta}:\bigl[\tau_\ve, \tau'_\ve\bigr]\times\R\mapsto\R,$$
where
$$\tau_\ve\,\doteq-\ve^{-1/2}\tau,\qquad\qquad \tau'_\ve\,\doteq\,\ve^{-1/2}(\tau'-\tau).$$
Calling
\bel{feptheta} 
f_{\ve,\theta}(U)~\doteq~\ve^{-1/4} \Big[ \theta \Tilde f(\ve^{1/4} U) + (1-\theta) f(\ve^{1/4} U)
\bigr],\eeq
for $\theta\in [0,1]$ 
the $U^{\ve,\theta}$ satisfy
\bel{Uepthe}
U^{\ve,\theta}_t + f_{\ve,\theta}(U^{\ve,\theta})_x
~=~U^{\ve,\theta}_{xx}\,,\qquad \qquad U^{\ve,\theta} (\tau_\ve,x) ~=~\theta \Tilde U^\ve (\tau_\ve,x) 
+(1-\theta) U^\ve (\tau_\ve,x). \eeq
An estimate on the difference $\Tilde U^{\ve,\theta} - U^{\ve,\theta} $ will be obtained
by proving uniform bounds on 
$$V^{\ve,\theta}(t,x)~\doteq~{\partial\over\partial \theta} U^{\ve,\theta},\qquad\qquad \theta\in [0,1].$$
\v
{\bf 4.}
In view of (\ref{feptheta})-(\ref{Uepthe}),  $V^{\ve,\theta}$ satisfies the linearized equation
\bel{Vet}
V^{\ve,\theta}_t  +  f'_{\ve,\theta}(U^{\ve,\theta}) V^{\ve,\theta}_x
~=~V^{\ve,\theta}_{xx}-\left({\partial\over\partial\theta} f'_{\ve,\theta}(U^{\ve,\theta}) 
\right)V^{\ve,\theta}\,U^{\ve,\theta}_x\,,\eeq
with initial data at $t=\tau_\ve =-\ve^{1/2}\tau$
\bel{Ve0} V^{\ve,\theta} (\tau_\ve,x) ~=~ \Tilde U^\ve  (\tau_\ve,x)
-U^\ve (\tau_\ve,x) \,.\eeq
In the rescaled variables, the domain $\Tilde D$ is transformed into
\bel{DefDve}\Tilde\D_\ve~\doteq~\Big\{ (t,x)\,;~~(\tau+\ve^{1/2} t\,,\, \ve^{3/4} x)\in \Tilde \D\Big\}.\eeq
Similarly, we call $\Tilde \Gamma_\ve^+, \Tilde \Gamma^-_\ve\subset \Tilde\D_\ve$
the domains obtained from $\Tilde \Gamma^+, \Tilde \Gamma^-$, after the rescaling.
We now observe that
\begi
\item[(i)] At time $t=\tau_\ve$, the initial data $V^{\ve,\theta}$ vanishes on
$$\Tilde \D_\ve\cap\{(t,x)\,;~t=\tau_\ve\}.$$
\item[(ii)]  For all $t\in [\tau_\ve,\tau'_\ve]$, the source term ${\partial\over\partial\theta} f_{\ve,\theta}(U^{\ve,\theta}) $ vanishes on $\Tilde \D_\ve$.
\item[(iii)] By (\ref{cross6}), on $\Tilde \Gamma^+_\ve$ and on $\Tilde \Gamma^-_\ve$ the flux is strictly transversal:
\bel{cross5}\left\{ \bega{cl} (f^{\ve,\theta})'\bigl(U^{\ve,\theta}(t,x)\bigr) > - \lambda +\delta_0/2 
\qquad&\forall (t,x)\in \Tilde \Gamma^+_\ve,\\[2mm]
(f^{\ve,\theta})'\bigl(U^{\ve,\theta}(t,x)\bigr)<  \lambda -\delta_0/2 
\qquad&\forall (t,x)\in \Tilde\Gamma^-_\ve.\enda\right.\eeq
Moreover, after the rescaling the regions $\Tilde \Gamma^\pm_\ve$ have thickness
$\ve^{-3/4}\delta\to +\infty$.
\item[(iv)]  In the original coordinates, the functions $u_{\ve,\theta}$ are not bounded on the entire domain $[0,\tau]\times \R$.    Indeed, they grow as $|x|^{1/3}$ as 
$|x|\to \infty$.    However, all these functions remain uniformly bounded 
on the domain $\Tilde \D$ in (\ref{TDdef}).  
As a consequence, on the corresponding rescaled domains 
\bel{TDep}\Tilde \D_\ve~\doteq~\Big\{ (t,x)\,;~~(\tau+\ve^{1/2} t,~\ve^{3/4} x) \in \Tilde \D\Big\} \eeq
we have $\bigl|U^{\ve,\theta}(t,x)\bigr|\leq C\ve^{-1/4}$ for some constant $C$.
By (\ref{fthe}), this implies that both quantities
\bel{ftb}(f^{\ve,\theta})'(U^{\ve,\theta}),\qquad\qquad {\partial\over\partial\theta}(f^{\ve,\theta})'(U^{\ve,\theta}),\eeq
remain uniformly bounded on $\Tilde\D_\ve$.
\endi
By constructing upper and lower solutions, similarly as in the proof of part (iii) of Lemma~\ref{l:62},
we obtain
\bel{uvecnew}\sup_{(t,x)\in \D^*_\ve}
\bigl| V^{\ve,\theta}(t,x)\bigr|~\leq~C_q\,\ve^q,\eeq
for all $\theta\in [0,1]$ and every $\ve>0$ sufficiently small.
Here $q>0$ can be any exponent, and $C_q$ is a constant independent of $\ve,\theta$.
In particular, we can take $q=1$.  Integrating w.r.t.~$\theta$, we conclude
\bel{UTUes}\sup_{(t,x)\in \D^*_\ve}
\bigl| \Tilde U^\ve(t,x) - U^\ve(t,x)\bigr|~\leq~C_1\,\ve,\eeq
and hence (\ref{visconv}).
\endproof
\v

%
%

\section{Proof of Theorem~\ref{t:2} - the main argument}
\label{sec:10}
\setcounter{equation}{0}

After all these preliminaries, we now work out details, in several steps.  In view of Lemma~\ref{l:91}, w.l.o.g.~we 
can assume that the flux function $f$ and the profile of the inviscid solution at the 
time $\tau $ of shock formation satisfy (\ref{ibu3})--(\ref{f00}).   
\v
{\bf 1.} We begin by establishing {\bf (I)}. 
Consider the rescaled inviscid solutions $U^{(\ve)}$.  For $t\in [-t_\ve, 0]$, 
$t_\ve \doteq -\ve^{-\alpha}$, we denote by 
$U\mapsto X^{(\ve)}(t,U)$ the inverse function, so that 
$$U^{(\ve)} \bigl(t,  X^{(\ve)}(t,U)\bigr)~=~U.$$
By the assumptions (\ref{ufa})-(\ref{ifu}),   at the terminal time $t=0$
this inverse function   has the form
\bel{xfa}X^{(\ve)}(0, U)~=~ - U^3 + \O(1)\cdot \ve^{1/4} U^4.\eeq
Since $X^{(\ve)}_t(t,U) ~=~f'_\ve(U)$, we obtain
\bel{xu}X^{(\ve)}(t,U)~=~-  U^3 + \O(1)\cdot \ve^{1/4} U^4 + t U +\O(1) \cdot \ve^{1/4} t U^2, \eeq
where, in view of the simplifying assumptions {\bf (P1)-(P2)}, the terms
$\O(1)$ are uniformly bounded. Choosing $\epsilon_0, \epsilon_1>0$ small enough,
 as $t\to -\infty$we obtain the decay of the first derivatives
$${1\over t-3z^2}\,\leq\,z_x(t,x)~<\,0~\qquad\qquad {2\over t-3( U^{(\ve)})^2}\,\leq\, U^{(\ve)}_x(t,x)\,<\,0,\qquad\qquad t<0.$$
This yields
\bel{err3}\bega{ll}
\bigl|U^{(\ve)}(t,x) - z(t,x)\bigr| ~=~
\O(1)\cdot \ve^{1/4} \bigl( z^4 - t z^2\bigr)\cdot |z_x|\\[4mm]
\qquad \ds=~\O(1)\cdot \ve^{1/4}\,
{ z^4 - t z^2\over |t-3z^2|}~=~\O(1)\cdot (z^2-t)\,.\enda\eeq
Integrating at time $t_\ve = -\ve^{-\alpha}$, one obtains
\bel{inverr}\int_{|x|\leq \ve^{-\beta}} \Big| U^{(\ve)}(t_\ve , x) - z(t_\ve , x)\Big|\, dx~=~\O(1)\cdot \ve^{-\beta}  \ve^{1/4}\bigl[ \ve^{-2\beta} + \ve^{-\alpha}\bigr]
\eeq
This converges to zero as $\ve\to 0$ provided we choose $\alpha<\beta<3/20$.

\v

%
{\bf 2.} 
Having established the $\L^1$ convergence of rescaled inviscid solutions on the intervals  $I_\ve$, 
we now study the convergence of the corresponding viscous solutions, as claimed in 
{\bf (II)}.
By Remark~\ref{r:82} it follows that, for $t<0$,
\bel{zZL}
\bigl\| z( t,\cdot)-Z(t,\cdot)\bigr\|_{\L^\infty(\R)} ~\leq ~{2\over 3}\, |t|^{-3/2}.\eeq
When $t=t_\ve \doteq -\ve^{-\alpha}$ this  implies
\bel{zzint} \int_{|x|<\ve^{-\beta}} 
\bigl| z(t_\ve,x)-Z(t_\ve,x)\bigr|\, dx ~=~\O(1)\cdot \ve^{-\beta}\ve^{3\alpha/2}.\eeq
\v
%
We claim that a similar estimate  holds for the difference
between the rescaled inviscid solutions $U^{(\ve)}$
and the corresponding viscous solutions $U^\ve$.  
Repeating the computations at (\ref{xzz})-(\ref{zxx}) with slightly perturbed data,
setting $\bar x(u)\doteq x(\tau, u)$,
for the original solution of (\ref{1}) we obtain
\bel{xu2}\left\{\bega{rl} x(t,u)&=~(t-\tau) f'(u) + x(\tau, u),\\[1mm]
 x_u(t,u)&=~ (t-\tau) f''(u) +  x_u(\tau,u),\enda\right.\qquad
\qquad t\in [0,\tau].\eeq
In connection with the rescaling 
\bel{resc3}
u^\ve = \ve^{-1/4} u, \qquad x^\ve = \ve^{-3/4} x,\qquad t^\ve = \ve^{-1/2} (t-\tau),\eeq
we have
$$f'_\ve(u^\ve) = {dx^\ve\over dt^\ve}~=~{\ve^{-3/4}\over \ve^{-1/2}} {dx\over dt}~=~
\ve^{-1/4} f'(\ve^{1/4} u^\ve),$$
$$f''_\ve(u^\ve) ~=~{d\over du^\ve} f'_\ve(u^\ve)~=~ {d\over du^\ve}\Big(\ve^{-1/4} f'(\ve^{1/4} u^\ve)\Big)~=~f''(\ve^{1/4} u^\ve),$$
$$f'''_\ve(u^\ve) ~=~{d\over du^\ve} f''_\ve(u^\ve)~=~\ve^{1/4} f'''(\ve^{1/4} u^\ve).$$
The above relations imply
\bel{fep5}
f'_\ve(v)\,=\,v + \O(1)\cdot \ve^{1/4} v^2,
\qquad f''_\ve(v)\,=\,1 + \O(1)\cdot \ve^{1/4} v,\qquad f'''_\ve(v)\,=\, \O(1)\cdot \ve^{1/4}.
\eeq
Concerning the rescaled data at $t^\ve=0$ we have
$$\bar x^\ve(u^\ve)~=~\ve^{-3/4}\bar x(\ve^{1/4} u^\ve),
\qquad \bar x^\ve_u(u^\ve)~=~\ve^{-1/2} \bar x_u(\ve^{1/4} u^\ve),\qquad 
 \bar x^\ve_{uu}(u^\ve)~=~\ve^{-1/4} \bar x_{uu}(\ve^{1/4} u^\ve)$$
The above relations imply
\bel{xep}\bar x^\ve(v)~=~-v^3 +\O(1) \cdot \ve^{1/4} v^4,
\qquad \bar x^\ve_u(v)~=~-3v^2 +\O(1) \cdot \ve^{1/4} v^3,\eeq
\bel{xxep}
\bar x^\ve_{uu}(v)~=~-6v +\O(1) \cdot \ve^{1/4} v^2,\qquad \bar x^\ve_{uuu}(v)~=\,-6 +\O(1) \cdot \ve^{1/4} v.
\eeq
We now have (writing $t,v$ in place of $t^\ve, u^\ve$ to simplify notation)
$$x^\ve(t,v)~=~t f_\ve'(v) +\bar x^\ve(v)  ~=~ tv-v^3 +\ve^{1/4} \,\bigl[ \O(1)\cdot  tv^2 +\O(1)\cdot  v^4\bigr], $$
$$x^\ve_u(t,v)~=~t f_\ve''(v) +\bar x_u^\ve(v)  ~=~ t-3v^2 +\ve^{1/4} \,\bigl[ \O(1)\cdot  tv +\O(1)\cdot  v^3\bigr]. $$
Moreover, by (\ref{ibu2})--(\ref{f00}), we can assume the implications
\bel{easy}
\ve^{1/4} |v|~\geq~\epsilon_0\qquad \implies\qquad 
f'_\ve(v)~=~v,\qquad \bar x^\ve(v)~=~-v^3.\eeq

{}From (\ref{fep5}) and (\ref{xep}), for $t<0$  it thus follows
\bel{ux5}
\bega{l} \ds 0~>~U^\ve_x(t,v)~\ds=~{1\over x^\ve_u(t,v)}~=~{1\over t f''_\ve(v) + 
\bar x^\ve_u(v)}\\[4mm]
\quad \ds =~{1\over   t-3v^2 +\ve^{1/4} \cdot\bigl[ \O(1)\cdot  tv 
+\O(1)\cdot  v^3\bigr] }~
\ds \geq~{2\over  t-3v^2}\,.\enda\eeq
Indeed, by (\ref{easy}), if $|v|\geq \ve^{-1/4}\epsilon_0$ the last inequality is trivial.
On the other hand, if $|v|< \ve^{-1/4}\epsilon_0$, then 
\bel{1tv}\bega{l} \ve^{1/4} \cdot\Big| \O(1)\cdot  tv +\O(1)\cdot  v^3\Big| ~=~
\ve^{1/4} \cdot\Big| \O(1)\cdot  t\ve^{-1/4}\epsilon_0 -3 v^2\cdot \O(1)\cdot\ve^{-1/4}\epsilon_0  \Big| \\[3mm]
\qquad \ds =~|t-3v^2|\cdot \O(1)\cdot \epsilon_0~<~{1\over 2} |t-3v^2|,\enda
\eeq
provided that $\epsilon_0>0$ was chosen sufficiently small. This yields (\ref{ux5}).

Next, from
\bel{uxx6} 
U^\ve_{xx}(t,v)\,=\,- { t f'''_\ve(v) + \bar x^\ve_{uu}(v)\over\bigl(  tf''_\ve(v) + \bar x^\ve_u(v)\bigr)^3}\eeq
by (\ref{xxep}) and (\ref{ux5}) we obtain
\bel{uxx7}\bega{l} \ds \bigl| U_{xx}^\ve(t,v)\bigr|~\leq~8\cdot {\Big|-6v +\ve^{1/4} \bigl[ \O(1)\cdot t + \O(1)\cdot v^2\bigr]\Big|\over |t-3v^2|^3}\\[4mm]
\ds\qquad \leq ~{\Big| -48 v +\O(1)\cdot \epsilon_0 v\Big|\over  |t-3v^2|^3}
+ {\O(1)\cdot \ve^{1/4} \over  |t-3v^2|^2}\,.\enda
\eeq
For a suitable constant $C_3$, this yields
\bel{uxx8}
\bigl|U_{xx}^\ve(t,v)\bigr| ~\leq~C_3\Big(
|t|^{-5/2}  +  \ve^{1/4} |t|^{-2} \Big).
\eeq
For any $\ve>0$ sufficiently small, we claim that
\bel{Huu} \bigl|U^\ve(t,x)-U^{(\ve)} (t,x)\bigr|~\leq~C_4 \, \left( |t|^{-3/2}+\ve^{1/4}\,{\bigl|\ln |t|\bigr|+|\ln\ve|\over |t|}\right),\eeq
for some constant $C_4$  independent of  $\ve$.
The bound (\ref{Huu}) will be proved by constructing upper and lower solutions,
as in the proof of (\ref{Zbb}).
Observing that by (\ref{ux5})
\bel{Uxb}\bigl|U^\ve_x(t,x)\bigr|~\leq~2 |t|^{-1},\eeq
an upper solution to the equation
\bel{rveq} U^\ve_t + f'_\ve(U^\ve)_x ~=~U^\ve_{xx}\,,\qquad\qquad U^\ve(-\ve^{-1/2},x)~=~U^{(\ve)}(-\ve^{-\alpha},x),\eeq
 is obtained as follows. Recalling    (\ref{fB}), define
\bel{etaep}\sigma_\ve(t)~\doteq~C_f \int_{-\ve^{-1/2}\tau}^t C_3 \left( {2\over 3}|s|^{-3/2}+\ve^{1/4} |s|^{-1}\right) ds,\qquad\qquad C_f~\doteq~\max_{u} f''(u)~\leq~1+\epsilon_0,\eeq
and
set
$$U^+(t,x)~\doteq~U^{(\ve)} \bigl(t,x-\sigma_\ve(t)\bigr) +  C_3 \left( {2\over 3} |t|^{-3/2} + \ve^{1/4} |t|^{-1}\right). $$
To check that $U^+$ is indeed an upper solution, using (\ref{uxx8}) and (\ref{etaep})  and observing  that 
$$U^+_x(t,x) ~=~U^{(\ve)}_x  \bigl(t,x-\sigma_\ve(t)\bigr) ~<~0,$$ we compute
$$\bega{rl}U^+_t(t,x)&=~U^{(\ve)}_t \bigl(t, x-\sigma_\ve(t)\bigr) -\dot\sigma_\ve(t) U^{(\ve)}_x\bigl(t, x-\sigma_\ve(t)\bigr) + C_3 \Big(  |t|^{-5/2} + \ve^{1/4} |t|^{-2}\Big)
\\[2mm]
&\geq~-f'_\ve\left(U^+ -C_3 \Big({2\over 3}  |t|^{-3/2}+\ve^{1/4} |t|^{-1}\Big)\right) U^+_x  - \dot\sigma_\ve(t) U^+_x  + 
 U^+_{xx}\\[2mm]
&\geq ~-f'_\ve(U^+) U^+_x  - C_f C_3 \Big({2\over 3}  |t|^{-3/2}+\ve^{1/4} |t|^{-1}\Big)\  |U^+_x|   + \dot\sigma_\ve(t) |U^+_x | +  U^+_{xx}
\\[2mm]
&= ~-f'_\ve(U^+) U^+_x   +  U^+_{xx}\,.

\enda
$$
A lower solution is constructed in the same way.  
Recalling (\ref{Uxb}) and the definition of $\sigma_\ve$ at (\ref{etaep}), we have
$$\bega{l}\ds U^+(t,x) - U^{(\ve)}(t,x)~\leq~\sigma_\ve(t) \bigl\|U^{(\ve)}_x(t,\cdot)\bigr\|_{\L^\infty}
+ C_3 \left( {2\over 3} |t|^{-3/2} + \ve^{1/4} |t|^{-1}\right)\\[4mm]
\ds\qquad \leq ~C_4\left(  |t|^{-3/2} +\ve^{1/4}{\bigl|\ln|t|\bigr|+|\ln \ve|\over |t|}\right), \enda$$
and a similar estimate holds for the  lower solution.
This yields (\ref{Huu}).
\v
{\bf 3.} 
Putting together the three bounds (\ref{inverr}), (\ref{zzint}) and (\ref{Huu}), 
at time $t_\ve = -\ve^{-\alpha}$ we obtain
\bel{err6}\bega{l}\ds
\int_{|x|<\ve^{-\beta}} \bigl|U^\ve(t_\ve ,x)-Z(t_\ve ,x)\bigr|\, dx
~\leq~\int_{|x|<\ve^{-\beta}} \bigl|U^{(\ve)}(t_\ve ,x)- z( t_\ve ,x)\bigr|\, dx\\[4mm]
\ds \qquad +\int_{|x|<\ve^{-\beta}} \Big(\bigl|U^{(\ve)}(t_\ve ,x)-U^\ve(t_\ve ,x)\bigr| 
+\bigl|z(t_\ve ,x)-Z(t_\ve ,x)\bigr|\Big)\, dx\\[4mm]
~=~\O(1)\cdot \ve^{1/4} \ve^{-5\beta/3} + \O(1)\cdot \ve^{-\beta} \ve^{3\alpha/2} +
\O(1)\cdot \ve^{-\beta} \bigl( \ve^{3\alpha/2}+\ve^{1/4}\ve^\alpha |\ln \ve| \bigr).\enda\eeq
As $\ve\to 0$ all the integrals on the right hand side of (\ref{err6}) approach zero
provided we choose
\bel{choose}
\beta<{3\over 20}\,,\qquad\alpha>{2\beta\over 3}\,,\qquad \alpha>\beta-{1\over 4}\,.\eeq
To guarantee that the domains of dependence of the intervals $I_\ve$ in 
(\ref{Iepdef}) invade the entire $t$-$x$ plane, we shall also need $\alpha<\beta$.
Throughout the following, we thus choose the constants $\alpha,\beta$ so that 
\bel{abchose}
{2\beta\over 3}~<~\alpha~<~\beta~<~{3\over 25}\,.\eeq
 \v
{\bf 4.}  Next, we need to estimate the difference between the viscous solutions $U^\ve$ and $Z$ on the domain of dependency of the interval $I_\ve$ at (\ref{Iepdef}).  For this purpose, we can modify both of these functions at time $t_\ve= -\ve^{-\alpha}$, replacing them with smooth functions
$\Hat U^\ve(t_\ve,\cdot)$, $\Hat Z(t_\ve,\cdot)$, such that 
\bel{HUZ}
\left\{ \bega{rl} \Hat U^\ve(t_\ve,x)=U^\ve(t_\ve,x),\qquad \Hat Z(t_\ve,x)=Z(t_\ve,x)\quad &\hbox{if}\quad |x|\leq -\ve^{-\beta},\\[2mm]
\Hat U^\ve(t_\ve,x) ~=~ \Hat Z(t_\ve,x)~=~Z(t_\ve,2\ve^{-\beta})\qquad &\hbox{if}~~x\geq 2\ve^{-\beta},\\[2mm]
\Hat U^\ve(t_\ve,x) ~=~ \Hat Z(t_\ve,x)~=~Z(t_\ve, -2\ve^{-\beta})\qquad &\hbox{if}~~x\leq -2\ve^{-\beta},\enda\right.\eeq
In view of (\ref{err6}) we can arrange so that
\bel{err7}\int_{-\infty}^\infty \bigl|\Hat U^\ve(t_\ve,x)-\Hat Z(t_\ve,x)\bigr|\, dx~=~
\int_{|x|<2\ve^{-\beta}} \bigl|\Hat U^\ve(t_\ve,x)-\Hat Z(t_\ve,x)\bigr|\, dx~\to~0\eeq
as $\ve\to 0$.
For $t>t_\ve$, we shall denote $\Hat U^\ve(t, \cdot)$ and $\Hat Z(t,\cdot)$ the solutions to the 
viscous conservation laws
$$\Hat U^\ve_t + f'_\ve(\Hat U^\ve) \Hat U^\ve_x~=~\Hat U^\ve_{xx}\,,
\qquad\qquad \Hat Z_t + \Hat Z\Hat Z_x~=~\Hat Z_{xx}\,,$$
with the above initial data assigned at time $t=t_\ve$.

\v
{\bf 5.} To estimate the difference   $\bigl\|\Hat U^\ve(t, \cdot)-\Hat Z(t,\cdot)\bigr\|_{\L^1}$, we denote by $S$
the semigroup generated by the viscous Burgers' equation (\ref{Burv}).   This is well 
defined for all bounded measurable initial data 
and is contractive w.r.t.~the $\L^1$ distance. We use the error formula
\bel{L1e}\bega{l}\ds
\bigl\|\Hat U^\ve(t, \cdot)-\Hat Z(t,\cdot)\bigr\|_{\L^1}~\leq~\bigl\|\Hat U^\ve(t_\ve, \cdot)-\Hat Z(t_\ve,\cdot)\bigr\|_{\L^1}
+\int_{t_\ve}^t \lim_{h\to 0+} {1\over h} \Big\| \Hat U^\ve(\tau+h, \cdot) - S_h \Hat U^\ve(\tau,\cdot) \Big\|_{\L^1}
\, d\tau\\[4mm]
\qquad \doteq E_1+E_2(t)\,. \enda\eeq
For each $\tau\in [t_\ve, t]$, by (\ref{fep5}), recalling that $\Hat U^\ve(\tau,\cdot)$ is monotone decreasing and by (\ref{HUZ})
$$\bigl\| \Hat U(t_\ve,\cdot)\bigr\|_{\L^\infty} ~\leq~\Bigl|Z(t_\ve, 2\ve^{-\beta})\bigr|~=~(2\ve)^{1/3},$$
 the integrand function in (\ref{L1e}) can be  bounded as
\bel{err8}\bega{l}\ds\int \Big| f'_\ve\bigl( \Hat U^\ve(\tau,x)\bigr)- \Hat U^\ve(\tau,x)\Big| \, \bigl| \Hat U^\ve_x(\tau,x)\bigr|\, dx~
\leq~ \Big\| f'_\ve\bigl( \Hat U^\ve(\tau,\cdot)\bigr)- \Hat U^\ve(\tau,\cdot)\Big\|_{\L^\infty}\cdot \int \bigl| \Hat U^\ve_x(\tau,x)\bigr|\, dx \\[4mm]
~\leq~\ds \O(1)\cdot \ve^{1/4} \bigl\|\Hat U^\ve(\tau,\cdot)^2\bigr\|_{\L^\infty} \cdot  \bigl\|\Hat U^\ve(\tau,\cdot)\bigr\|_{\L^\infty}
~=~\O(1)\cdot \ve^{1/4}\cdot \ve^{-\beta}.
\enda
\eeq
The first term on the right hand side of (\ref{L1e}) is bounded by (\ref{err6}). 
In view of (\ref{err8}), for $t>t_\ve$ we have 
\bel{E12} E_2(t)~=~\O(1) \cdot (t-t_\ve) \ve^{1/4}\cdot \ve^{-\beta}~=~\O(1)\cdot  \ve^{-\alpha}\ve^{1/4}\cdot \ve^{-\beta},\eeq
which approaches zero as $\ve\to 0$.
\v
{\bf 6.} 
Finally, consider the trapezoidal domain (see Fig.~\ref{f:asy66})
\bel{Dedef2}\D_\ve~\doteq~\Big\{ (t,x)\,;~t\in [-\ve^{-\alpha}, \ve^{-\gamma}]\,,\qquad |x|\leq q_\ve(t)
\Big\},\eeq
\bel{qepdef}q_\ve(t)~\doteq~\ve^{-\beta} - {\ve^{-\beta}-2\ve^{-\delta}\over  \ve^{-\gamma}+\ve^{-\alpha} } (t+\ve^{-\alpha}),
\qquad\qquad \bar q_\ve \doteq q_\ve(0).\eeq
Here we choose
\bel{abcd}0\,<\,\gamma\,<\,\delta\,<\,\beta-\alpha,\qquad\qquad \gamma\,<\,{\beta-\alpha\over 4}\,.\eeq
We need to check that the differences 
\bel{epdif}\bega{l}\ds
e_1(t)~\doteq~\int_{|x|\leq q_\ve(t)-\ve^{-\delta}} |\Hat U^\ve(t,x) - U^\ve(t,x)\bigr|\, dx,\\[4mm]
\ds e_2(t)~\doteq~\int_{|x|\leq q_\ve(t)-\ve^{-\delta}}|\Hat Z^\ve(t,x) - Z(t,x)\bigr|\, dx\enda\eeq
both approach zero as $\ve\to 0$.  For this purpose, we shall use Lemma~\ref{l:62}, 
showing that characteristics move strictly outward from the domain $\D_\ve$.

For Burgers' equation (\ref{zbur}), the characteristic speed is simply $z$.  Moreover, we have
\bel{Zz}\bigl|\Hat Z(t,x)\bigr|~\leq~\bigl|Z(t,x)\bigr|~\leq~\bigl|z(t,x)\bigr|\qquad\hbox{ for all}~ t,x.\eeq
We compare the  the speed $\pm \lambda$ of the boundary of $\D_\ve$ with
speed of characteristics, in the strips
\bel{Depm}\bega{l}\ds\D_\ve^+~\doteq~\left\{ (t,x)\,;~t\in [-\ve^{-\alpha}, \ve^\gamma]\,,\qquad x\in \Big[q_\ve(t)-{\ve^{-\delta}}, ~q_\ve(t)\Big]
\right\},\\[4mm]
\ds\D_\ve^-~\doteq~\left\{ (t,x)\,;~t\in [-\ve^{-\alpha}, \ve^\gamma]\,,\qquad x\in \Big[-q_\ve(t),~-q_\ve(t)+{\ve^{-\delta}}\Big]
\right\}.\enda
\eeq

On $\D_\ve^+$, since $z(t,x)<Z(t,x) <\Hat Z(t,x)<0$, the speed of characteristics satisfies
\bel{cspeed}\bega{l}\Hat Z(t,x)~\geq~Z(t,x)~\geq~z(t,x)~\geq ~z(0,\bar q_\ve) ~=~-\bar q_\ve ^{1/3}~=~-\left( \ve^{-\beta} - {\ve^{-\beta}-2\ve^{-\delta}\over  \ve^{-\gamma}+\ve^{-\alpha} } \ve^{-\alpha}\right)^{1/3}\\[4mm]
\qquad \approx~- \Big( \ve^{-\beta} - \ve^{-\beta} (1-2\ve^{\beta-\delta}) (1-\ve^{\alpha-\gamma})\Big)^{1/3}
~\approx~-(2\ve^{-\delta}+ \ve^{\alpha -\beta-\gamma})^{1/3}.
\enda
\eeq
On the other hand, the speed of the boundary of $\D_\ve^+$ is  
\bel{bspeed}-\lambda~=~ - {\ve^{-\beta}-\ve^{-\delta}\over  \ve^{-\gamma}+\ve^{-\alpha} } 
~<~-{1\over 2} \ve^{\alpha-\beta}
\eeq
for all $\ve>0$ sufficiently small.
By construction, the thickness of the boundary layer $\D_\ve^+$ is $\ve^{-\delta}$. By the choice of $\delta$ at (\ref{abcd}), in view of (\ref{cspeed})-(\ref{bspeed}), for all
$\ve> 0$ sufficiently small we thus have
\bel{62}\bega{l}
{\hbox{[difference in speed]}\times\hbox{[thickness of the boundary layer]}}\\[4mm]
\qquad >~\left(-2\bigl(2\ve^{-\delta}+ \ve^{\alpha -\beta-\gamma}\bigr)^{1/3}+{1\over 2} \ve^{\alpha-\beta} \right)\ve^{-\delta}~\to~+\infty.\enda
\eeq
The same estimate holds for $\D_\ve^-$.
We can thus use Lemma~\ref{l:62} with $U=Z$, $V=\Hat Z$, $q>1$,  and conclude that 
\bel{lim14} 
\sup\left\{ \bigl|Z(t,x) - \Hat Z(t,x)\bigr| \,;~~t\in [-\ve^{-\alpha}, \ve^{-\gamma}],~~
|x|\leq q_\ve(t) - {\ve^{-\delta}\over 2}\right\}~\leq~\ve\eeq
for all $\ve>0$ sufficiently small.  This proves that the integral $e_2(t)$  in (\ref{epdif}) 
approaches zero as $\ve\to 0$,
for all $t\in [\ve^{-\alpha}, \ve^{-\gamma}]$.

 A similar argument shows that the same holds for the first integral $e_1(t)$.

\begin{figure}[ht]
\centerline{\hbox{\includegraphics[width=9cm]{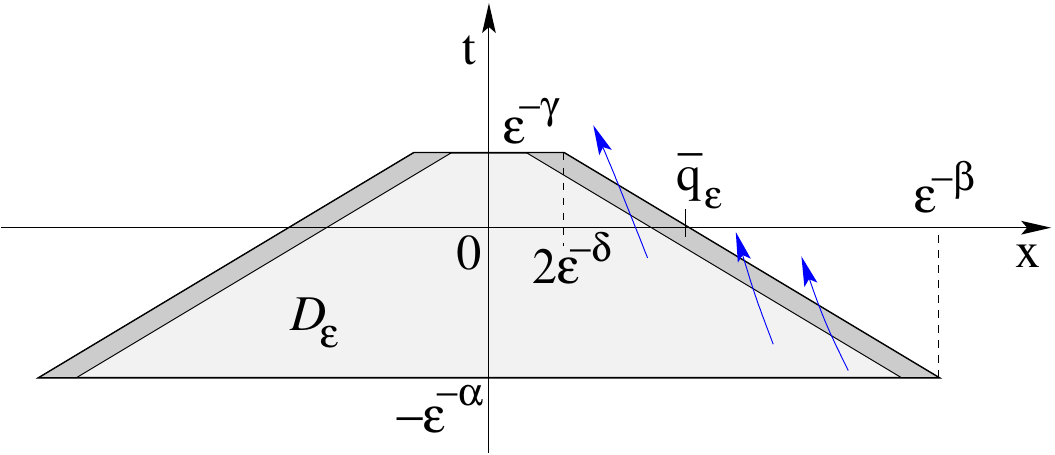}}}
\caption{\small The trapezoidal domain  $\D_\ve$ at (\ref{Dedef2}).  The darker shaded parallelograms
are the regions $\D_\ve^-, \D_\ve^+$ defined at (\ref{Depm}).}
\label{f:asy66}
\end{figure}

\v
{\bf 7.} We now argue as in the proof of Lemma~\ref{l:61}. 
Since the domains $\D_\ve$ invade the entire plane $\R^2$, and the functions $Z, U^\ve$ are uniformly Lipschitz continuous on bounded sets, the integral bounds 
\bel{difin} 
\sup_{t\in[ -\ve^{-\alpha}, \ve^{-\gamma}]}
\int_{|x|\leq q_\ve(t)-\ve^{-\delta}} | U^\ve(t,x)- Z(t,x)\bigr|\, dx ~\to ~0 \qquad\qquad \hbox{as} ~~\ve\to 0\eeq
yield the convergence $U^\ve(t,x)-Z(t,x)\to 0$, uniformly for $t,x$ in bounded sets.

This proves Theorem~\ref{t:2} in the case (\ref{uepzero}) where all solutions $u_\ve$
have the same initial data.
\v
{\bf 8.} To cover the general case, it suffices to observe that, if the high order convergence of the initial data (\ref{uelim2}) holds, 
then in the corresponding inviscid solutions $u^{(\ve)}$
a new singularity is formed at points $(\tau_\ve, \xi_\ve)\to (\tau,\xi)$.   Calling 
$u\mapsto x^{(\ve)} (\tau_\ve, u)$ 
a local inverse to the map $x\mapsto u^{(\ve)}(\tau_\ve,x)$, one has the convergence
relations
\bel{Tayk}u^{(\ve)}(\tau_\ve,\xi_\ve)\,\to\, u(\tau,\xi),\quad\qquad{\partial^k\over \partial u^k} x^{(\ve)} \Big(\tau_\ve, u^{(\ve)} (\tau_\ve, \xi_\ve)\Big) ~\to~{\partial^k\over\partial u^k}
x\bigl(\tau, u(\tau, \xi)\bigr), \quad k=1,2,3.\eeq
In this case, the previous analysis provides an estimate on the difference
between the rescaled functions 
$U^\ve$ in (\ref{resc22})  and the functions
$$\Phi_\ve(t,x)~=~c_\ve\, Z(\sigma_\ve t, \, x- \lambda_\ve t),$$
for suitable $(c_\ve, \sigma_\ve, \lambda_\ve)\to (c,\sigma,\lambda)$.  
This convergence is uniform on a family of compacts $D_\ve$ that invades $\R^2$.
Hence the conclusion.
\endproof
\v
\begin{remark}{\rm
In Theorem~\ref{t:2}  the convergence of rescaled solutions was proved assuming that the initial data
converge in $\C^4$.   This assumption simplifies some of the estimates, but is likely not needed.
Indeed, we expect that the same conclusion can be reached assuming that the convergence takes place locally 
in $\C^3$.

On the other hand, $\C^2$ convergence of the initial data $u_\ve(0,x)\to \bar u(x)$ as $\ve\to 0$ is not enough.
For example, let $u, v$ be the inviscid solutions of Burgers' equation where a shock forms at time $\tau=1$, such that
$$u(1,x)~=~-2 x^{1/3},\qquad\qquad v(1,x) = - x^{1/3}\,.$$
Let $\bar u(x)=u(0,x)$ and  $\bar v(x)=v(0,x)$ be the corresponding initial data.  

One can now construct a sequence of smooth initial data $\bar u_n$ such that 
$$\bigl\|\bar u_n- \bar u \bigr\|_{\C^2}~\to~0,\qquad\qquad \bar u_n(x)=\bar v(x)\quad\hbox{for}~|x|\leq {1\over n}\,.$$
By the previous analysis, for  each fixed $n\geq 1$, the rescaled viscous solutions $U^\ve_n$ converge to the function $Z$ constructed in Lemma~\ref{l:82} uniformly on bounded sets.
Hence, taking a suitable sequence $\ve_n\to 0$, we still have the convergence
$U^{\ve_n}_n(t,x)\to Z(t,x)$ uniformly on bounded sets.
However, the limit of the rescaled solutions $U^\ve$ does not converge to $Z$.
}\end{remark}

\v
{\small
{\bf Acknowledgments.} The research of A.\,Bressan was partially supported by NSF with
grant  DMS-2306926, ``Regularity and approximation of solutions to conservation laws".

L. Caravenna would like to thank the Italian Ministry of University and Research (MUR) for the project PRIN PNRR P2022XJ9SX of the European Union – Next Generation EU "Heterogeneity on the road—modeling, analysis, control", and the GNAMPA group of the "Istituto Nazionale di Alta Matematica F. Severi". She also acknowledges the hospitality of Penn State University, where this work was initiated.}


\begin{thebibliography}{6111}

\bibitem{ACG} J.\,Anderson, S.\,Chaturvedi and C.\,Graham,
Shock formation in 1D conservation laws II: vanishing viscosity.
Preprint 2025. arXiv:2506.17156.



\bibitem{Bi} S.\,Bianchini, Hyperbolic limit of the Jin-Xin relaxation model.
{\it Comm. Pure Appl. Math.} {\bf 59} (2006), 688--753.


\bibitem{BiB02}
S.\,Bianchini and A.\,Bressan, On a Lyapunov functional relating shortening curves and viscous conservation laws, {\it 
Nonlinear Analysis T.M.A.} {\bf 51} (2002),  649--662.

\bibitem{BiB05}
S.\,Bianchini and A.\,Bressan, 
Vanishing viscosity solutions of nonlinear hyperbolic systems,
{\it Annals of Mathematics}  {\bf 161} (2005), 223--342.

\bibitem{Bbook} A.\,Bressan, \textit{Hyperbolic systems of conservation
    laws.  The one-dimensional Cauchy problem}. Oxford University
  Press, Oxford, 2000.

\bibitem{BC17} A.\,Bressan and G.\,Chen,
Generic regularity of conservative solutions to a nonlinear wave equation.
{\it Ann. Inst. H. Poincar\'e}  {\bf 34} (2017),  335--354.

\bibitem{BD} A.\,Bressan and C.\,Donadello,
 On the formation of scalar viscous shocks,
{\it Int. J. Dynam. Syst.  Differ. Equat.}
{\bf 1}  (2007), 1--11.

\bibitem{BD2} A.\,Bressan and C.\,Donadello, On the convergence of viscous approximations
after shock interactions. {\it Discr. Cont. Dynam. Syst.}
{\bf 23} (2009),
29--48.

\bibitem{BGu} A.\,Bressan and G.\,Guerra,
Shift differentiability of the flow generated by a conservation
law, {\it Discrete Cont. Dynamical Systems} {\bf 3} (1997), 35--58.


\bibitem{BHY} A.\,Bressan, T.\,Huang and F.\,Yu,
 Structurally stable singularities 
for a nonlinear   wave equation. {\it Bull. Inst. Math. Acad. Sinica},
{\bf 10} (2015), 449--478.

\bibitem{BMa} A.\,Bressan and A.\,Marson,
 A variational calculus for discontinuous
solutions of conservative systems, {\it Comm. Part. Diff. Equat.} {\bf 20}
(1995), 1491--1552.


\bibitem{BS2} A.\,Bressan and W.\,Shen,
On traffic flow with nonlocal flux: a relaxation  representation, 
{\it Arch. Rational Mech. Anal.} {\bf 237} (2020), 1213--1236.



\bibitem{CG} S.\,Chaturvedi and C.\,Graham,
The inviscid limit of viscous Burgers at nondegenerate shock formation.
{\it Annals of PDE}  (2023) 9:1.

\bibitem{CDS}  M.\,Chiani, D.\,Dardari and M.\,K.\,Simon, New exponential bounds and approximations for the computation of error probability in fading channels. {\it  IEEE Transactions on Wireless Communications}
{\bf 2} (2003)  840--845. 

\bibitem{CJL}
J.\,Choi, C-Y.\,Jung
and 
H.\,Lee,
On boundary layers for the Burgers equations in a bounded domain
{\it Comm. Nonlin. Sci. Num. Simul.}
{\bf 67} (2019), 637--657.

\bibitem{CCS}
M.\,Colombo, G.\,Crippa,  and  L\,V.\,Spinolo,
On the singular local limit for conservation laws with nonlocal fluxes.
{\it Arch.~Rational Mech.~Anal.} {\bf 233} (2019),  1131--1167.


\bibitem{Dafermos} C.\,Dafermos, {\it Hyperbolic Conservation Laws in Continuum Physics}, 4-th ed., Springer-Verlag, Berlin, 2016.


\bibitem{F}
 P.~C.~Fife, {\it
Mathematical Aspects of Reacting
and Diffusing Systems.}  
Lecture Notes in Biomathematics, Springer, 1979.


\bibitem{FS}
H.\,Freist\"uhler and D.\,Serre,  
$L^1$ stability of shock waves in scalar viscous conservation laws.
{\it Comm. Pure Appl. Math.} {\bf 51} (1998),  291--301.

\bibitem{GX} J.\,Goodman and Z.\,Xin, 
Viscous limits for piecewise smooth solutions to systems of conservation laws.
{\it Arch. Rational Mech. Anal.} {\bf 121} (1992), 235--265.

\bibitem{GS} M.\,Golubitsky and D.\,Schaeffer,
 Stability of shock waves for a single conservation law.
{\it Adv. Math.} {\bf 15} (1975),  65-71.

\bibitem{Gk}
J.\,Guckenheimer, Solving a single conservation law. In
``Dynamical systems - Warwick 1974'', pp. 108--134.
Springer Lecture Notes in Math, {\bf 468}.

\bibitem{HR}
H.\,Holden and N.\,Risebro, 
{\it Front Tracking for Hyperbolic Conservation Laws.} Second Edition.  Springer-Verlag, Berlin, 2015.

\bibitem{KP} \newblock A.\,Keimer and L.\,Pflug, 
\newblock{On approximation of local conservation laws by nonlocal conservation laws.}
\newblock\emph{J. Math. Anal. Appl.} {\bf 475} (2019), 1927--1955.


\bibitem{Kr} S.\,Kruzhkov, First order quasilinear equations with several space
variables, {\it Math. USSR Sbornik} {\bf 10} (1970), 217--243.



\bibitem{KT} A.\,Kurganov and E.\,Tadmor,
Stiff systems of hyperbolic conservation laws: convergence and error estimates.
{\it SIAM J. Math. Anal.} {\bf 28}  (1997), 1446--1456.


\bibitem{LV} R.\,J.\,LeVeque,
{\it 
Numerical Methods for Conservation Laws.}
Birkh\"auser, Basel, 1992.

\bibitem{LZ} M.\,Li and Q.\,Zhang, 
Generic regularity of conservative solutions to Camassa-Holm type equations.
{\it SIAM J. Math. Anal.} {\bf 49} (2017), 2920--2949.

\bibitem{Liu}
T.\,P.\,Liu, Hyperbolic conservation laws with relaxation.
{\it Commun. Math. Phys.} {\bf 108} (1987),153--175.

\bibitem{Liubook} T.\,P.\,Liu, {\it Shock Waves.} American Math. Soc., Providence, RI, 2021.



\bibitem{S} D.\,Schaeffer, A regularity theorem for conservation laws,
{\it Adv. in Math.} {\bf  11}
(1973), 368--386.

\bibitem{SX}
 W.\,Shen and Z.\,Xu, Vanishing viscosity approximations to hyperbolic conservation laws,
 {\it  J. Differential  Equations} {\bf 244} (2008), 1692–1711.

\bibitem{VH}
 A.\,Van Harten and R.\,R.\,Van Hassel,  A quasilinear, singular perturbation problem of hyperbolic type, 
 {\it SIAM J. Math. Anal.} {\bf 16} (1985), 1258--1267.

\bibitem{Y} Wen-An Yong,
Singular perturbations of first-order hyperbolic systems with stiff source terms,
{\it J. Differential Equations} {\bf  155} (1999), 89--132.

%

\end{thebibliography}
\end{document}